\documentclass[11pt]{article}

\usepackage{amsmath,enumerate,enumitem}
\usepackage{amssymb}
\usepackage{amsthm}
\usepackage{bbm}
\usepackage{graphicx}
\usepackage{tikz}
\usepackage{xcolor}
\usepackage{hyperref}
\usepackage{mathtools}
\usepackage{url}
\mathtoolsset{showonlyrefs}
\usepackage{enumitem}

\allowdisplaybreaks

\usepackage[margin=3.5cm]{geometry}

\def\bs{\backslash}   
\def\subset{\subseteq}

\def\E{\mathbb{E}}
\def\var{\rm {var}}
\def\lip{\textup{Lip}}

\def\bin{{\rm Bin}}
\def\pois{{\rm Pois}}

\def\ind{\mathbbm{1}}

\def\hom{\textup{Hom}}

\def\E{\mathbb{E}}
\def\R{\mathbb{R}}
\def\P{\mathbb{P}}
\def\Z{\mathbb{Z}}
\def\Q{\mathbb{Q}}
\def\N{\mathbb{N}}

\def\p{\partial}

\def\eps{\varepsilon}
\def\del{\delta}
\def\gam{\gamma}

\def\cA{\mathcal{A}}
\def\cB{\mathcal{B}}
\def\cC{\mathcal{C}}
\def\cD{\mathcal{D}}
\def\cE{\mathcal{E}}
\def\cF{\mathcal {F}}
\def\cG{\mathcal {G}}
\def\cH{\mathcal {H}}

\def\cL{\mathcal {L}}
\def\cM{\mathcal {M}}
\def\cN{\mathcal {N}}
\def\cO{\mathcal{O}}
\def\cP{\mathcal{P}}
\def\cR{\mathcal{R}}
\def\cS{\mathcal{S}}
\def\cT{\mathcal{T}}
\def\cU{\mathcal{U}}
\def\cV{\mathcal{V}}
\def\cW{\mathcal{W}}

\def\1{\mathbf{1}}

\def\lam {\lambda}

\def\gam {\gamma}
\def\Gam {\Gamma}

\tikzstyle{P} = [draw, circle, black, fill, inner sep = 0pt, minimum width = 3pt]
\tikzstyle{every loop} = []
\newcommand{\tikzHind}{
  \begin{tikzpicture}[baseline, yshift=1pt]
    \path[use as bounding box] (-.15,-.1) rectangle (.6,.35);
    \draw (0.5,0) node[P] {} -- (0,0) node[P] {} edge[-,in = 45, out = 135, loop] ();
  \end{tikzpicture}
}

\newcommand{\bydef}{\coloneqq}

\newtheorem*{theorem*}{Theorem}
\newtheorem{theorem}{Theorem}
\numberwithin{theorem}{section}
\newtheorem{lemma}[theorem]{Lemma}
\newtheorem{cor}[theorem]{Corollary}
\newtheorem{defn}[theorem]{Definition}
\newtheorem*{defn*}{Definition}
\newtheorem{prop}[theorem]{Proposition}
\newtheorem*{prop*}{Proposition}
\newtheorem{conj}[theorem]{Conjecture}
\newtheorem*{conj*}{Conjecture}
\newtheorem{claim}[theorem]{Claim}

\newtheorem*{fact*}{Fact}
\newtheorem{fact}{Fact}
\newtheorem{example}{Example}

\numberwithin{equation}{section}

\begin{document}
\title{Homomorphisms from the torus}

\author{Matthew Jenssen\thanks{University of Birmingham, m.jenssen@bham.ac.uk.}\and Peter Keevash\thanks{University of Oxford, keevash@maths.ox.ac.uk.\newline Research supported in part by ERC Consolidator Grant 647678.} }

\date{\today}

\maketitle

\begin{abstract}
We present a detailed probabilistic and structural analysis of the set of weighted homomorphisms from the discrete torus $\mathbb{Z}_m^n$, where $m$ is even, to any fixed graph: we show that the corresponding probability distribution on such homomorphisms is close to a distribution defined constructively as a certain random perturbation of some dominant phase. This has several consequences, including solutions (in a strong form) to conjectures of Engbers and Galvin and a conjecture of Kahn and Park. Special cases include sharp asymptotics for the number of independent sets and the number of proper $q$-colourings of $\mathbb{Z}_m^n$ (so in particular, the discrete hypercube). We give further applications to the study of height functions and (generalised)
rank functions on the discrete hypercube and disprove a conjecture of Kahn and Lawrenz. For the proof we combine methods from statistical physics, entropy and graph containers and exploit isoperimetric and algebraic properties of the torus. 
\end{abstract}

\setcounter{tocdepth}{1}
\tableofcontents

\section{Introduction}
\subsection{Spin Models}
A central notion at the intersection of combinatorics 
and statistical physics is that of a \emph{graph homomorphism}.
Let $G$ and $H$ be graphs (possibly with loops).
We call a function $f: V(G)\to V(H)$ a $\emph{homomorphism}$
(or \emph{$H$-colouring}) if the map preserves edges,
that is, $f(e)\in E(H)$ for all $e\in E(G)$.
We write $\hom (G,H)$ for the set of homomorphisms from $G$ into $H$.
From a combinatorial perspective,
graph homomorphisms provide a unifying
framework for a number of important graph theory concepts. 
For the statistical physicist, homomorphisms arise in the study of
\emph{spin models} and their critical phenomena. 
Tools from each field have enriched the other
and recent years have witnessed an explosion of literature at their interface. 

A \emph{weighted graph} 
is a pair $(H,\lam)$ where $H$ is a graph (possibly with loops)
and $\lam: V(H)\to \R_{>0}$ is a function assigning 
each vertex of $H$ an `activity' $\lam_v$. 
Weighted graphs provide 
a rich set of probability distributions
on the space $\hom(G,H)$.
Indeed, given a weighted graph $(H, \lambda)$, 
we may define a probability measure $\mu_{H,\lam}$
on $\hom (G,H)$ 
given by
\begin{align}\label{eqmuoriginal}
\mu_{H,\lam}(f)\bydef\frac{\prod_{v\in V(G)}\lam_{f(v)}}{Z_G^H(\lam)}\, ,
\end{align}
for $f\in \hom (G,H)$ where
\begin{align}\label{eqZGdef}
Z_G^H(\lam)\bydef\sum_{f\in \hom (G,H) }\prod_{v\in V(G)}\lam_{f(v)}\, .
\end{align}
This type of probability distribution
is called a \emph{spin model} (with hard constraints)
and the normalising factor $Z_G^H(\lam)$ is the 
\emph{partition function} of the model.
One may think of the vertices of $G$ as sites
and the vertices of $H$ as a set of particles (or spins)
that can occupy these sites. 
The edges of $G$ represent bonds between sites and
the edges of $H$ represent constraints on which pairs of particles
can occupy bonded sites. 
A spin model is then a probability distribution 
on the set of legal spin configurations on the sites of $G$.
Two of our main motivating examples will be the following:
\begin{example}
\label{exKq}
\emph{The $q$-colouring model (zero-temperature anti-ferromagnetic Potts model):}
$H=K_q$, $\lam\equiv 1$.
 In this case $\mu_{H,\lam}$
 is the uniform measure over proper $q$-colourings of $G$  and
 $Z_G^H(\lam)$ is the number of such colourings. 
 \end{example}

 \begin{example}
 \label{exind}
 \emph{The hard-core model:} 
 $H=\tikzHind$, that is, an edge $ \{v_{\text{in}}, v_{\text{out}}\}$ with a loop at vertex $v_{\text{out}}$. We let $\lam(v_{\text{out}})=1$ and $\lam(v_{\text{in}})=x$ for some fixed $x>0$ called the \emph{fugacity}. 
 $Z_G^H(\lam)$ is the {hard-core model partition function} 
 (also known as the {independence polynomial}).
In particular when $x=1$, $Z_G^H(\lam)$ is the number of independent sets in $G$. 
\end{example} 

In the statistical physics literature,
spin models are traditionally studied
on the integer lattice $\Z^n$,
a setting
where the phenomenon of 
$\emph{phase coexistence}$ can be 
rigorously studied. 
We mention two landmark results in this field. 
Galvin and Kahn \cite{galvin2004phase} establish
phase coexistence for 
the hard-core model on $\Z^n$ where the fugacity is allowed to tend to $0$ with $n$ (see also~\cite{peled2014odd})
and Peled and Spinka \cite{peled2018rigidity} prove
phase coexistence for
the $q$-colouring model.
Informally these results 
show that
in large subregions of $\Z^n$,
a typical sample from 
 $\mu_{H,\lam}$ 
(with the appropriate choice of $(H,\lam)$)
exhibits long-range order by
correlating with some dominant phase.

In this paper we study spin models on the $n$-dimensional discrete torus $\Z_m^n$
where $m$ is a fixed even integer and $n$ is large. 
In this setting we establish the phase coexistence phenomenon
 in a strong form:\ 
 we show that on $\Z_m^n$, the measure $\mu_{H,\lam}$ can be closely
 approximated by a measure defined constructively as a random perturbation 
 of a random dominant phase. The random perturbation has the 
 distribution of a \emph{polymer model} with convergent \emph{cluster expansion}.
 Via the cluster expansion, we are able to gain an essentially complete
 probabilistic description of the measure $\mu_{H,\lam}$
 and hence a precise structural description of the set $\hom(\Z_m^n, H)$.
 Establishing convergence of the cluster expansion is a non-trivial task
 and requires a combination of entropy tools, 
 the method of graph containers and algebraic and isoperimetric properties of the torus.
Our synthesis of container and entropy methods has its roots in the work of Peled and Spinka~\cite{peled2018rigidity} and later appears in the work of Kahn and Park~\cite{kahn2018number} 
(see Section~\ref{secover} for a more detailed discussion of the use of entropy tools and the container method in this context).

In order to formally state our main result (Theorem~\ref{mainTV} below) we require some 
preliminary notation and language
 from the theory of polymer models which we will now introduce. 
 
 \subsection{The defect polymer model}\label{subsecdefect}
 The $n$-dimensional discrete torus $\Z_m^n$ 
is the graph on vertex set $\{0,\ldots, m-1\}^n$ where two vertices $x, y$
are adjacent if and only if there exists a coordinate $i\in [n]$ such that
$x_i=y_i \pm1 \pmod m$ and $x_j=y_j$ for all $j \neq i$.
The case $m=2$ returns the $n$-dimensional hypercube for which we use the more familiar notation $Q_n$.
We note that when $m$ is even, $\Z_m^n$ is a bipartite graph
and we denote its vertex classes by $\cE, \cO$.
We note also that $\Z_m^n$ is a $d$-regular graph where 
$d=(1+\ind_{m>2})n$.

Given a weighted graph $(H,\lam)$
and $A,B\subset V(H)$, we write $A\sim B$ if
$\{a,b\}\in E(H)$ for all $a\in A$ and $b\in B$.
We call such a pair $(A,B)$ a \emph{pattern}.
Letting $\lam_X\bydef \sum_{v\in X}\lam_v$ for $X\subset V(H)$
we define
\begin{align}\label{eqetadef}
\eta_\lam(H)\bydef \max\{ \lam_A \lam_B: A, B \subset V(H), A \sim B  \}\, .
\end{align}
Following Engbers and Galvin~\cite{engbers2012h}, we call a pattern $(A,B)$ \emph{dominant} if $\lam_A \lam_B= \eta_\lam(H)$, 
and we let $\cD_\lam(H)$ denote the collection of all dominant patterns. 
In Example~\ref{exKq} of the previous section, the dominant patterns are all pairs $(A,B)$
where $V(K_q)=A \cup B$ and $\{|A|, |B|\}=\{\lfloor q/2 \rfloor, \lceil q/2 \rceil\}$.
In Example~\ref{exind}, there are two dominant patterns: $(\{v_{\text{out}}\}, \{v_{\text{in}}, v_{\text{out}}\})$ and $( \{v_{\text{in}}, v_{\text{out}}\}, \{v_{\text{out}}\})$.
We call a homomorphism $f\in\hom(\Z_m^n,H)$ a \emph{dominant colouring}
if $f(\cE)\subset A$, $f(\cO)\subset B$ for some dominant pattern $(A,B)$.

In the following, we let $G$ denote $\Z_m^n$ 
and we fix a weighted graph $(H,\lam)$.
We say that a subset $S\subset V(G)$ is 
\emph{$G^2$-connected} if the graph $G^2[S]$ is connected
(here $G^2$ denotes the square of the graph $G$).
We will define a family $\cP$ of 
$G^2$-connected subsets of $V(G)$
the elements of which we call
\emph{polymers} 
(see Definition~\ref{defpoly} for a full description of $\cP$).
For now we can safely think of $\cP$ as the set of \emph{all}
$G^2$-connected subsets of $V(G)$.

Given a colouring $f\in \hom(G,H)$ and a pattern $(A,B)$,
we say that $f$ \emph{agrees} with $(A,B)$ at $v\in V(G)$ if 
$v\in\cO$ and $f(v)\in A$ or if $v\in\cE$ and $f(v)\in B$.
We say that $f$ disagrees with $(A,B)$ at $v$ otherwise.  
For a subset $S\subset V(G)$, 
let $\chi_{A,B}(S)$ be the set of 
 $f\in\hom(G,H)$ such that $f$
\emph{disagrees} with $(A,B)$ at each $v\in S$ 
and agrees with $(A,B)$ at each $v\in V\bs S$.
We define the \emph{weight} of $S$
(with respect to $(A,B)$) to be
\begin{align}\label{eqconv}
w_{A,B}(S)=
 \frac{\sum_{f\in\chi_{A,B}(S)} \prod_{v\in V}\lam_{f(v)}}{\eta_\lam(H)^{m^n/2}}
\, .
\end{align}

We say two polymers $\gam_1, \gam_2$ are \emph{compatible}
 if the graph distance between $\gam_1, \gam_2$ in $G$ is $>2$; 
 that is if $\gam_1 \cup \gam_2$ is not $G^2$-connected. 
 Otherwise we say that $\gam_1$ and $\gam_2$ are \emph{incompatible}.
 We let $\Omega$ denote the family consisting of 
 all sets of mutually compatible polymers.
The weights $w_{A,B}$ allow us to construct a 
probability distribution $\nu_{A,B}$ on $\Omega$.
 For $\Gamma\in \Omega$, let
\begin{align}\label{eqnudef}
\nu_{A,B}(\Gamma)=\frac{\prod_{\gam\in\Gamma}w_{A,B}(\gam)}{\Xi_{A,B}}\, ,
\end{align}
where the normalising factor
\begin{align}\label{eqXidef}
\Xi_{A,B}=\sum_{\Gamma\in\Omega}\prod_{\gam\in\Gamma}w_{A,B}(\gam)
\end{align}
is the \emph{polymer model partition function.}
We thus have a polymer model for each
dominant pattern $(A,B)$.
We call $\nu_{A,B}$ the \emph{$(A,B)$ polymer model}.
Using these polymer models,
we construct a probability distribution $\hat \mu_{H,\lam}$ 
which will serve as a very good approximation to the measure $\mu_{H,\lam}$.

\begin{defn}\label{defmuhatintro}
 For $f\in \hom(G,H)$, let 
$\hat \mu_{H,\lam}(f)$ 
denote the probability that $f$ is selected by the following
three-step process:

\begin{enumerate}
\item Choose $(A,B)\in \cD_\lam(H)$ with probability proportional to $\Xi_{A,B}$.
\item \label{step20} Choose a random polymer configuration $\Gamma \in \Omega$
from the distribution $\nu_{A,B}$.
\item Letting $S=\bigcup_{\gam\in\Gamma} \gam$, select a colouring $f\in\chi_{A,B}(S)$ with probability proportional to $\prod_{v\in V}\lam_{f(v)}$.
\end{enumerate}

\end{defn}

We may now state our main theorem.
For two probability measures $\mu$ and $\hat\mu$,
we let $\|\hat\mu-\mu\|_{TV}$ denote their {total variation distance}.
\begin{theorem}\label{mainTV}
For $m\ge2$ a fixed even integer and $(H,\lam)$ a fixed weighted graph,
the measures $\mu_{H,\lam}$ and $\hat\mu_{H,\lam}$ on 
$\hom(\Z_m^n,H)$ satisfy
\[
\|\hat\mu_{H,\lam}-\mu_{H,\lam}\|_{TV}\le e^{-\Omega(m^n/n^2)}\, .
\]
\end{theorem}

We can think of Step \ref{step20} in the definition of 
$\hat\mu_{H,\lam}$ as identifying 
the \emph{defect vertices} of the random colouring: 
given a colouring $f\in \hom(\Z_m^n,H)$
we identify the dominant pattern $(A,B)$
which agrees most with $f$ (breaking ties arbitrarily if necessary) and
 call the vertices at which $f$ disagrees with $(A,B)$
the {defect vertices}.
Theorem~\ref{mainTV} shows that up to very small error,
the defect vertices of a sample from $\mu_{H,\lam}$ has the distribution 
of a mixture of the polymer models $\nu_{A,B}$. 
The power of Theorem~\ref{mainTV} stems from the fact that
each of these polymer models admits a 
\emph{convergent cluster expansion} allowing us to obtain
 an essentially complete understanding of the measures $\nu_{A,B}$ and
 therefore of $\mu_{H,\lam}$ also.

In the remainder of this introduction we discuss the various 
consequences of Theorem~\ref{mainTV}. 

\subsection{The structure of $\hom(\Z_m^n, H)$}
Theorem~\ref{mainTV}, combined with a large deviation result 
for the measures $\nu_{A,B}$ (see Theorem~\ref{lempolyLD}),  can be used 
to prove a detailed structural description of the
set $\hom(\Z_m^n, H)$. 

For a weighted graph $(H,\lam)$ and dominant pattern $(A,B)\in \cD_{\lam(H)}$,
we say that an $H$-colouring of $\Z_m^n$ is
\emph{$\lam$-balanced} with respect to $(A,B)$
 if for each
$k\in A$ ,
the proportion of vertices of $\cO$ 
coloured $k$ is within $1/n$
of $\lam_k/\lam_A$ and
for each
$\ell \in B$,
the proportion of vertices of $\cE$ 
coloured $\ell$ is within $1/n$
of $\lam_\ell/\lam_B$.

The following theorem
provides a structural decomposition of the set
$\hom(\Z_m^n, H)$ showing in particular that, with respect to the measure
$\mu_{H,\lam}$, almost all elements of $\hom(\Z_m^n, H)$ are 
$\lam$-balanced with respect to some dominant pattern.

\begin{theorem}
\label{thmgenstruc0}
Fix a weighted graph $(H,\lam)$ and $m\ge2$ even. 
There is a partition of $\hom(\Z_m^n,H)$ into
$|\cD_\lam(H)|+1$ classes
\[
\hom(\Z_m^n,H)= 
F(0)\cup
\bigcup_{(A,B)\in \cD_\lam(H)}F(A,B)
\]
with the following properties. With $\mu=\mu_{H,\lam}$, 
\begin{enumerate}
\item 
$\mu(F(0)) \le e^{-\Omega(m^n/n^3)}$.
\item 
For $(A,B)\in \cD_\lam(H)$, each
$f\in F(A,B)$ is $\lam$-balanced
with respect to $(A,B)$.

\item \label{itemmeas00}
For each $(A,B)\in \cD_\lam(H)$,
\[
\mu(F(A,B))= 
\frac{1}{Z_G^H(\lam)} 
 \eta_\lam(H)^{\frac{m^n}{2}} 
 \Xi_{A,B}\left(1\pm e^{-\Omega(m^n/n^3)}\right)\, .
\]
\end{enumerate}
\end{theorem}

Theorem~\ref{thmgenstruc0} is inspired by a decomposition result of Engbers and Galvin~\cite[Theorem 1.2]{engbers2012h}.
We in fact prove a strengthening of Theorem~\ref{thmgenstruc0}
(Theorem~\ref{thmgenstruc}), 
which directly strengthens \cite[Theorem 1.2]{engbers2012h} and resolves conjectures of the authors in a strong form~\cite[Conjectures 6.1, 6.2, 6.3]{engbers2012h}. We defer a more detailed discussion of these conjectures to Section~\ref{secgenstruc}.

As mentioned in the previous section,
we will be able to gain a precise understanding of the measures $\nu_{A,B}$ 
and therefore also of the partition functions $\Xi_{A,B}$. 
Via conclusion \ref{itemmeas00} of Theorem~\ref{thmgenstruc0},
 we can therefore obtain 
detailed asymptotic expressions for the partition function $Z_G^H(\lam)$.
In particular this yields a plethora of asymptotic formulae for
combinatorial quantities such as the number of proper $q$-colourings
and the number of independent sets in $\Z_m^n$.

\subsection{Asymptotic enumeration and the cluster expansion}\label{subsecasym}
In order to state our approximation results for partition functions 
$Z_G^H(\lam)$
we first introduce the cluster expansion formally. 

For a multiset $\Gamma$ of polymers, 
 the \textit{incompatibility graph}, $I_{\Gamma}$, 
 is the graph with vertex set $\Gamma$ 
 and an edge between any two incompatible polymers.  
 A \textit{cluster} $\Gamma$ is an \emph{ordered} multiset of 
 polymers so that $I_{\Gamma}$ is connected.  
 We let $\cC$ denote the set of all clusters.

For a cluster $\Gamma\in \cC$ and pattern $(A,B)\in \cD_{\lam}(H)$, we let 
\begin{align}\label{eqwdef}
w_{A,B}(\Gamma)\bydef \phi (I_{\Gamma}) \prod_{\gam \in \Gamma} w_{A,B}(\gam) \,,
\end{align}
where $\phi(F)$ is the \textit{Ursell function} of a graph $F$, defined by
\begin{align}
\label{eqUrsell}
\phi(F) \bydef \frac{1}{|V(F)|!} \sum_{\substack{E \subset E(F)\\ \text{spanning, connected}}}  (-1)^{|E|} \, .
\end{align}

The \emph{cluster expansion} of the partition function 
$\Xi_{A,B}$ (as defined in~\eqref{eqXidef}) is the formal power series
\begin{align}\label{eqclusterexp}
\ln \Xi_{A,B}=\sum_{\Gamma\in\cC}w_{A,B}(\Gamma)\, .
\end{align}

As mentioned above, establishing the convergence of the expansion in~\eqref{eqclusterexp} is a non-trivial task
and requires a careful combination of methods.
We use the graph container method of Korshunov and Sapozhenko~\cite{korshunov1983number, sapozhenko1987number}
to replace polymers
with `approximate polymers' that are less numerous than the polymers themselves. 
We then use entropy tools and isoperimetric properties of the torus
to bound both the total number and total 
weight of polymers that have a fixed approximation. 
Crucially, we make an extra saving in the count on polymers by 
exploiting an algebraically constructed vertex partition of $\Z_m^n$
with good `covering properties' (Lemma~\ref{lemcoverprop}).

Equipped with the convergence of the cluster expansion
we obtain the following approximation result obtained by truncating the cluster expansions of $ \Xi_{A,B}$. The \emph{size} of a cluster $\Gamma$ is 
 $\|\Gamma\| \bydef \sum_{\gamma \in \Gamma} |\gamma|$.   
For $k\ge 1$ we define
\begin{align}
\label{eqLkDef}
L_{A,B}(k) \bydef  \sum_{\substack{\Gamma \in \cC : \\ \|\Gamma\|= k}}  w_{A,B}(\Gamma)    \, .
\end{align} 

 \begin{theorem}
\label{cormain}
For $m$ an even integer, $k\in \N$ and $(H,\lam)$ a weighted graph, 
we have
\begin{align}
Z_G^H(\lam)=
 \eta_\lam(H)^{\frac{m^n}{2}}
\sum_{(A,B)\in \cD_\lam(H)}
 \exp \left\{
  \sum_{j=1}^{k-1} L_{A,B}(j) + \eps_k
 \right\}\, 
\end{align}
where $\eps_k=O(m^n n^{2(k-1)} \delta^{nk})$ 
for some $\delta\in[0,1)$ depending only on $(H,\lam)$.
In particular there exists a constant $K$ depending only on $(H,\lam)$ and $m$
 such that $\eps_{K}=o(1)$.
\end{theorem}

We emphasise that $\eps_K=o(1)$ for some constant $K$ in Theorem~\ref{cormain} and so to compute an explicit asymptotic formula (correct up to a $(1+o_n(1))$ factor) for $Z_G^H(\lam)$ one only needs to compute a constant number of the terms $L_{A,B}(j)$. In Section~\ref{secAlg} (see Theorem~\ref{lemLkform})
we show that for fixed $(H,\lam)$ and $m$,
one can obtain an explicit 
 expression for $L_{A,B}(j)$ as a function of $n$
 in time $e^{O(j \ln j)}$. 
 In particular obtaining an explicit asymptotic formula for $Z_G^H(\lam)$ 
 is a finite task.  We can of course go further and obtain yet more detailed approximations to $Z_G^H(\lam)$ by computing more of the terms $L_{A,B}(j)$.
 
 To give a concrete example of the type of expression that arises in Theorem~\ref{cormain},
we note that for a dominant pattern $(A,B)\in \cD_\lam(H)$,
\begin{align}\label{eqrefLAB1}
L_{A,B}(1)=\left[\frac{1}{\lam_A\lam_B^d}\sum_{v\in A^c}\lam_v \lam_{N(v)\cap B}^d + \frac{1}{\lam_B\lam_A^d}\sum_{v\in B^c}\lam_v \lam_{N(v)\cap A}^d\right]\frac{m^n}{2}\, ,
\end{align}
(where we recall that $d=(1+\ind_{m>2})n$, the degree of a vertex in $\Z_m^n$).

 Theorem~\ref{cormain}
 contains a variety of asymptotic formulae for
combinatorial quantities such as the number of proper $q$-colourings
and the number of independent sets in $\Z_m^n$.
As a special case we resolve a conjecture of Kahn and Park \cite[Conjecture 5.2]{kahn2018number}
on the number of proper $q$-colourings of the hypercube $Q_n$ for $q\in\{5,6\}$
(see Corollary~\ref{corcurious}).
 
Specialising to the case where
 $m=2$, $H=K_q$ and $\lam\equiv1$, we obtain an
 estimate for $c_q(Q_n)\bydef |\hom(Q_n, K_q)|$, the number of proper $q$-colourings of $Q_n$.  
 \begin{theorem}
\label{conjqcol}
For $q\ge3$,
\begin{align}
c_q(Q_n)= (1+\ind_{\{q \text{ odd}\}})
\left\lfloor \frac{q}{2} \right\rfloor^{2^{n-1}} \left\lceil \frac{q}{2} \right\rceil^{2^{n-1}} 
\binom{q}{\left\lfloor q/2 \right\rfloor}\cdot
\exp\left\{
f(n) + o(f(n))
\right\}\, ,
\end{align}
as $n\to\infty$
where 
\begin{align}
f(n)=
\frac{\lceil q/2 \rceil}{2 \lfloor q/2 \rfloor}
\left(2-\frac{2}{\lceil q/2 \rceil}\right)^n +
\frac{\lfloor q/2 \rfloor }{2 \lceil q/2 \rceil}
\left(2-\frac{2}{\lfloor q/2 \rfloor}\right)^n
\, .
\end{align}
\end{theorem}
In Section~\ref{seclabelqcol} we will strengthen
Theorem~\ref{conjqcol} considerably, providing 
 a detailed picture of the approximations to 
$c_q(\Z_m^n)$ contained in Theorem~\ref{cormain}.

Galvin~\cite{galvin2003homomorphisms} proved the $q=3$ case of  
Theorem~\ref{conjqcol} showing that $c_3(Q_n)\sim 6e2^{2^{n-1}}$
and 
recently Kahn and Park \cite{kahn2018number} established
 the case $q=4$:
$c_4(Q_n)\sim6e2^{2^n}$. 
Kahn and Park also tentatively conjecture that the 
error term, $o(f(n))$, in Theorem~\ref{conjqcol} 
(a conjecture at the time)
is negative.
We will show that this is in fact not the case (see Theorem~\ref{thmL1L2}),
this error term is positive for all $q$.

The expression in Theorem~\ref{conjqcol} was obtained by calculating $L_{A,B}(1)$
for a dominant pattern $(A,B)\in \cD_{\lam}(H)$ where $H=K_q, \lam\equiv 1$.
Theorem \ref{conjqcol} does not
give an asymptotic formula for $c_q(Q_n)$ (correct up to a $(1+o_n(1))$ factor) except in the
cases $q=3$ and $q=4$. 
However, there is no reason for us to stop 
at a calculation of $L_{A,B}(1)$.
We collect some asymptotic formulae that arise
from calculating more terms in the expansion of
Theorem~\ref{cormain}.
The following are all easily obtained by hand and 
the reader may enjoy deriving some further examples of their own.
We let $i(G)$ denote the number of independent sets in the graph $G$.

\begin{cor}\label{corcurious}
\[
c_5(Q_n)\sim 20\sqrt[3]{e}
\cdot 6^{2^{n-1}}
\exp\left\{
\left(\frac{4}{3}\right)^n
\right\}\, 
\]
\[
c_6(Q_n)\sim 20 \cdot 3^{2^n}\exp\left\{
\left(\frac{4}{3}\right)^n
\right\}\,
\]
\begin{align}
c_7(Q_n)\sim 70 \cdot e^{n/2} 12^{2^{n-1}}
\exp\left\{
\left(\frac{3}{2}\right)^{n-1}+\frac{1}{2}\left(\frac{4}{3}\right)^{n-1}
+\frac{n^2-n-54}{108}
\left(\frac{9}{8}\right)^{n-1}
\right\}\, 
\end{align}
\[
c_8(Q_n)\sim 70 \cdot 4^{2^{n}}
\exp\left\{
\left(\frac{3}{2}\right)^n+
\frac{n^2+41n-54}{108}
\left(\frac{9}{8}\right)^{n}
\right\}\, 
\]
\[
i(\mathbb Z_m^n)\sim 2^{m^{n}/2+1}\exp\left\{\frac{m^n}{2^{d+1}}\right\} \text{ for } m=2, 4,6,8,10,12,14
\]
\[
i(\mathbb Z_m^n)\sim2^{m^{n}/2+1}\exp\left\{\frac{1}{2}\left(\frac{m}{4}\right)^n + 2n(2n-1)\left(\frac{m}{16}\right)^n\right\} \text{ for } m=16,18,\ldots,62.
\]
\end{cor}

The first two expressions of Corollary~\ref{corcurious}
resolve a conjecture of 
Kahn and Park \cite[Conjecture 5.2]{kahn2018number}.
We also remark that the case
$m=2$ in the fifth expression of Corollary~\ref{corcurious}
is the asymptotic $i(Q_n) \sim  2 \sqrt{e}\cdot   2^{2^{n-1}} $,
a classical result of 
Korshunov and Sapozhenko~\cite{korshunov1983number}.
See also \cite{galvin2011threshold}, \cite{jenssen2019independent} and \cite{balogh2020independent} for 
further results on independent sets and the hard core model in the hypercube.

The expressions for $c_q(Q_n)$ in Corollary~\ref{corcurious}
were obtained by calculating $L_{A,B}(1)$, $L_{A,B}(2)$ for a
dominant pattern $(A,B)$. 
In Section~\ref{seclabelqcol} we show that more generally, in order to obtain
an asymptotic formula for $c_q(\Z_m^n)$, 
one necessarily has to compute the terms $L_{A,B}(1),\ldots, L_{A,B}(k-1)$ where
$k$ is the least integer for which
 $m(1-1/\lceil q/2 \rceil)^{(1+\ind_{m>2})k}<1$.

 \subsection{The defect distribution}\label{subsecdef}
Theorem~\ref{thmgenstruc0} shows that with probability 
$1-e^{-\Omega(m^n/n^3)}$, a sample from $\mu_{H,\lam}$
agrees with a dominant colouring on all but a $O(1/n)$ fraction 
of the vertices of $\Z_m^n$. 
Recall that we call the set of vertices at which an element $f\in\hom(\Z_m^n,H)$
differs from its closest dominant colouring (breaking ties arbitrarily if necessary)
the \emph{defect vertices} of $f$.
It is natural to ask for the distribution of the number of defect 
vertices in a sample from $\mu_{H,\lam}$.
By Theorem~\ref{mainTV} this amounts to understanding the
distribution of the number of vertices of $\Z_m^n$
contained in a sample from the polymer measures
$\nu_{A,B}$.
 
 We say that an event $\cA_n$ occurs
 \emph{with high probability} (whp)
 if $\P(\cA_n)\to 1$ as $n\to \infty$. 
 We use `$\overset{d}{\longrightarrow}$'
 to denote convergence in distribution 
 as $n\to \infty$.
 We let $\text{Pois}(\rho)$ denote the 
 Poisson distribution with mean $\rho$
 and let $N(0,1)$ denote the standard
  Normal distribution 
 with mean $0$ and variance $1$.
 
 We have the following central limit theorem for the measures 
 $\nu_{A,B}$.

\begin{theorem}\label{thmasymdist}
 Fix $m\ge2$ even, $(H,\lam)$ a weighted graph 
 and  $(A,B)\in\cD_\lam(H)$.
Let $\mathbf \Gamma$ be a random configuration drawn from the distribution $\nu_{A,B}$.
Letting $L\bydef L_{A,B}(1)$ and $\|\mathbf \Gamma\|\bydef \sum_{\gamma\in \mathbf\Gamma}|\gamma|$,
\begin{enumerate}
\item
if $L\to 0$ as $n\to\infty$, then $\|\mathbf\Gamma\|=0$ {whp}, 
\item 
if $L\to \rho$ for some constant $\rho>0$,
then $\|\mathbf\Gamma\|\overset{d}{\longrightarrow} \text{Pois}(\rho)$,
\item
if $L\to \infty$,
then 
\[
\frac{\|\mathbf\Gamma\|-L}{\sqrt{L}} \overset{d}{\longrightarrow} N(0,1)\, .
\]
\end{enumerate}
\end{theorem}

We record the following corollary of Theorem~\ref{thmasymdist} in the case of 
the $q$-colouring model on $\Z_m^n$. 

\begin{cor}\label{corqcol}
Let $X$ denote the number of defect vertices
in a uniformly chosen $q$-colouring of $\Z_m^n$.
\begin{itemize}
\item If $m\in\{2,4\}$ and $q\in\{3,4\}$ then $X \overset{d}{\longrightarrow} \text{Pois}(1)$,
\item otherwise $\frac{X-f(n)}{\sqrt{f(n)}} \overset{d}{\longrightarrow} N(0,1)$,
where
\[
f(n)=
\frac{m^n}{2}\left[
\frac{\lceil q/2 \rceil}{ \lfloor q/2 \rfloor}
\left(1-\frac{1}{\lceil q/2 \rceil}\right)^d +
\frac{\lfloor q/2 \rfloor }{ \lceil q/2 \rceil}
\left(1-\frac{1}{\lfloor q/2 \rfloor}\right)^d
\right]\, .
\, 
\]
\end{itemize}
\end{cor}

Specialising to the case $m=2$,
Corollary~\ref{corqcol} shows that 
$q$-colourings of $Q_n$ exhibit a 
sharp change in behaviour from $q=4$ to $q=5$
where a typical $q$-colouring goes from having a constant number of defect vertices 
to an exponential number.
This transition indicates why it is more difficult to enumerate 
$q$-colourings of $Q_n$ when $q\geq5$.

\subsection{Generalised rank functions and height functions}
A \emph{rank function} on the Boolean lattice $2^{[n]}$ is a function 
$
f: 2^{[n]}\to\N
$
such that $f(\emptyset)=0$ and
\[
0\leq f(A\cup\{x\})-f(A)\leq1
\]
for all $A\subset [n]$ and $x\in [n]\bs A$.
In an early application of entropy methods in combinatorics,
Kahn and Lawrenz~\cite{kahn1999generalized} showed that the number of rank functions on $2^{[n]}$
is at most $2^{(1+o(1))2^{n-1}}$, thus resolving a conjecture of Athanasiadias \cite{athanasiadis1996algebraic}.
The authors in fact consider a natural generalisation of the notion of rank function:
a function $f: 2^{[n]}\to\N$ is \emph{$k$-bounded},
if $f(\emptyset)=0$ and
\[
0\leq f(A\cup\{x\})-f(A)\leq k
\]
for all $A\subset [n]$ and $x\in [n]\bs A$
\footnote{For notational convenience, we have modified the definition in~\cite{kahn1999generalized} where $f: 2^{[n]}\to\N$ is said to be $k$-bounded if $f(\emptyset)=0$ and
$ 0\leq f(A\cup\{x\})-f(A)< k$ for all $A\subset [n]$ and $x\in [n]\bs A$.}.
Let $\cB_k(n)$ denote the set of 
all $k$-bounded functions. 
Kahn and Lawrenz~\cite{kahn1999generalized} show 
that 
\[
\ln |\cB_k(n)|\leq(1+O(n^{-1/2})){2^{n-1}}\ln \left(\left \lfloor \frac{k}{2}+1 \right\rfloor \left\lceil \frac{k}{2}+1 \right\rceil\right)
\]
and conjectured that in fact 
\(
\ln |\cB_k(n)|\leq{2^{n-1}}\log \left(\left \lfloor \frac{k}{2}+1 \right\rfloor \left\lceil \frac{k}{2}+1 \right\rceil\right) + O(1),
\)
\cite[Conjecture 1]{kahn1999generalized}. 
We disprove this conjecture, 
providing an asymptotic expression for $|\cB_k(n)|$ which shows that the conjecture
holds if and only if $k=1$ or $2$.

\begin{theorem}\label{thmkbd}
For fixed $k\geq1$,
\[
 |\cB_k(n)|=(1+\ind_{k\text{ odd}}) \left(\left \lfloor \frac{k}{2}+1 \right\rfloor \left\lceil \frac{k}{2}+1 \right\rceil\right)^{2^{n-1}}
  \exp\left\{\frac{(1+\ind_{k\text{ even}})}{\lfloor k/2+1\rfloor}\left(\frac{2 \lceil k/2 \rceil}{\lceil k/2+1\rceil}\right)^n(1+o(1))\right\}.
\]
\end{theorem}

For the proof we construct a bijection between the set 
$\cB_k(n)$ to a subset of 
$\hom(Q_n, H)$ for an 
appropriately chosen Cayley graph $H$ 
and then apply Theorem~\ref{cormain}
(in particular more precise asymptotics than those recorded in
 Theorem~\ref{thmkbd} are available to us).
In the special case of rank functions, 
such a bijection was already well-known: as observed by Mossel (see~\cite{kahn2001range}),
there is a bijection between the set of rank functions $\cB_1(n)$
and the set of \emph{height functions} (also known as `cube-indexed random walks')
 \[
\mathcal{F}(n)\bydef \{f: V(Q_n)\to\mathbb Z: f(\mathbf 0)=0 \text{ and } u\sim v \implies |f(u)-f(v)|=1  \}\, 
 \]
 (here $\mathbf 0$ denotes the vertex $(0,\ldots,0)$ of $Q_n$).
 Randall (see \cite{galvin2003homomorphisms}) 
 observed that there is bijection between the set $\mathcal{F}(n)$
 and the set of proper $3$-colourings of $Q_n$ which assign the vertex $\mathbf 0$ 
a fixed colour.

 For $f\in \mathcal{F}(n)$, let 
 \[
 R(f)= |\{f(x) : x\in V(G)\}|\, ,
 \]
 the size of the range of $f$.  
In Section~\ref{seckbd} we sketch how a large deviation inequality for our polymer models 
(Theorem~\ref{lempolyLD}, a key ingredient in the proofs of Theorems~\ref{mainTV} and~\ref{thmgenstruc0}) can be used to recover the following special case of a result due to Peled~\cite{peled2017high}.
  
Henceforth we let $\log$ denote
the base $2$ logarithm and let $\ln$
denote the natural logarithm. 
The \emph{binary entropy function} is the map
$H:[0,1]\to\R$ given by 
\[
H(p)=-p\log p-(1-p)\log (1-p)\, 
\]
where we interpret $0\log 0$ as $0$.

 \begin{theorem}\label{thmBHM}
If $t\in(0,1]$, 
 and $f$ is chosen uniformly at random from $\mathcal{F}(n)$, then
 \[
\log \log\left(1/ \P(R(f)\geq tn) \right)\sim H(t/2)n\, .
 \]
 \end{theorem}
 The upper bound on $\P(R(f)\geq tn)$ in Theorem~\ref{thmBHM}
 is a special case of \cite[Theorem 2.1]{peled2017high} which applies to a general class of tori
 including 
$\Z_m^n$ (where $m$ is allowed to be large with respect to $n$)
and provides strong bounds on
\(
 \P(R(f)\geq k)
 \)
for arbitrary $k$. 
The lower bound~in Theorem~\ref{thmBHM} seems to be new, though it is not difficult and follows from a simple construction. 
 
Another well-studied class of functions in this context are Lipschitz functions (see for example~\cite{peled2017high, peled_samotij_yehudayoff_2013,  peled2013grounded}). We note that the tools of this section extend naturally to the analysis of Lipschitz functions on $Q_n$ (and more generally $\Z_m^n$ for $m$ even) but we do not pursue the details here.

\subsection{Torpid Mixing}\label{subsectorpid}
In this section we discuss an algorithmic 
implication of Theorem~\ref{mainTV}.
We establish torpid mixing for a natural class
of Markov chain algorithms (including Glauber dynamics)
designed to sample from the space $\hom(\Z_m^n,H)$.

The main tool we use for deriving algorithmic results from 
Theorem~\ref{mainTV}
is the notion of $\emph{conductance}$
of a Markov chain. Conductance was first 
introduced to the field of Markov chain algorithms by
Jerrum and Sinclair \cite{sinclair1989approximate}.
 
Let $\cM$ be a Markov chain
on a finite state space $\Omega$, 
with transition probabilities
$P(\omega, \omega')$, $\omega, \omega'\in\Omega$.
 Let $\pi$ denote the stationary distribution of $\cM$. 
 For $\omega_0\in \Omega$, we denote by
$P_{t,\omega_0}(\omega)$ the probability that
the system is in the state $\omega$ at time $t$
 given that $\omega_0$ is the initial state.
 The mixing time of the Markov chain $\cM$ is defined as
 \[
 \tau_\cM\bydef  \min
  \left\{t_0: \max_{\omega_0\in\Omega}
  \|P_{t,\omega_0}-\pi\|_{TV}\le \frac{1}{e}\, , \,  \forall t>t_0
   \right\}
 \]
where $\|\cdot\|_{TV}$ denotes {total variation distance}.

The Markov chains we consider 
will have state space $\Omega\subset \hom(\Z_m^n,H)$
for some fixed graph $H$. 
We call a Markov chain \emph{$\beta$-local}
if $f_1, f_2$ differ at $\le \beta m^n$ vertices 
whenever $P(f_1, f_2)\neq 0$.
In other words, the chain updates 
the colours of $\le \beta m^n$ vertices at each step.
We only consider Markov chains that are ergodic
(i.e.\ connected and aperiodic).

If $H$ is connected and bipartite,
with vertex partition $V(H)=V^{+}\cup V^{-}$,
then any $f\in \hom(\Z_m^n,H)$ satisfies either
(i) $f(\cO)\subset V^{+}, f(\cE)\subset V^{-}$
or
(ii) $f(\cE)\subset V^{+}, f(\cO)\subset V^{-}$.
In particular, if $\beta<1$, a $\beta$-local 
Markov chain on $\hom(\Z_m^n,H)$ cannot be connected.
In this case, we let $\hom^{+}(\Z_m^n,H)$ denote the set of 
$f$ satisfying (i) and 
let $\hom^{-}(\Z_m^n,H)$ denote the set of 
$f$ satisfying (ii). 
Further, we let $\mu_{H,\lam}^{+}$
denote the measure $\mu_{H,\lam}$
conditioned on $\hom^{+}(\Z_m^n,H)$ and 
define $\mu_{H,\lam}^{-}$ similarly.

Finally, let us call a non-bipartite weighted graph $(H,\lam)$ \emph{trivial}
if $(H,\lam)$ has only one dominant pattern.
We say that $(H,\lam)$ is non-trivial otherwise. 
We call a bipartite $(H,\lam)$ trivial if it only has two dominant patterns.
Note that any bipartite $(H,\lam)$ has at least two dominant patterns since if 
$(A,B)$ is dominant, then $(B,A)$ is a distinct dominant pattern.

\begin{theorem}\label{thmslowmix}
For $(H,\lam)$ a fixed non-bipartite, non-trivial, weighted graph,
there exists $\beta>0$
 such that for $m\ge 2$ even,
 any $\beta$-local ergodic 
 Markov chain 
$\cM$
on $\hom(\Z_m^n,H)$
with stationary distribution $\mu_{H,\lam}$ 
satisfies
\[
\tau_\cM= e^{\Omega(m^n/n^2)}\, .
\]
For $(H,\lam)$ bipartite, the analogous statement holds with $\hom(\Z_m^n,H)$ replaced with 
$\hom^{\pm}(\Z_m^n,H)$ and $\mu_{H,\lam}$ replaced with $\mu^{\pm}_{H,\lam}$. 
\end{theorem}

In the case of the hard-core model, Theorem~\ref{thmslowmix} follows from arguments of
Galvin and Tetali~\cite{galvin2006slow} who establish torpid mixing for the Glauber dynamics of the  hard-core model on a class of bipartite expander graphs. 
In the regime where $n$ is fixed and $m$ is large, analogues of Theorem~\ref{thmslowmix} have been proved for specific choices of $(H,\lam)$ \cite{chayes2004sampling, galvin2008sampling, galvin2007torpid}.  We return to this topic in the concluding remarks.

\subsection{Outline}
The outline for the remainder of this paper is as follows. 
In Section~\ref{secCluster} we introduce abstract polymer models 
and the cluster expansion
 and in Section~\ref{secHcol} we introduce the concrete polymer models 
 that we work with in this paper and explore some of their basic properties.
 
 In Section~\ref{secover}, we give an  overview of the proof of our main theorem, 
 Theorem~\ref{mainTV} and show how it follows from two key lemmas (Lemmas~\ref{lemH01} and~\ref{lemH21}). We also show how the asymptotic approximations to $Z_G^H(\lam)$ in Theorem~\ref{cormain} follow from these key lemmas and the convergence of the cluster expansion.

In Section~\ref{secentropy} we collect some standard tools based on entropy.

In Sections~\ref{seciso} and~\ref{secapprox} we establish some isoperimetric results
in the discrete torus and introduce some notions from the theory of graph containers. 

Sections~\ref{secgroup},~\ref{seccHg} and \ref{secVKP} are then dedicated to proving the convergence of the cluster expansion of our polymer models.

In Section~\ref{secLD} we show how the convergence of the cluster expansion can be used to prove a large deviation result for the polymer measures $\nu_{A,B}$.
In Section~\ref{seccapture}, we show how this large deviation result and an entropy argument can be used to establish the key lemmas, Lemmas~\ref{lemH01} and~\ref{lemH21}, thus concluding the proof of our main theorem, Theorem~\ref{mainTV} and Theorem~\ref{cormain}.

In Section~\ref{secgenstruc}
we show how Theorem~\ref{mainTV} 
can be used to give a precise structural description of the set $\hom(\Z_m^n, H)$ thus resolving conjectures of Engbers and Galvin \cite[Conjectures 6.1, 6.2,  6.3]{engbers2012h}.

In Section~\ref{sectorpid} we prove the torpid mixing result, Theorem~\ref{thmslowmix}.

In Section~\ref{secdefect} we prove Theorem~\ref{thmasymdist} and Corollary~\ref{corqcol}
on the asymptotic distribution of the number of defect vertices.

In Section~\ref{secAlg} we provide an algorithm for computing the terms of the cluster expansion. 
In Section~\ref{seclabelqcol} we specialise the discussion to the $q$-colouring model and provide a detailed picture of the cluster expansion in this setting.

In Section~\ref{seckbd}
 we prove Theorem~\ref{thmkbd} and
then Theorem~\ref{thmBHM} on $k$-bounded functions and height functions.

In Section~\ref{secconc} we end with some concluding remarks and directions of future research.

\subsection{Notation and terminology}
Here we gather some notation for ease of reference.

For $k\in \N$, 
we let $[k]$ denote the set $\{1,\ldots, k\}$. 
For a set $X$, we let $2^X$ denote the power set of $X$ and for $k\in \N$
let $\binom{X}{k}$ denote the collection of subsets of $X$ of size $k$.

Given a graph $G=(V, E)$ and $x\in V$, we let
$N(x)=\{y\in V: x\sim y\}$ and for $X\subset V$,
we let $N(X)=\bigcup_{x\in X} N(x)$.
For $X\subset V$ we let
$\p(X)=N(X)\bs X$, 
the \emph{vertex boundary} of $X$.
We let $X^+=X\cup \p (X)$.
For $v\in V$, we let $d(v)=|N(v)|$
and for a subset $Y\subset V$
we let $d_Y(v)\bydef |N(v)\cap Y|$.
For two vertices $u,v\in V$ we write $d_G(u,v)$
for the graph distance between $u, v$ i.e.\ the length of 
the shortest path in $G$ with endpoints $u, v$.
For a set $S\subset V$, we let $G[S]$ denote the graph
 with vertex set $S$ and edge set $E\cap \binom{S}{2}$.
For $k\in \N$ we let $G^k$ denote the $k$th power of $G$, that is,
the graph with vertex set $V(G)$ such that $u\sim v$ if and only if
$d_G(u,v)\leq k$. 
We say a set $S\subset V$ is \emph{$G^k$-connected} if 
$G^k[S]$ is connected. 

Throughout this paper we use $\log$ to denote the base $2$ logarithm
and we use $\ln$ to denote the natural logarithm.

Henceforth we can assume that we have fixed a weighted graph $(H,\lam)$ and
an even integer $m\geq 2$. Standard asymptotic notation such as $O, \Omega, o, \Theta$, 
is used under the assumption $n\to\infty$ and 
the implicit constants
may depend on $(H,\lam)$ and $m$.

\section{Abstract polymer models and the cluster expansion}
\label{secCluster}

In this section we describe two classic tools from statistical physics, abstract polymer models~\cite{kotecky1986cluster} and the cluster expansion. 
Both tools have been used extensively to study phase diagrams of lattice spin models. 
Following the breakthough work of Helmuth, Perkins and Regts~\cite{helmuth2018contours}, these tools (along with contour models from Pirogov-Sinai Theory) have found a wealth of applications in an algorithmic setting: namely the design of efficient approximate counting and sampling algorithms for spin models on graphs~\cite{borgs2020efficient, cannon2020counting, galanis2020fast, helmuth2020finite, helmuth2018contours, JKP2, liao2019counting}.

We have already encountered the terms `polymer' and `cluster' in the previous section. 
Indeed, the polymers from the introduction are concrete examples of a more general notion which we introduce now. 

Let $\cP$ be a finite set whose elements we call `polymers'. We equip $\cP$ with a complex-valued weight $w(\gam)$ for each polymer $\gam\in \cP$ as well as a symmetric and reflexive incompatibility relation between polymers. We write $\gam \nsim \gam'$ if polymers $\gam$ and $\gam'$ are incompatible. Let $\Omega$ be the collection of pairwise compatible sets of polymers from $\cP$, including the empty set of polymers.  Then the polymer model partition function is
\begin{align*}
\Xi(\cP) &= \sum_{\Gamma \in \Omega} \prod_{\gam\in \Gamma} w(\gam) \,,
\end{align*}
where the contribution from the empty set is $1$.

For a multiset of polymers $\Gamma$,
we define the \emph{incompatibility graph}
$I_{\Gamma}$ to be the graph on vertex set $\Gamma$
where $\gam_1$ is adjacent to $\gam_2$ if and only if
$\gam_1, \gam_2$ are incompatible.
A \textit{cluster} is an \emph{ordered} multiset of polymers 
whose {incompatibility} graph is connected. 
Let $\cC$ be the set of all clusters. The \textit{cluster expansion} is the formal power series in the weights $w(\gam)$
\begin{align}\label{eqclustexp}
\ln \Xi(\cP) &= \sum_{\Gamma \in \cC} w(\Gamma) \,,
\end{align}
where
\begin{align}\label{eqwGam}
w(\Gamma) &=  \phi(I_{\Gamma}) \prod_{\gam \in \Gamma} w(\gam) \,,
\end{align}
and $\phi$ is the Ursell function as defined in~\eqref{eqUrsell}. 
In fact the cluster expansion is simply the multivariate Taylor series for $\ln \Xi(\cP)$ in the variables $w(\gam)$, as observed by Dobrushin~\cite{dobrushin1996estimates}.  See also Scott and Sokal~\cite{scott2005repulsive} for a derivation of the cluster expansion and much more.  

A sufficient condition for the convergence of the cluster expansion is given by a theorem of Koteck\'y and Preiss.
\begin{theorem}[Koteck\'y and Preiss \cite{kotecky1986cluster}]
\label{thmKP}
Let $f : \cP \to [0,\infty)$ and $g: \cP \to [0,\infty)$ be two functions.  Suppose that for all polymers $\gam \in \cP$, 
\begin{align}
\label{eqKPcond}
\sum_{\gam' \nsim \gam}  |w(\gam')| e^{f(\gam') +g(\gam')}  &\le f(\gam) \,,
\end{align}
then the cluster expansion converges absolutely.  Moreover, if we let $g(\Gamma) = \sum_{\gam \in \Gamma} g(\gam)$ and write $\Gamma \nsim \gam$ if there exists $\gam' \in \Gamma$ so that $\gam \nsim \gam'$, then for all polymers $\gam$,
\begin{align}
\label{eqKPtail}
 \sum_{\substack{\Gamma \in \cC \\  \Gamma \nsim \gam}} \left |  w(\Gamma) \right| e^{g(\Gamma)} \le f(\gam) \,.
\end{align}
\end{theorem}

As a preview of one of the applications of the above theorem, we remark that \eqref{eqKPtail} can be used to give tail bounds on the cluster expansion. This will allow us to show that certain truncations of the cluster expansion serve as good approximations to the logarithm of the partition function.

\section{Concrete polymer models} \label{secHcol}
In this section we formally introduce the 
concrete polymer models that will be the subject of study
in this paper and explore some of their properties. 
These polymer models are a generalisation of those used by the current authors and Perkins~\cite{JKP2} to design efficient approximate counting and sampling 
algorithms for the hard-core model and $q$-colouring model on bipartite expander graphs.

We repeat some of the notation and definitions from the introduction
for the reader's convenience. 
Throughout this paper we let $G$ denote the $n$-dimensional discrete torus $\Z_m^n$,
 the graph on vertex set $\{0,\ldots, m-1\}^n$ where two vertices $x, y$
are adjacent if and only if there exists a coordinate $i\in [n]$ such that
$x_i=y_i \pm1 \pmod m$ and $x_j=y_j$ for all $j \neq i$.
We consider only the case where $m\ge2$ is even, in which case
$G$ is a bipartite graph
with vertex classes
\begin{align}
\cE\bydef  \left\{ x: \sum_{i=1}^n x_i \equiv 0 \pmod 2 \right\} \text { and }
\cO\bydef  \left\{ x: \sum_{i=1}^n x_i \equiv 1 \pmod 2 \right\} \, .
\end{align}
Note that we have $|\cO|=|\cE|=m^{n}/2$.
We also let $V=V(G)=\cO\cup\cE$.
When $m=2$,
the graph $\Z_m^n$ is the familiar $n$-dimensional hypercube
and we use the more standard notation $Q_n$ in this case.
We note that $Q_n$ is an $n$-regular graph,
whereas $\Z_m^n$ is a $2n$-regular graph when $m>2$.
To avoid having to reiterate this distinction,
we will simply say that $\Z_m^n$ is $d$-regular,
where $d=(1+\ind_{m>2})n$.

Recall that given a subset $S\subset V$, a colouring $f$ and a pattern $(A,B)$,
we say that $f$ \emph{agrees} with $(A,B)$ at $v\in V$ if 
$v\in\cO$ and $f(v)\in A$ or if $v\in\cE$ and $f(v)\in B$.
We say that $f$ disagrees with $(A,B)$ at $v$ otherwise.  
Let $\chi_{A,B}(S)$ be the set of colourings $f$ such that 
$f$ \emph{disagrees} with $(A,B)$ at each $v\in S$ 
and agrees with $(A,B)$ at each $v\in V\bs S$.
For each pattern $(A,B)$
 we have a weight function $w_{A,B}: 2^V\to \R$
 where 
 \begin{align}\label{eqwdef}
 w_{A,B}(S)=\frac{\sum_{f \in \chi_{A,B}(S)} \prod_{v\in V}\lam_{f(v)}}
 {\eta_\lam(H)^{m^n/2}}\, .
\end{align}
 
 For a colouring $f\in \hom(G,H)$ and subset $S\subset V$, 
 we let $f_S$ denote the restriction of $f$ to $S$.
 A more convenient 
 (from a computational perspective)
form for the weight function $w_{A,B}$
 is the following: 
let $\hat \chi_{A,B}(S)$ be the set of colourings $f_{S^+}$
where $f\in \chi_{A,B}(S)$.
Given any $h\in \hat \chi_{A,B}(S)$,
we may extend $h$ to an element of $\chi_{A,B}(S)$ 
by arbitrarily assigning vertices of $\cO\bs S^+$ colours from $A$
and arbitrarily assigning vertices of $\cE\bs S^+$ colours from $B$.
Thus
\begin{align}\label{eqconv}
w_{A,B}(S)=
 \frac{\sum_{f\in\hat\chi_{A,B}(S)} \prod_{v\in S^+}\lam_{f(v)}}{\lam_A^{|S^+\cap\cO|} \lam_B^{|S^+\cap\cE|}}
\, .
\end{align}
 
 To illustrate the point,
 let us calculate $w_{A,B}(S)$
 in the case where $S=\{v\}$
for some $v\in \cE$.
In this case $\hat\chi_{A,B}(S)$
is the set of homomorphisms
$f: G[v\cup N(v)]\to H$
such that $f\in A^c$ and $f(N(v))\subset B$.
Using~\eqref{eqconv}, we have
\[
w_{A,B}(S)=\frac{1}{\lam_A\lam_B^d}\sum_{v\in A^c}\lam_v \lam_{N(v)\cap B}^d\, .
\]

We now define our collection of polymer models and partition functions.
Recall that we say that $S\subset V(G)$ is 
\emph{$G^2$-connected} if the graph $G^2[S]$ is connected
(here $G^2$ denotes the square of the graph $G$).

\begin{defn}\label{defpoly}
 We say a subset $\gam \subset V$ is a \emph{polymer}
if it is $G^2$-connected and
 $|N(\gam\cap\cE)|, |N(\gam\cap\cO))|<(1-\alpha)m^n/2$
 where $0<\alpha<1$ is some constant 
(dependent only on $(H,\lam)$)
which will be specified later (see \eqref{eqadef}).

 We say that two polymers $\gam_1, \gam_2$ are
 \emph{compatible}, denoted $\gam_1\sim \gam_2$, 
 if $d_G(\gam_1, \gam_2)>2$
(i.e.\ $\gam_1\cup\gam_2$ is not $G^2$-connected).
\end{defn}

We let $\cP$ denote the set of all polymers in $G$ 
and let $\Omega$ denote the family of all sets of 
mutually compatible polymers from $\cP$. 

\begin{defn}\label{defnABpoly}
For each dominant pattern $(A,B)\in\cD_\lam(H)$,
the \emph{$(A,B)$ polymer model} is the 
polymer model with polymer set $\cP$,
compatibility relation $\sim$ and 
weight function $w_{A,B}$.
\end{defn}

The partition function of the $(A,B)$ polymer model is
\begin{align}\label{eqpfdef}
\Xi_{A,B}=\sum_{\Gam\in \Omega}\prod_{\gam\in\Gam} w_{A,B}(\gam)\, .
\end{align}

We thus have $|\cD_\lam(H)|$
distinct polymer models determined 
by the choice of weight function $w_{A,B}$
where $(A,B)\in\cD_\lam(H)$
(note that the set of polymers, $\cP$, 
is the same for each polymer model).

A useful property of the weight function
is that it is multiplicative over $G^2$-connected components of a set $S$.
 
\begin{lemma}\label{lemmult}
Let $(A,B)\in \cD_\lam(H)$.
For $S\subset V$ with 
$G^2$-connected components $\gam_1, \ldots, \gam_k$
we have
\begin{align}
w_{A,B}(S)=\prod_{i=1}^k w_{A,B}(\gam_i)\, .
\end{align}
\end{lemma}
\begin{proof}
Since the $\gam_i$ are the $G^2$-connected components of $S$,
the sets $\gam_i^+$ are mutually disjoint. 
Moreover there is a one-one correspondence between elements
$f\in \hat\chi_{A,B}(S)$ and tuples
$(f_1,\ldots,f_k)\in \hat\chi_{A,B}(\gamma_1)\times\ldots\times \hat\chi_{A,B}(\gamma_k)$.
Using the expression~\eqref{eqconv} 
for the weight function we have
\begin{equation}\label{eqextend}
w_{A,B}(S)=
 \frac{\sum_{f\in\hat\chi_{A,B}(S)} \prod_{v\in S^+}\lam_{f(v)}}{\lam_A^{|S^+\cap\cO|} \lam_B^{|S^+\cap\cE|}}
=\prod_{i=1}^k
 \frac{\sum_{f\in\hat\chi_{A,B}(\gam_i)} \prod_{v\in \gam_i^+}\lam_{f(v)}}{\lam_A^{|\gam_i^+\cap\cO|} \lam_B^{|\gam_i^+\cap\cE|}}
 =\prod_{i=1}^k w_{A,B}(\gam_i)
 \, .
\end{equation}
\end{proof}

The partition function $\Xi_{A,B}$ from~\eqref{eqpfdef} is the normalising constant
of a probability distribution $\nu_{A,B}$ on $\Omega$ defined by 
\begin{align}\label{eqnudef}
\nu_{A,B}(\Gamma)=\frac{\prod_{\gam\in\Gamma}w_{A,B}(\gam)}{\Xi_{A,B}}\, .
\end{align}

These measures  allow us to build a new probability measure
$\hat\mu_{H,\lam}$ on $\hom(G,H)$.

\begin{defn}\label{defmuhat}
 For $f\in \hom(G,H)$, let 
$\hat \mu_{H,\lam}(f)$ 
denote the probability that $f$ is selected by the following
four-step process:

\begin{enumerate}
\item \label{step1} Choose $(A,B)\in \cD_\lam(H)$ with probability proportional to $\Xi_{A,B}$.
\item \label{step2} Choose a random polymer configuration $\Gamma \in \Omega$
from the distribution $\nu_{A,B}$.
\item \label{step3} Letting $S=\bigcup_{\gam\in\Gamma} \gam$, select a colouring $f\in\hat\chi_{A,B}(S)$ with probability proportional to $\prod_{v\in S^+}\lam_{f(v)}$.
\item \label{step4} Independently assign each $v\in\cO\bs S^+$ the colour $i\in A$ with probability $\lam_i/\lam_A$ and each $v\in\cE\bs S^+$ the colour $j\in B$ with probability $\lam_j/\lam_B$. 
\end{enumerate}
\end{defn}

We note that this definition is equivalent to 
Definition~\ref{defmuhatintro} from the introduction where 
we have expanded the last step 
into two steps. We do this to 
emphasise the independence in Step \ref{step4}
which will be 
useful for later calculations.

We end this section by recording 
a couple of the 
basic properties of the measure 
$\hat\mu_{H,\lam}$.

\begin{defn}\label{defcap}
Let $(A,B)\in \cD_\lam(H)$.
We say that a colouring $f\in\hom(G,H)$ 
is \emph{captured} by $(A,B)$
if each of the $G^2$-connected components of
$S(f)\bydef (f^{-1}(A^c)\cap \cO)\cup (f^{-1}(B^c)\cap \cE)$
 is a polymer.
\end{defn}

In the following lemma, 
we let $\mathbf f$ denote a random colouring
chosen according to $\hat \mu_{H,\lam}$.
Let $\mathbf D$ denote the random pattern 
selected at Step \ref{step1} in the definition of $\hat \mu_{H,\lam}$ 
(Definition~\ref{defmuhat}).

\begin{lemma}\label{eqpcount}
Let $f\in \hom(G,H)$.
\begin{enumerate}[label=(\roman*)]
\item \label{itemf} If $f$ is captured by precisely $p$ polymer models then
\begin{align}
\P(\mathbf f = f)= p\frac{\prod_{v\in V}\lam_{f(v)}}{\tilde Z_G^H(\lam)}\, .
\end{align}
\item \label{itemD} For $(A,B)\in\cD_\lam(H)$, 
\begin{align}
\P(\mathbf f =f \mid \mathbf D=(A,B))= 
\frac{\prod_{v\in V}\lam_{f(v)}}{\eta_\lam(H)^{m^n/2}\cdot\Xi_{A,B}}
\cdot \ind_{\{\text{$f$ captured by $(A,B)$}\}}\, .
\end{align}
\end{enumerate}
\end{lemma}
\begin{proof}
We begin with the second claim. 
Fix $f\in \hom(G,H)$, $(A,B)\in \cD_{\lam}(H)$
and let $S=S(f)$ as in Definition~\ref{defcap}.
Let the $G^2$-connected components of $S$
be $\gamma_1,\ldots, \gamma_k$ and let 
$\Gamma=\{\gamma_1,\ldots, \gamma_k\}$.
If $f$ is captured by the $(A,B)$ polymer model then
\begin{align}
\P(\mathbf f =f \mid \mathbf D=(A,B))&= 
\frac{\prod_{\gam\in\Gamma}w_{A,B}(\gam)}{\Xi_{A,B}}\cdot 
\frac{\prod_{v\in V}\lam_{f(v)}}{\sum_{h\in\chi_{A,B}(S)}\prod_{v\in V}\lam_{h(v)}}\\
&= \frac{\prod_{v\in V}\lam_{f(v)}}{\eta_\lam(H)^{m^n/2}\cdot\Xi_{A,B}}
\end{align}
where for the final equality we 
used~\eqref{eqconv} and Lemma~\ref{lemmult}.
If $f$ is not captured by the $(A,B)$ polymer model then
one of the $G^2$-connected components of $S$ 
is not a polymer and so the probability that 
$\Gamma$ is selected at Step \ref{step2} in Definition~\ref{defmuhat}
is $0$. Therefore $\P(\mathbf f =f \mid \mathbf D=(A,B))=0$.
The first claim follows by noting that
  \[
  \P(\mathbf D=(A,B))=\frac{\Xi_{A,B}}{\sum_{(C,D)\in \cD_\lam(H)}\Xi_{C,D}}
  \]
  for any dominant pattern $(A,B)$.
  \end{proof}

\section{Overview of the proof}\label{secover}
In this section we give an overview of the proof of our main theorem,
Theorem~\ref{mainTV}. We also show how the proof of Theorem~\ref{mainTV} ties in with the proof of Theorem~\ref{cormain} which provides detailed asymptotic approximations to the partition function $Z_G^H(\lam)$.

Theorem~\ref{mainTV} asserts that 
$\hat \mu_{H,\lam}$ is very close to the original
spin measure $\mu_{H,\lam}$ (defined in~\eqref{eqmuoriginal}) in total variation distance. 
On the way to proving Theorem~\ref{mainTV} we will also show
 that a linear combination
of our polymer model partition functions $\Xi_{A,B}$
serves as a very good approximation to the 
spin model partition function
 $Z_G^H(\lam)$ (defined in~\eqref{eqZGdef}).

\begin{lemma}\label{lempolymerapprox}
With
\[
\tilde Z_G^H(\lam)\bydef \eta_\lam(H)^{m^n/2}
\sum_{(A,B)\in\cD_\lam(H)}\Xi_{A,B}\, ,
\]
we have 
\[
1- e^{-\Omega(m^n/n^2)} \leq \frac{\tilde Z_G^H(\lam)}{ Z_G^H(\lam)}\leq 1+ e^{-\Omega(m^n/n^2)}\, .
\]
\end{lemma}

In order to prove Lemma~\ref{lempolymerapprox}, 
we will show that almost every element of $\hom(G,H)$
is captured (see Definition~\ref{defcap}) by precisely $1$ polymer model. 
This is also key for the proof of Theorem~\ref{mainTV}.

Let $\hom_0(G,H)$, $\hom_1(G,H)$, $\hom_2(G,H)$ denote the sets of all colourings
which are captured by $0$, precisely $1$, and $\geq2$ dominant patterns respectively.

\begin{lemma}\label{lemH01}
There exists $\zeta=\zeta(H,\lam)<\eta_\lam(H)$ so that 
\begin{align}
Z_0\bydef \sum_{f\in\hom_0(G,H)}\prod_{v\in V}\lam_{f(v)}
\le \zeta^{m^n/2}\, .
\end{align}
\end{lemma}

\begin{lemma}\label{lemH21}
 \begin{align}
Z_2\bydef \sum_{f\in\hom_2(G,H)}\prod_{v\in V}\lam_{f(v)}\le e^{-\Omega(m^n/n^2)}
\tilde Z_G^H(\lam)\, .
\end{align}
\end{lemma}

We will prove Lemmas~\ref{lemH01} and~\ref{lemH21} in Section~\ref{seccapture}. 
We now show how both Lemma~\ref{lempolymerapprox} and 
Theorem~\ref{mainTV}
follow from these two lemmas.

\begin{proof}[Proof of Lemma~\ref{lempolymerapprox} assuming Lemmas~\ref{lemH01} and~\ref{lemH21}]

By Lemma~\ref{eqpcount}~\ref{itemf} we have
\[
\tilde Z_G^H(\lam)=\sum_{f\in \hom(G,H)}p_f \prod_{v\in V}\lam_{f(v)}\, ,
\]
where $p_f$ denotes the number of polymer models
that capture $f$.
Comparing this to 
\[
Z_G^H(\lam)=\sum_{f\in \hom(G,H)} \prod_{v\in V}\lam_{f(v)}
\]
we obtain
 \[
 \tilde Z_G^H(\lam) +Z_0-4^q Z_2 \le Z_G^H(\lam)\le \tilde Z_G^H(\lam) + Z_0 - Z_2\, ,
 \]
where $4^q$ is a bound on the number of dominant patterns. 
The lemma follows
 from Lemmas~\ref{lemH01} and \ref{lemH21}. 
 \end{proof}
 
\begin{proof}[Proof of Theorem~\ref{mainTV} assuming Lemmas~\ref{lemH01} and~\ref{lemH21}]
Let $\mu, \hat \mu$ denote $\mu_{H,\lam}, \hat \mu_{H,\lam}$
respectively.
Recalling a 
formula for the total variation distance
between discrete probability measures
we have
\begin{align}\label{eqrecall}
\|\mu-\hat \mu\|_{TV}
&=
\sum_{\substack{f\in \hom(G,H):\\ \hat \mu(f)> \mu(f)}} 
\hat \mu(f) - \mu(f)\, .
\end{align}
We consider the contribution to the above sum from 
elements of $\hom_0(G,H)$, $\hom_1(G,H)$ and $\hom_2(G,H)$ separately.
If $f\in  \hom_0(G,H)$ then $\hat \mu (f)=0$ by Lemma~\ref{eqpcount}~\ref{itemf}
and so elements of $\hom_0(G,H)$ do not contribute to the sum in~\eqref{eqrecall}. 
Since $\tilde Z_G^H(\lam)$
is within a factor $1+e^{-\Omega(m^n/n^2)}$ of  $Z_G^H(\lam)$ 
by Lemma~\ref{lempolymerapprox} (which holds by our assumption that Lemmas~\ref{lemH01} and~\ref{lemH21} hold)
we have, 
\begin{align}
\sum_{\substack{f\in \hom_1(G,H):\\ \hat \mu(f)> \mu(f)}} 
\hat \mu(f) - \mu(f)
&=
\sum_{\substack{f\in \hom_1(G,H):\\ \hat \mu(f)> \mu(f)}} \left(
\frac{\prod_{v\in V}\lam_{f(v)}}{\tilde Z_G^H(\lam)}-\frac{\prod_{v\in V}\lam_{f(v)}}{ Z_G^H(\lam)}
\right)\\
&\le 
e^{-\Omega(m^n/n^2)} \label{eqhom1}
\end{align}
where for the first equality we used Lemma~\ref{eqpcount}~\ref{itemf}.
To bound the contribution from elements of $\hom_2(G,H)$,
we use Lemma~\ref{lemH21} to obtain
\begin{align}
\sum_{\substack{f\in \hom_2(G,H):\\ \hat \mu(f)> \mu(f)}} 
\hat \mu(f) - \mu(f)
&\le
\sum_{\substack{f\in \hom_2(G,H):\\ \hat \mu(f)> \mu(f)}} \left( 4^q
\frac{\prod_{v\in V}\lam_{f(v)}}{\tilde Z_G^H(\lam)}-\frac{\prod_{v\in V}\lam_{f(v)}}{ Z_G^H(\lam)}
\right)\\
&\le 
e^{-\Omega(m^n/n^2)} \label{eqhom2}
\end{align}
where again we use
Lemma~\ref{eqpcount}~\ref{itemf} for the first inequality and the fact that elements of 
 $\hom_2(G,H)$ are captured by $\le 4^q$ dominant patterns.
 The result follows by summing the bounds~\eqref{eqhom1}
 and \eqref{eqhom2} to bound the sum in~\eqref{eqrecall}.
\end{proof}

In order to prove Theorem~\ref{mainTV} it remains to
prove Lemmas~\ref{lemH01} and~\ref{lemH21}.
A crucial step toward the proof of these lemmas
is to verify that the Koteck\'y-Preiss condition~\eqref{eqKPcond}
holds for each of the polymer models introduced in Definition~\ref{defnABpoly}.

\begin{lemma}
\label{lemKP}
There exist functions
$f : \cP \to [0,\infty)$ and $g: \cP \to [0,\infty)$ 
such that for each $(A,B)\in \cD_{\lam}(H)$
and all polymers $\gam \in \cP$
we have
\begin{align}
\label{eqKPcond2}
\sum_{\gamma':d(\gamma', \gamma)\leq2}   w_{A,B}(\gam') e^{f(\gam') +g(\gam')}  &\le f(\gam) \,.
\end{align}
In particular the cluster expansion of $\ln \Xi_{A,B}$ converges absolutely.  
\end{lemma}
We will prove Lemma~\ref{lemKP} in 
Section~\ref{secVKP}.
A key consequence of Lemma~\ref{lemKP} is that it provides 
strong tail bounds for the cluster expansion. 
In order to state our tail bound,
we introduce a family of parameters that will play an important role
throughout this paper. 
For a dominant pattern $(A,B)\in\cD_{\lam}(H)$ let 
\begin{align}\label{eqdeltadef}
\delta_{A,B}\bydef  \max\left\{\max_{v\in B^c}\frac{\lam_{N(v)\cap A}}{\lam_A},\max_{u\in A^c} \frac{\lam_{N(u)\cap B}}{\lam_B}   \right\}\, .
\end{align}

In Section~\ref{secVKP} we prove the following.

\begin{lemma}\label{lemexpL11}
For $(A,B)\in \cD_\lam(H)$ and $k\in \N$,
\begin{align}
\sum_{j=k}^\infty |L_{A,B}(j)|= O\left(m^n d^{2(k-1)}\delta_{A,B}^{dk}\right)\, .
\end{align}
\end{lemma}

We note that $\delta_{A,B}<1$ for each dominant pattern $(A,B)$ and so Lemma~\ref{lemexpL11} does indeed provide an effective tail bound.

\begin{lemma}\label{lemdelta1}
$\delta_{A,B}<1$ for each $(A,B)\in \cD_{\lam}(H)$.
\end{lemma}
\begin{proof}
Let $(A,B)\in \cD_{\lam}(H)$.
For any $u\in A^c$ there must exist $w\in B$ such that $u \nsim w$
else $A\cup\{u\}\sim B$ contradicts the fact that $(A,B)$ is a dominant pattern.
In other words, $N(u)\cap B\neq B$, so 
$\lam_{N(u)\cap B}<\lam_B$ for all $u\in A^c$.
Similarly $\lam_{N(v)\cap A}<\lam_A$ for all $v\in B^c$.
\end{proof}

To end this section we show how
 the following, refined version of Theorem~\ref{cormain}
follows from the above results.

 \begin{theorem}
\label{cormain2}
For $m$ an even integer, $k\in \N$ and $(H,\lam)$ a weighted graph, 
we have
\begin{align}
Z_G^H(\lam)=
 \eta_\lam(H)^{\frac{m^n}{2}}
\sum_{(A,B)\in \cD_\lam(H)}
 \exp \left\{
  \sum_{j=1}^{k-1} L_{A,B}(j) + \eps_k
 \right\}\, 
\end{align}
where $\eps_k=O(m^n n^{2(k-1)} \delta_{A,B}^{dk})$.
\end{theorem}

\begin{proof}[Proof of Theorem~\ref{cormain2} assuming Lemmas~\ref{lemH01},~\ref{lemH21}and~\ref{lemexpL11}]
Let $k\in \N$. By Lemma~\ref{lemexpL11} 
\begin{align}\label{eqclusterXi}
\ln \Xi_{A,B} = \sum_{j=1}^{\infty} L_{A,B}(j) =  \sum_{j=1}^{k-1} L_{A,B}(j) + \eps_k
\end{align}
where $\eps_k=O(m^n d^{2(k-1)}\delta_{A,B}^{dk})$.
The result follows from Lemma~\ref{lempolymerapprox} (which holds by our assumption that Lemmas~\ref{lemH01} and~\ref{lemH21} hold).
\end{proof}

In the next section we introduce the entropy tools that play a central role in the proofs of Lemmas~\ref{lemH01} and~\ref{lemKP}.
Pioneered by Kahn and Lawrenz~\cite{kahn1999generalized}, entropy methods have seen a huge range of applications in the study of graph homomorphisms and beyond (see for example \cite{engbers2012h, engbers2012h2, galvin2004weighted, kahn2001entropy, kahn2001range, kahn2018number, peled2017condition, peled2018rigidity}). 
The proof of Lemma~\ref{lemH01} is an adaptation of a delicate entropy argument due to Engbers and Galvin~\cite{engbers2012h}. 
Entropy arguments will also play a central role in the proof of Lemma~\ref{lemKP}.
Roughly speaking, to prove Lemma~\ref{lemKP} we need to balance two quantities:
 $(i)$ the number of polymers of a given size and
 $(ii)$ the total weight (or `cost') of polymers of that size.
 In actuality, rather than fixing the size,
  a more complicated constraint will be used but we defer the details for now.
  A careful combination of entropy tools, the graph container method, isoperimetric and algebraic properties of the torus
  will ultimately allow us to balance
  $(i)$ and $(ii)$ and verify the Koteck\'y-Preiss condition.
 
 Starting with the work of Korshunov and Sapozhenko~\cite{korshunov1983number}, 
 the graph container method has been widely applied in the study of spin models on the torus and integer lattice (see for example \cite{feldheim2019long, galvin2003homomorphisms, galvin2008sampling, galvin2011threshold, galvin2007sampling, galvin2004phase, galvin2015phase, galvin2007torpid, galvin2006slow, kahn2018number,  korshunov1983number, peled2017high, peled2014odd, peled2017condition, peled2018rigidity}). As noted in the introduction, the synthesis of container and entropy methods first appears in the work of Peled and Spinka~\cite{peled2018rigidity} and later appears in the work of Kahn and Park~\cite{kahn2018number} (with our treatment following the latter more closely).

\section{Entropy tools}
\label{secentropy}
In this section we gather some
 tools based on entropy.
In what follows all random variables will be discrete.
Recall that we use $\log$ to denote the base $2$ logarithm
and that the \emph{binary entropy function} is the map
$H:[0,1]\to\R$ given by 
\[
H(p)=-p\log p-(1-p)\log (1-p)\, .
\]
where we interpret $0\log 0$ as $0$.
The \emph{entropy} of a random variable $X$ is
\[
H(X)=\sum_{x}-\P(X=x) \log \P(X=x)\, .
\]
The inequality that makes entropy a useful tool for counting is
\begin{align}\label{equniformbound}
H(X)\le \log |\text{range}(X)|\, ,
\end{align}
and we have equality if and only if $X$ has the uniform distribution.
For two random variables $X$ and $Y$,
the \emph{conditional entropy} of $X$ given $Y$ is
\begin{align}
H(X | Y)=-\sum_y \P(Y=y) \sum_x \P(X=x \mid Y=y) \log \P(X=x \mid Y=y)\, .
\end{align}
We collect a few useful properties of entropy in the form of a lemma.
\begin{lemma}\label{lementropybasics}
\begin{enumerate}
\item $H(X,Y)= H(X) + H(Y|X)$,
\item $H(X_1,\ldots, X_n |Y)\le \sum_i H(X_i | Y)$,
\item if $Z$ is determined by $Y$, then $H(X|Y)\le H(X | Z)$.
\end{enumerate}
\end{lemma}

For a random vector $X=(X_1, \ldots, X_k)$ and 
$S\subset[k]$, set $X_S=(X_i : i\in S)$.
We also use the following version of Shearer's Lemma.

\begin{lemma}\label{lemShearer}
If $X=(X_1, \ldots, X_k)$ is a random vector
and $\varphi : 2^{[k]}\to \R_{\ge0}$ satisfies
\begin{align}
\sum_{S \ni i} \varphi(S)\ge1\, \text{ for each } i\in[k]\, ,
\end{align}
then for any partial order $\prec$ on $[k]$,
\begin{align}
H(X)\le\sum_{S\subset [k]}  \varphi(S) H(X_S | (X_i: i\prec S))\, ,
\end{align}
where $ i\prec S$ means  
$i\prec s$ for all $s\in S$.
\end{lemma}

Throughout the paper we will also use the following
standard estimate on binomial coefficients in terms of 
the binary entropy function.
\begin{align}\label{eq:entropy}
\binom{x}{\le \beta x}\bydef \sum_{i=0}^{\lfloor \beta x \rfloor}\binom{x}{i}\le2^{H(\beta)x}\text{ \, \, where }x\in \N,\,  \beta\in[0,1/2]\, .
\end{align}

We refer the reader to~\cite{cover2012elements} for more background on information theory.

\section{Isoperimetry in the torus}\label{seciso}
In this short section we collect some isoperimetric 
 properties of the torus. These results will be key for bounding the weights of polymers.
 
Our first lemma gives a lower bound on the bipartite expansion factor
of the torus. 
The following proof is an adaptation 
of an argument due to 
Christofides, Ellis and the second author
\cite[Theorem 3]{christofides2013approximate}.

\begin{lemma}\label{lembipexp}
If $X\subset \cE, \cO$ with $|X|=\beta m^n/2$, where $0<\beta\le1$
 then
\begin{align}
|N(X)| \ge |X|\left(1+4\sqrt{2}\cdot \frac{(1-\beta)}{m^2\sqrt{n}}\right)\, . 
\end{align}
\end{lemma}

\begin{proof}
We proceed by induction on $n$
and note that the case $n=1$ is simple to verify
(note that if $m=2$ then $\beta>0$ implies $\beta=1$).
We let $\cE_n, \cO_n$ denote
the even/odd side of the bipartition 
of the graph $\Z_m^n$
(we make the dependence on $n$ explicit to clarify the induction step).
For a subset $Y\subset \Z_m^n$ and $i\in\Z_m$ we let
\begin{align}
Y_i\bydef  \left\{(y_1, \ldots, y_{n-1}) \in \Z_m^{n-1}: (y_1, \ldots, y_{n-1}, i) \in Y \right\}\, .
\end{align}
We note that since $X\subset \cO_n$ or $X \subset \cE_n$,
we have $X_i\subset \cO_{n-1}$ or $X\subset \cE_{n-1}$ for each $i$.
Letting $|X_i|=\beta_i m^{n-1}/2$, 
we may appeal to the inductive hypothesis to conclude that
\begin{align}\label{eqindhyp}
|N(X_i)| \ge |X_i|\left(1+4\sqrt{2}\cdot\frac{(1-\beta_i)}{m^2\sqrt{n-1}}\right)
\end{align}
for each $i \in \Z_m$.
Now, we have
\begin{align}\label{eqnbdslice}
|N(X)| &= \sum_{i \in \Z_m} |N(X)_i| 
= \sum_{i \in \Z_m} |N(X_i)\cup X_{i+1} \cup X_{i-1}|\, . 
\end{align}
Let $\delta=4\sqrt{2}\beta(1-\beta)/(m\sqrt{n})$ and 
suppose that $|\beta_i- \beta_{i+1}|\ge\delta$ for some $i\in \Z_m$. 
By the above inequality and the 
fact that $|N(X_i)\cup X_{i+1} \cup X_{i-1}|\ge \max \{|X_{i-1}|, |X_{i}|, |X_{i+1}|\}$
 we would then have
\begin{align}
|N(X)|\ge \delta m^{n-1}/2 + \sum_{i \in \Z_m} |X_i|
=  \delta m^{n-1}/2 +|X|
= |X|\left(1+4\sqrt{2}\cdot\frac{(1-\beta)}{m^2\sqrt{n}}\right)\, .
\end{align}
We may therefore assume that 
$|\beta_i- \beta_{i+1}|<\delta$
for each $i\in \Z_m$. 
Since $\beta=\sum_i \beta_i/m$ we have 
$\beta_i\in (\beta-m\delta/4, \beta + m\delta/4)$
for each $i$.
By \eqref{eqindhyp} and \eqref{eqnbdslice} we have
\begin{align}
|N(X)|\ge |X| +\frac{m^{n-1}}{2}\sum_{i\in \Z_m} 4\sqrt{2}\cdot \frac{\beta_i (1-\beta_i)}{m^2 \sqrt{n-1}}\, .
\end{align}
Letting $g(\beta)=\beta(1-\beta)$, 
we would therefore be done if we could show that
\begin{align}
\frac{1}{m}\sum_{i\in \Z_m}  g(\beta_i) \ge g(\beta) \sqrt{1-1/n} \, .
\end{align}
Since $g$ is concave, 
the minimum of the sum 
$\sum_{i\in \Z_m}  g(\beta_i)$
under the constraints 
$\beta=\sum_i \beta_i/m$ and
$\beta_i\in [\beta-m\delta/4, \beta + m\delta/4]$
for all $i$ is achieved when half of the $\beta_i$
take value $\beta-m\delta/4$ and the other half take value
$\beta+m\delta/4$.
Therefore
\begin{align}
\frac{1}{m g(\beta)}\sum_{i\in \Z_m}  g(\beta_i) 
&\ge \frac{g(\beta-m\delta/4)+g(\beta+m\delta/4)}{2 g(\beta)}\\
&=1-\frac{m^2 \delta^2}{16\beta(1-\beta)}\\
&= 1-\frac{2\beta(1-\beta)}{ n}\\
&\ge 1-1/(2n)\\
&\geq \sqrt{1-1/n}
\end{align}
where we have used the inequality $\sqrt{1-x}\le 1-x/2$ for $x\in[0,1]$.
\end{proof}

We remark that Riordan \cite{riordan1998ordering} 
exhibited
an ordering of the vertices of $\Z_m^n$
whose initial segments
have minimal vertex boundary for their size.
We suspect that initial segments of 
Riordan's ordering restricted to 
$\cE$ give the correct lower bound for Lemma~\ref{lembipexp}.
However, the approximate statement of
Lemma~\ref{lembipexp} 
will suffice for our purposes.

We will also make use of the fact that sets
of polynomial size have much more substantial expansion. 
The following is a consequence of \cite[Theorem 1.1]{riordan1998ordering}. 

\begin{lemma}
\label{thmviso}
There exists $C>0$ such that 
if $X\subset V$, $|X|= O(n^4)$ then $|N(X)|\ge C |X|$, 
and if $|X|\le n/2$, then $|N(X)|\ge n|X|-2 |X| (|X|-1)$.
\end{lemma}

\section{Approximate polymers}\label{secapprox}
In this section we introduce some tools belonging to
 the graph container method
introduced by Korshunov and Sapozhenko~\cite{korshunov1983number, sapozhenko1987number}.
We adapt the excellent exposition of 
Galvin~\cite{galvin2003homomorphisms},
giving full details for the convenience of the reader
as our setting is sufficiently different.

Recall from the introduction to Section~\ref{secentropy}
that our goal is to balance the number of polymers
satisfying a certain constraint with the total weight of such polymers.

Sapozhenko's approximation tools will allow us to 
replace polymers by `approximate polymers'
which are far less numerous, 
but retain enough information so that 
we can effectively bound the total weight of polymers
with a given approximation.

We write $2^V$ to denote the power set of $V$.
For a subset $Y\subset V$ and $v\in V$,
we let $d_Y(v)\bydef |N(v)\cap Y|$.
Recall that $G$ is a $d$-regular graph 
where $d=(1+\ind_{m>2})n$.

Let $\psi=\psi(d)>0$.
A \emph{$\psi$-approximating pair} for a subset $X\subset V$
is a pair $(F,S) \in 2^V \times 2^V$ satisfying
\begin{align}\label{eqapprox1}
F\subset N(X),\, \, S\supseteq X\, ,
\end{align}
\begin{align}\label{eqapprox2}
d_{V\bs F}(u)\le \psi \hspace{0.5cm}
\forall u\in S
\end{align}
and
\begin{align}\label{eqapprox3}
d_S(v)\le \psi \hspace{0.5cm} 
\forall v\in V\bs F\, .
\end{align}

We think of $\psi$ as small compared to $d$
and so conditions \eqref{eqapprox2} and \eqref{eqapprox3}
say that the graph
between $S$ and $V\bs F$ is sparse.
On the other hand the graph between $X$ and $V\bs N(X)$ is empty.
We therefore view $S$ as an approximation of $X$ and $F$ as an 
approximation of $N(X)$.

For any $g\in \N$, let 
\begin{align}\label{eqcHdef}
\cH(g)=\{X\subset V: X \text{\,  is $G^2$-connected, } |N(X)|=g\}\, .
\end{align}
Our goal in this section is to prove the following.
\begin{lemma}\label{corcontainers}
Let $\psi\le d/2$.
There is a family $\cU=\cU(g)\subset 2^V\times 2^V$ with
\begin{align}
|\cU|\le 2^{O(g (\log d) / \psi+d)}
\end{align}
such that every $X\in\cH(g)$ has a $\psi$-approximating pair in $\cU$.
\end{lemma}
The key point here is that the
 bound on the size of $\cU(g)$ 
 is much smaller than the size of $\cH(g)$.

We say that a set $Y\subset V$ \emph{covers}
a set $X\subset V$ if each $x\in X$
has a neighbour in $Y$,
and each $y\in Y$
has a neighbour in $X$
(note that this relation is symmetric despite the asymmetric terminology).
 
\begin{lemma}\label{lemcover}
Let $X\subset V$,
then there exists a set $Y\subset V$
such that $Y$ covers $X$ and
$|Y| \le 4 |N(X)|/d$.
\end{lemma}
\begin{proof}
Let $\ell=2^k$ be the largest power of $2$
less than or equal to $n$.
We label the first $\ell$ coordinates of $V=\{0,\ldots, m-1\}^n$
with elements of the group $\Z_2^k$, that is we fix
a bijection $\phi: [\ell] \to \Z_2^k$.
We then define the function
$\Phi: V \to  \Z_2^k$ given by
\[
\Phi(x)=\sum_{i\in[\ell]}x_i\cdot \phi(i)\, .
\]
Observe that for each $i\in [\ell]$,
the set $N(X)\cap\Phi^{-1}(\phi(i))$ covers $X$
and one such set must have size at most $|N(X)|/\ell \le 2 |N(X)|/n$.
The result follows by recalling that $d=(1+\ind_{m>2})n$.
\end{proof}

The function $\Phi$
in the above proof was inspired by a construction 
of Linial, Meshulam and Tarsi \cite{linial1988matroidal}
used to bound the chromatic number of the graph on vertex set 
$\{0,1\}^n$ where $x\sim y$ if and only if $x$ and $y$ 
are at Hamming distance $2$.
A similar construction will reappear in the next section in a crucial role
(see Lemma~\ref{lemcoverprop}). 

For the remainder of this section, 
it will be useful to generalise our notion of a
$G^2$-connected set:
for $k\in \N$, say that a subset $S\subset V(G)$ is 
\emph{$G^k$-connected} if the graph $G^k[S]$ is connected
(here $G^k$ denotes the $k$th power of the graph $G$).

The following well-known lemma is a useful enumeration tool 
when considering $G^k$-connected sets. 

\begin{lemma}[\cite{borgs2013left}, Lemma 2.1]\label{lemConCount}
In a graph of maximum degree at most $\Delta$, the number of connected, induced subgraphs of order $t$ containing a fixed vertex $v$ is at most $(e\Delta)^{t-1}$.
\end{lemma}

\begin{lemma}\label{lemcoverapprox}
For all $g\in \N$, there exists a family $\cV= \cV(g) \subset 2^V$ with
\[
|\cV|\leq 2^{O(g (\log d)/d+d)}\, 
\]
such that for any $X\in \cH(g)$, 
$\cV$ contains a set which covers $X$. 
\end{lemma}
\begin{proof}
Given $X\in \cH(g)$,
Lemma~\ref{lemcover} gives us a set $F\subset N(X)$
which covers $X$ where $|F|\le 4g/d$. 
Moreover, since every vertex of $F$
is adjacent to a vertex of $X$ and vice versa,
we see that $F$ must be $G^4$-connected. 
We may therefore take $\cV$ to be the set of
all $G^4$-connected subsets of $V$ of size at most $4g/d$.
Note that $G^4$ has maximum degree $\le d^4$
and so the number of $G^4$-connected sets of some size $t$,
through a given vertex in $G$ is, by Lemma~\ref{lemConCount},
at most $(ed^4)^{t-1}$.
The number of $G^4$-connected subsets of $V$
 of size at most $4g/d$ is therefore at most
\[
m^d\sum_{t=1}^{4g/d}(ed^4)^{t-1}=2^{O(g (\log d)/d+d)}\, . \qedhere
\] 
\end{proof}

In order to prove Lemma~\ref{corcontainers},
it will suffice to prove the following. 

\begin{lemma}\label{lemapproxpair}
Let $\psi\le d/2$.
For each $Y\subset V$
there is a family $\cW=\cW(Y,g) \subset 2^V\times 2^V$ with
\[
|\cW|\le2^{O(g (\log d)/ \psi)}\, 
\]
such that any $X\subset V$ covered by $Y$
with $|N(X)|=g$
has a $\psi$-approximating pair in $\cW$.
\end{lemma}

Indeed, letting
\[
\cU(g)=\bigcup_{Y\in \cV(g)}\cW(Y,g)
\]
where  $\cV(g)$ and $\cW(Y,g)$
are the families from Lemmas~\ref{lemcoverapprox} and \ref{lemapproxpair},
it is clear that $\cU(g)$ has the properties claimed in Lemma~\ref{corcontainers}. 
It remains to prove Lemma~\ref{lemapproxpair}.

We give an algorithm which, 
given a set $X\subset V$ and $Y\subset N(X)$,
produces a $\psi$-approximating pair for $X$.
\medskip

\noindent \underline{\textbf{Algorithm $1$}}:
\medskip

\noindent \textbf{Input:} $(X,Y)\in 2^V \times 2^V$ such that $Y\subset N(X)$.

\medskip

\noindent {\textbf{Step $1$}:} Set $F'=Y$.
If $\{u \in X :  d_{N(X) \bs
F'}(u)\geq \psi \}
 \neq \emptyset$, pick $u$ in this
set and update $F'$ by $F' \longleftarrow F' \cup N(u)$. Repeat
until $\{u \in X :  d_{N(X) \bs F'}(u)\geq \psi \}
= \emptyset$. 

\medskip

\noindent {\textbf{Step $2$}: }Set $S'= \{u \in
V : d_{V\bs F'}(u) \le \psi\}$.
If $\{w \in V \bs N(X) :
d_{S'}(w) \geq \psi \} \neq \emptyset$, pick
 $w$ in this set and update $S'$ by $S'
\longleftarrow S' \bs N(w)$. Repeat until $\{w \in
V \bs N(X) :  d_{S'}(w) \geq \psi \} =
\emptyset$. 

\medskip

\noindent \textbf{Output:}
 $(S,F)=(S', F'\cup \{w \in V
: d_{S'}(w) \geq \psi \})$.

\begin{lemma}\label{lemalg1}
The output of Algorithm 1 is a $\psi$-approximating pair for $X$.
\end{lemma}
\begin{proof}
To see that $F\subset N(X)$
first note that at termination, $F' \subset N(X)$
since $Y\subseteq N(X)$, 
and the vertices added to $F'$
 in Step $1$ are all in $N(X)$.
  Moreover
 $\{w \in V: d_{S'}(w) \geq \psi \}\subset N(X)$
else Step $2$ would not have terminated.
To see that $S\supseteq X$, 
note that initially $S'\supseteq X$
else Step $1$ would not have terminated
and Step $2$ deletes
from $S'$ only neighbours of $V \bs N(X)$.

To verify~\eqref{eqapprox2},
note that at the start of Step $2$,
$d_{V\bs F'}(u)\le\psi$ for all $u\in S'$ by definition 
and $F'\subset F$ and $S\subset S'$.
Condition~\eqref{eqapprox3} is immediate from the definition of $F$.
\end{proof}

We now prove Lemma~\ref{lemapproxpair} 
(and hence also Lemma~\ref{corcontainers}).

\begin{proof}[Proof of Lemma~\ref{lemapproxpair}]
Let $Y\subset V$. 
We will show that as 
the input $(X,Y)$ 
to Algorithm 1
runs over $X$
such that $Y$ covers $X$ 
and $|N(X)|=g$,
we get at most
 $2^{O(g (\log d) / \psi)}$ distinct outputs.
 We are then done by taking $\cW$ to be the set of these outputs.
 
Suppose then that $Y$ covers $X$ where $|N(X)|=g$.
Since $Y$ covers $X$ we have
$X\subset N(Y)$ and $Y\subset N(X)$.
In particular $|Y|\le g$.
Consider Algorithm 1 with input $(X,Y)$.
The output
of Step $1$ of the algorithm is determined by the set of $u$'s
whose neighbourhoods are added to $F'$ in Step $1$, and the set of
$w$'s whose neighbourhoods are removed from $S'$ in Step $2$.

Each iteration in Step $1$ reduces the size of $N(X)\bs F'$
by at least $\psi$, so there are at most $g /\psi$ iterations. The $u$'s
in Step $1$ are all drawn from $X$ and hence $N(Y)$,
 a set of
size at most $dg$ (since $|Y|\le g$).
The number of possible outputs for Step $1$ is
therefore at most
\begin{align}
\binom{d g}{\le g/\psi}
  = 2^{O(g (\log d) / \psi)}.
\end{align}

We perform a similar analysis on Step $2$. 
At the start of Step $2$, $d_{N(X)}(u)\ge d-\psi$ 
for each $u\in S'$
so that
\[
|S'|(d-\psi)\le \sum_{u\in S'}d_{N(X)}(u)=\sum_{u\in N(X)}d_{S'}(u)\le dg
\]
and so initially $|S'|\le 2g$ (since $\psi\le d/2$).
 Each iteration in Step $2$ reduces the size of $S'$
 by at least $\psi$, 
 so there are at most $2g/\psi$ iterations. 
 Each iteration draws a vertex $w$ from $N(S')$, 
 a set whose elements are all at distance at most $4$
 from $Y$ 
 (indeed, $S'\subset N(F')$ by definition, $F'\subset N(X)$ and $X\subset N(Y)$).
 We thus have $|N(S')|\le d^4 g$ and so 
 as in Step $1$, 
 the total number of outputs for Step $2$ is $2^{O(g (\log d) / \psi)}.$
\end{proof}

\section{Grouping polymers and decomposing the torus}\label{secgroup}
In this section we build on the results of the previous one
in order to enumerate polymers satisfying certain constraints. 
One of these constraints corresponds to a measure of `roughness' of the polymer 
which we refer to as the \emph{neighbourhood distribution} (see~\eqref{eqnbdistdef} below).
Again, as outlined at the beginning of Section~\ref{secentropy},
our current goal is to verify the Koteck\'y-Preiss
 condition~\eqref{eqKPcond} (for suitably chosen functions $f$ and $g$) 
 for each of our polymer models
 (Lemma~\ref{lemKP}).
A key component in our enumeration steps will 
be the entropy tools introduced in Section~\ref{secentropy}.
We will also rely on an
algebraically constructed vertex partition of $\Z_m^n$
which enjoys good `covering properties' (see Lemma~\ref{lemcoverprop}).
We note that this is one of the key places 
where we take advantage of 
the specific structure of the torus. 

A quirk of our algebraic construction is that 
the results of this section go through more easily 
if the dimension of the torus is assumed to be a power of $2$.
In the general case, we must decompose the torus
into a collection of subtori each of dimension $n'$
where $n'$ is divisible by a large power of $2$. 
The reason for this will become clearer 
after our covering property is stated formally.
First we define our decomposition of $G$ into subtori. 

Choose
$k\in\N$ such that
\begin{align}\label{eqelldef0}
\ell\bydef 2^k=\Theta(\sqrt{d}\log^4 d)\, .
\end{align}
The precise value of $\ell$ will become relevant in the next section 
(see the proof of Lemma~\ref{lemgweight}).
Let $0\le p \le \ell-1$ be such that $p\equiv n \pmod \ell$.
We split the torus $G=\Z_m^n$ into $m^p$ subtori in the following way. 
For $x\in\{0,\ldots,m-1\}^p$ let
\[
T_x\bydef 
G\left[
\{v\in V:  \{v_{n-p+1},\ldots,v_n\}=x \}
\right]
\, .
\]
We note that $T_x\cong \Z_m^{n-p}$ for each $x\in\{0,\ldots,m-1\}^p$.
We let $G'$ denote the graph that is the disjoint union of 
all the $T_x$ and note that $V(G')=V(G)=V$.
For $v\in V$, let $N'(v)$ denote the neighbourhood of $v$ in 
$G'$ and let $d'(v)$ denote the degree of $v$
in $G'$.
We allow for similar modifications of degree notation,
for example if $X\subset V$ we let $d'_X(v)$ denote
the size of the set $N'(v)\cap X$.
We note that  $T_x$ is a $d'$-regular graph where
\begin{align}\label{eqdprime}
d-2\ell \leq d'=(1+\ind_{m>2})(n-p)\leq d\, .
\end{align}

We may now introduce the covering property alluded to above.

\begin{lemma}\label{lemcoverprop}
There exists a vertex partition $V=V_1\cup\ldots\cup V_\ell$
with the following property. 
For each $i\in [\ell]$ and $v\in V$,
\[
d'_{V_i}(v)= d'/ \ell\, .
\]
\end{lemma}
\begin{proof}
Fix a bijection
\begin{align}
\phi: [\ell]\to \Z_2^k\, .
\end{align}
For $t\in\N$, let
 $\overline t$ denote the element of $[\ell]$ 
 for which $t\equiv \overline t\pmod{\ell}$. 
We consider the function
$\Phi: V \to  \Z_2^k$ given by
\begin{align}
\Phi(v)=\sum_{t\in[n]}v_t\cdot \phi\left(\overline t\right)\, .
\end{align}
Let $V_i$ denote the set of $v\in V$ such that $\Phi(v)=\phi(i)$ and
note that $V_1\cup\ldots\cup V_\ell$ forms a partition of $V$.
For $t\in[n]$, let $e_t=(0,\ldots, 1, \ldots, 0)\in V$ 
be the standard basis vector
in direction $t$. 
For $v\in V$ we have
\[
N'(v)=\{v\pm e_1, \ldots, v\pm e_{n-p}\}\, 
\]
(where the addition is componentwise modulo $m$). 
Since for $t\in[n]$
\[
\Phi(v \pm e_t)= \Phi(v) + \phi(\overline t)\, ,
\]
we have for $s\in \{1, \ldots, (n-p)/\ell\}$
\[
\left \{ \Phi(v \pm e_t) : (s-1)\ell+1 \le t \le s \ell \right\}= \Z_2^k\, .
\]
For each $i\in [\ell]$, we therefore have 
$(n-p)/\ell$ distinct values of $t\in [n-p]$ 
for which $\Phi(v \pm e_t)=\phi(i)$
(or equivalently $v \pm e_t\in V_i$).
It follows that for each $i\in [\ell]$,
 $d'_{V_i}(v)=(1+\ind_{m>2})(n-p)/\ell =d' / \ell$.
\end{proof}

The fact each vertex $v\in V$
has identical degree (with respect to $G'$) into each set $V_i$ will be crucial.
By contrast, 
if we were to consider degrees with respect to $G$
in the above lemma, 
we would find that $d_{V_i}(v)$ takes one of two values.
This seemingly minor discrepancy introduces 
an imperfect covering system in an
application of Shearer's lemma
in the proof of the main enumeration result
of this section
(Lemma~\ref{lemrecon}).
The resulting error term produces a bound that is too weak for our purposes.

Before stating our polymer enumeration result
we require some preliminaries. 
We define
\[
\cN(g)=\{X\subset V: |N(X)|=g\}\, ,
\]
and note that this differs from 
the definition of $\cH(g)$ (see \eqref{eqcHdef})
in that the elements of $\cN(g)$ need not be $G^2$-connected.

For the remainder of this section we fix 
\begin{align}\label{eq:psibd}
0<\psi< d/3\, .
\end{align}

For a pair $(F,S)\in 2^V \times 2^V$,
we let 
\[
\cN(g,F,S)\bydef  
\left\{
X\in \cN(g):
\text{ $(F,S)$ is a $\psi$-approximating pair for $X$ } 
\right\}\, .
\]

For a set $X\subset V$,
we define the \emph{neighbourhood distribution} of $X$ 
to be the \emph{multiset}
\begin{align}\label{eqnbdistdef}
D_X\bydef \{d'_X(v): v\in N(X)\}\, .
\end{align}
We emphasise that
in the definition of $D_X$, 
$v$ ranges over the set $N(X)$
rather than $N'(X)$ and so some 
elements of the multiset may be $0$.
It turns out that the weight of a subset
$X\subset V$
in our polymer models (as defined in~\eqref{eqwdef})
is closely related to its neighbourhood distribution
and so it will be useful to group sets according to their distribution.

Our goal is to enumerate the number of polymers $\gamma\in\cN(g)$
with a fixed $\psi$-approximating pair $(F,S)$ 
and neighbourhood distribution $D$ (Lemma~\eqref{lemrecon} below).
Fixing the degree distribution is a rather delicate constraint
and the covering lemma (Lemma~\ref{lemcoverprop})
is crucial for dealing with this.
Roughly speaking, 
our strategy is to specify a set $X\in \cN(G)$
with a given neighbourhood distribution
as follows:
first specify $N(X)\cap V_i$ (with $V_i$ as in Lemma~\ref{lemcoverprop})
for some $i$
and then specify $X$ by selecting 
subsets of $N'(u)$ for each $u\in N(X)\cap V_i$.
By Lemma~\ref{lemcoverprop}, each vertex will be specified 
$d'/\ell$ times which we account for with an 
application of Shearer's Lemma.
We choose the subsets of $N'(u)$ to have 
sizes compatible with the fixed neighbourhood distribution.
We do this for each $i$ and then average. 
We gain a critical entropy saving from the fact that 
on average the sets $N(X)\cap V_i$ are
smaller than the full neighbourhood $N(X)$ by a factor of $\ell$.

Let $D$ be a multiset of size $g$ and 
let $(F,S)\in 2^V \times 2^V$. We define

\begin{align}
\cN(D, F, S)\bydef 
\left\{
X\in \cN(g, F, S): D_X= D 
\right\}\, .
\end{align}

Let us also define
\begin{align}\label{eqxdef}
x=x_{F,S,D}: F \to \N
\end{align}
to be a function that 
maximises the quantity 
\[
\prod_{v\in F}\binom{d'_S(v)}{x(v)}
\]
subject to the constraint
$\{x(v): v\in F\}\subset D$
(multiset inclusion). 
Note that we use the convention $\binom{0}{0}=1$.

In the following we will make use of the
classical theorem of Hardy and Ramanujan 
\cite{hardy1918asymptotic} on integer partitions. 
An \emph{integer partition} of a natural number $k$
is a multiset of natural numbers whose elements sum to $k$.
We let $p(k)$ denote the number of integer partitions of $k$. 

\begin{theorem}[Hardy and Ramanujan~\cite{hardy1918asymptotic}]
\label{thmHR}
\[
p(k)\sim \frac{1}{4k\sqrt{3}}e^{\pi \sqrt{2k/3}} \text{\, \,    as $k\to\infty$}\, .
\]
\end{theorem}
We will not need the precise asymptotics of 
the above theorem,
the bound $p(k)=e^{O(\sqrt{k})}$ will suffice for our purposes.
The paper of Hardy and Ramanujan~\cite{hardy1918asymptotic} in fact begins 
with an elementary proof of this weaker bound.

We may now state our main enumeration lemma.

\begin{lemma}\label{lemrecon}
For
$(F,S)\in 2^V \times 2^V$ and
multiset $D$ of size $g$
we have
\begin{align}
|\cN(D, F, S)|\le 2^{O(g (\log d)/\ell+\sqrt{dg\ell})+(g- |F|)\psi/d'}
\prod_{v\in F}
\binom{d'_S(v)}{x(v)}^{1/d'}\, ,
\end{align}
where $x=x_{F,S,D}$ is defined in~\eqref{eqxdef}.
\end{lemma}
\begin{proof}
Let $V=V_1\cup\ldots\cup V_\ell$
be the partition of $V$ from Lemma~\ref{lemcoverprop}.
With this partition in mind, 
we refine our view of the set $\cN\bydef \cN(D, F, S)$ in the following way.
Let us fix an ordered $\ell$-tuple $\Pi=(\Pi_1,\ldots, \Pi_\ell)$,
where each $\Pi_i$ is a multiset of elements in $[d']$
and $D=\bigcup_i \Pi_i$ (multiset union).
For each $i\in [\ell]$, we let $\cN(\Pi_i)$ denote the elements $X\in \cN$
for which
\[
\Pi_i=\{d'_X(v): v\in V_i\cap N(X)\}\, ,
\]
and we let 
\[
\cN(\Pi)\bydef \bigcap_{i=1}^\ell \cN(\Pi_i)\, .
\]

We now bound the number of tuples $\Pi$
for which $\cN(\Pi)$ is non-empty.
Observe that by Lemma~\ref{lemcoverprop},
 for $X\in \cN$
we have
\begin{align}
\sum_{v\in V_i}d'_X(v)
= \sum_{v\in X}d'_{V_i}(v)
= 
d' |X|/ \ell\, 
\le
dg/\ell\, .
\end{align}
For $\cN(\Pi)$ to be non-empty,
the non-zero elements of $\Pi_i$ must therefore be an 
integer partition of some integer 
$k_i \le dg/\ell$.
The number of choices for each $\Pi_i$ is therefore at most
$
2^{O(\sqrt{dg/\ell})}
$
by Theorem~\ref{thmHR}
(note that there are $<g$ choices for the multiplicity of $0$ in $\Pi_i$)
and so the number of choices for $\Pi=(\Pi_1,\ldots, \Pi_\ell)$ is at most
\begin{align}
2^{O\left(\sqrt{d g \ell}\right)}\, .
\end{align}
We may therefore bound the size of $\cN$ by 
\begin{align}\label{eqmaxh}
|\cN|\le 2^{O\left(\sqrt{d g \ell}\right)} \max_{\Pi} |\cN(\Pi)|\, .
\end{align}

Given $\Pi=(\Pi_1,\ldots, \Pi_\ell)$, 
we bound the size of $\cN(\Pi)$ via
the following inequality which follows from 
the definition of $\cN(\Pi)$
\begin{align}\label{eqgeom}
|\cN(\Pi)|
\le 
\min_i |\cN(\Pi_i)|
\le
 \left(\prod_{i=1}^\ell |\cN(\Pi_i)|\right)^{1/\ell}\, .
\end{align} 
Fix $i\in[\ell]$ and let $X\in\cN(\Pi_i)$.

Recall that $(F,S)$ is a $\psi$-approximating pair for $X$ and so
by condition~\eqref{eqapprox2} in the definition of an 
approximating pair
\[
(d-\psi)|S|\le
\sum_{v\in S}d_{F}(v)=
\sum_{v\in F}d_{S}(v)\le dg\, .
\]
For the final inequality we used that
$|F|\le |N(X)|=g$ by condition~\eqref{eqapprox1}.
Since $\psi<d/3$ by assumption~\eqref{eq:psibd} we have
\begin{align}\label{eqSbd}
|S|\le 2g
\end{align}
and so
\begin{align}\label{eqNSbd}
|N(S)|\le 2dg\, .
\end{align}
We have $N(X)\subset N(S)$ by~\eqref{eqapprox1}.
There are therefore
at most $\binom{2dg}{|\Pi_i|}$ choices for 
$V_i\cap N(X)$.
Let us fix such a choice and denote it by $N_i$.
Fix a function
\begin{align}\label{eqassign}
\pi: N_i\to \{0,\ldots,d'\}\, ,
\end{align}
such that $\{\pi(v): v\in N_i\}=\Pi_i$.
Crudely there are at most $d^{|N_i|}=d^{|\Pi_i|}$ choices for the function $\pi$.
We now bound the number of choices for $X\in \cN(\Pi_i)$ such that $V_i\cap N(X)=N_i$ and 
\[
d'_X(v)=\pi(v) \text{ for each }v\in N_i\, .
\]
Recall that by Lemma~\ref{lemcoverprop},
for every $v\in V$,
we have 
$d'_{V_i}(v)=d'/\ell$.
Let
\[
S'=\{v\in S: d'_{N_i}(v)= d'/\ell\}
\]
and note that $X\subset S'$ as $N_i= V_i \cap N(X)$.
Let $\Psi$ denote the family of subsets $X\subset S'$ 
such that
\[
d'_X(v)=\pi(v) \text{ for each }v\in N_i.
\]
We will estimate $|\Psi|$ by bounding the entropy $H(X)$,
where $X$ is a uniformly random element of $\Psi$.
For each $v\in N_i$,
 let $X_v=X\cap N'(v)$.
Note that by definition of $S'$,
each vertex $w\in S'$ is contained in $d'/\ell$
of the sets $\{N'(v): v\in N_i\}$.
Let $F_i\bydef V_i\cap F$.
It follows by \eqref{equniformbound} and Shearer's Lemma (Lemma~\ref{lemShearer}), that
\begin{align}
\log |\Psi|= H(X)
&\le
\frac{\ell}{d'}\sum_{v\in N_i}H(X_v)\\
&\le
\frac{\ell}{d'}\sum_{v\in N_i}\log \binom{d'_S(v)}{\pi(v)}\\
&\le 
\frac{\ell}{d'}\left( \sum_{v\in F_i}\log \binom{d'_S(v)}{\pi(v)}+
|N_i \bs F_i| \psi\right)\\
&\le
\frac{\ell}{d'}\left( \sum_{v\in F_i}\log \binom{d'_S(v)}{x_i(v)}+
|N_i \bs F_i| \psi\right)
\end{align}
where $x_i\bydef x_{F_i, S, \Pi_i}$ (see~\eqref{eqxdef}).
For the penultimate inequality we used $d'_S(v)\le \psi$ for each 
$v\in N_i\bs F_i$ by the definition
of an approximating pair.

Putting everything together we have
\begin{align}
|\cN(\Pi_i)|\le
\binom{2dg}{|\Pi_i|}d^{|\Pi_i|}
\left[
2^{(|\Pi_i|- |F_i|)\psi}
\prod_{v\in F_i}\binom{d'_S(v)}{x_i(v)}
\right]^{\ell/d'} \, .
\end{align}
By \eqref{eqgeom} we then have
\begin{align}
|\cN(\Pi)|
&\le
\left[\prod_{i=1}^\ell 
\binom{2dg}{|\Pi_i|}d^{|\Pi_i|}2^{(|\Pi_i|- |F_i|)\psi \ell/d'}
\prod_{v\in F_i}\binom{d'_S(v)}{x_i(v)}^{\ell/d'}\right ]^{1/\ell}\\
&\le
\binom{2dg}{g/\ell}^{1/\ell}d^{g/\ell}2^{(g- |F|)\psi/d'}
\prod_{v\in F}
\binom{d'_S(v)}{x(v)}^{1/d'}\\
&=
2^{O(g (\log d)/\ell)+(g-|F|)\psi/d'}
\prod_{v\in F}
\binom{d'_S(v)}{x(v)}^{1/d'}\, .
\end{align}
The lemma follows from~\eqref{eqmaxh}.

  \end{proof}
We remark that the important factor 
in the above lemma is the product of binomial coefficients. 
It will be useful to keep the following bound in mind:
\begin{align}\label{eqexpH}
\prod_{v\in F}
\binom{d'_S(v)}{x(v)}^{1/d'} \le \prod_{t\in D}\binom{d'}{t}^{1/d'}
\le \exp_2\left\{\sum_{t\in D} H(t/d') \right\}\, .
\end{align}
The right hand side should be compared with the entropy terms
that arise in our bound for the total weight
of polymers with a given neighbourhood distribution 
(Lemma~\ref{lemfweight} to appear in the next section).
  
 One could now bound the number of polymers
  $\gamma\in\cH(g)$ with a given neighbourhood distribution
 simply by summing the bounds of Lemma~\ref{lemrecon}
 over all approximating pairs in $\cU(g)$ (as defined in Lemma~\ref{corcontainers}).
However, it turns out to be too costly to specify the entire neighbourhood
distribution of a polymer and we have to settle for specifying only part of it. 

\begin{defn}\label{defprint}
Let $\gamma\in \cP$ be such that 
$|N(\gam\cap\cE)|\ge |N(\gam\cap\cO)|$,
and suppose $(F,S)$ is a $\psi$-approximating pair for $\gam$.
We call the pair
\begin{align}
(\gamma\cap\cE, S\cap\cO)
\end{align}
a \emph{fingerprint} of the polymer $\gamma$.
We make the same definition with the roles of $\cE, \cO$ reversed. 
\end{defn}
Note that multiple polymers may have the same fingerprint.
Our goal is to bound the number of possible fingerprints 
$R=(\rho, T)$
such that $R$ is the fingerprint of a polymer $\gam$ with $|N(\gam)|=g$
and the neighbourhood distribution of $\rho$ is $D$. 
Call the set of all such fingerprints $\cR(g, D)$.
Given $(F,S)\in 2^V\times 2^V$, let $\cR(g,D,F,S)$
 denote the set of $R\in\cR(g, D)$ such that
 $R$ is the fingerprint for some $\gam\in\cH(g)$
 with $\psi$-approximating pair $(F,S)$.
Finally let 
\begin{align}
\cR_\cE(g,D)=
\{R\in\cR(g, D): R=(\rho,T) \text{ where } \rho\subset\cE\}\, ,
\end{align}
and
\begin{align}
\cR_\cE(g,D,F,S)=
\{R\in\cR(g, D, F, S): R=(\rho,T) \text{ where } \rho\subset\cE\}\, ,
\end{align}
and define 
$\cR_\cO(g,D)$ and $\cR_\cO(g,D,F,S)$
similarly.

\begin{lemma}\label{lemfprint}
For all $g$, $(F,S)\in 2^V \times 2^V$ 
and multisets $D$,

\begin{align}
|\cR_\cE(g,D,F,S)|\le
 2^{O(g (\log d)/\ell+\sqrt{dg\ell})+(|D|-|F\cap \cO|)\psi/d'}
\prod_{v\in F\cap\cO}
\binom{d'_{S\cap\cE}(v)}{x(v)}^{1/d'}\, ,
\end{align}
where $x=x_{ F\cap\cO, S\cap\cE, D}$
(as defined in~\eqref{eqxdef}).
The same bound holds with $\cO, \cE$ swapped.
\end{lemma}
\begin{proof}

Let $R=(\rho, T)\in\cR_\cE(g,D,F,S) $,
so that $(F\cap \cO, S\cap \cE)$ is a $\psi$-approximating pair for 
the set $\rho$. 
Note that by the definition of a fingerprint we also have 
$|D|=|N(\rho)|\le g$.
By Lemma~\ref{lemrecon}, 
there are therefore at most 
\begin{align}
2^{O(g (\log d)/\ell+\sqrt{dg\ell})+(|D|-|F\cap \cO|)\psi/d'}
\prod_{v\in F\cap\cO}
\binom{d'_{S\cap\cE}(v)}{x(v)}^{1/d'}\, ,
\end{align}
choices for $\rho$. 
The final claim of the lemma follows by symmetry. 
\end{proof}

\section{Bounding weights with entropy}\label{seccHg}
In this section we establish a bound on the total weight of polymers in 
the set $\cH(g)$ (as defined in~\eqref{eqcHdef}).
This will be a crucial step toward the proof of Lemma~\ref{lemKP}
(verifying the Koteck\'y-Preiss condition). 
We complete the proof of Lemma~\ref{lemKP} in the next section. 

Recall from Definition~\ref{defpoly} 
that for a polymer $\gam$,
$|N(\gam\cap\cO)|, |N(\gam\cap\cE)|<(1-\alpha)m^n/2$
where $\alpha=\alpha(H,\lam)$ will be specified later (see~\eqref{eqadef}).

\begin{lemma}\label{lemgweight}
There exists $\xi>0$
(depending on $m$ and $(H,\lam)$)
 such that for
$d^4 \le g \le(1-\alpha)m^n/2$, we have
\begin{align}
\sum_{\gam\in\cH(g)}w(\gam)\le e^{-\xi g/(\sqrt{d}\log^2 d )}\, .
\end{align}
\end{lemma}

There are two broad steps in the proof of
Lemma~\ref{lemgweight}. 
In the first step, we bound the total weight of all polymers
with a given \emph{fingerprint}.
In the second step, we sum over different classes of fingerprints 
and appeal to  Lemma~\ref{lemfprint}
(a bound on the number of fingerprints satisfying a certain set of constraints).
To carry out the first step, we require a careful entropy argument
inspired by Kahn and Park \cite{kahn2018number}.

Given a fingerprint $R$, 
let $\cP(R)$ denote the set of all polymers $\gam\in \cP$
with fingerprint $R$. 
Henceforth we let $q$ denote the number of vertices in $H$
 and identify $V(H)$ with $\{1,\ldots, q\}$.

For the next lemma we introduce the following 
 parameter of the weighted graph $(H,\lam)$:
 \begin{align}\label{eqdefr}
 r=r(H,\lam)\bydef \frac{\min_i \lam_i}{\lam_1+\ldots+\lam_q}\, .
 \end{align}
 
 For the remainder of this section we set
 \begin{align}\label{eq:psieq}
 \psi=r^2 d/3
 \end{align}
 and note $0<\psi<d/3$ so that the definition is consistent with
assumption~\eqref{eq:psibd}. We note also that 
\begin{align}\label{eqpsibound}
\psi< r^2d'/2
\end{align}
by~\eqref{eqelldef0} and~\eqref{eqdprime}.

In the following lemma, 
we bound polymer weights $w(\gam)$
via the inclusion $\hom(G,H)\subset \hom(G',H)$
(recall the definition of $G'$ from the start of the previous section).
Although this may seem wasteful,
it will help streamline the comparison with
the polymer count of Lemma~\ref{lemrecon}.
\begin{lemma}
\label{lemfweight}
Suppose $R=(\rho, T)$ is a fingerprint of a polymer in $\cH(g)$.
Then 
\begin{align}
 \sum_{\gam\in\cP(R)}w(\gam)\le 
 \exp_2\left\{-\sum_{t\in D}H\left({t}/{d'}\right)-\frac{r|D_\psi|}{5}+\frac{3qg}{d'}\right\}
 \, .
 \end{align} 
where $D$ is the neighbourhood distribution of $\rho$
and $D_\psi$ is the multiset of elements in $D$ which are $\le \psi$.
\end{lemma}

For the proof the following notation will be useful.
For $c\subset V(H)$, 
let 
\[
n(c)\bydef |\{v\in V(H): v\sim c\}|\, .
\]
For a graph $F$ let 
\begin{align}\label{eqetanon}
\eta(F)\bydef \max\{|A||B|: A,B\subset V(F), A\sim B\}\, ,
\end{align}
that is $\eta(F)=\eta_{\dot\iota}(F)$ (as defined in~\eqref{eqetadef}) where $\dot\iota: V(F) \to \R$
takes only the value $1$ (we say that $F$ is unweighted in this case).

\begin{proof}[Proof of Lemma~\ref{lemfweight}]
Suppose $\gamma\in \cP(R)$.
Recall from~\eqref{eqwdef} that 
\begin{align}
 w(\gam)=\frac{\sum_{f\in\chi_{A,B}(\gam)} \prod_{v\in V}\lam_{f(v)}}
 {\eta_\lam(H)^{m^n/2}}\, ,
\end{align}
where $\chi_{A,B}(\gam)$ is the set of colourings $f\in \hom(G,H)$ such that 
$f$ disagrees with $(A,B)$ at each $v\in \gam$ and agrees at each $v\in V\bs \gam$.
Let $\chi(\rho,T)$ denote the set of colourings $f$ such that 
$f$ disagrees with $(A,B)$ at each $v\in \rho$ and agrees at each $v\in V\bs(\rho\cup T)$.
Note that we do not specify whether $f$ agrees or disagrees with $(A,B)$
on the vertices of $T$.
Suppose without loss of generality that $\rho\subset \cE$. 
Since $\gam\cap\cE=\rho$ and $\gam\cap\cO\subset T$ 
for each $\gam\in \cP(R)$ we have
\begin{align}\label{eqchibd}
\sum_{\gam\in\cP(R)}w(\gam)
\le
\frac{\sum_{f\in\chi(\rho, T)} \prod_{v\in V}\lam_{f(v)}}
 {\eta_\lam(H)^{m^n/2}}\, .
\end{align}

Our strategy is to use a standard
`blowup trick' that allows us
bound the right hand side of~\eqref{eqchibd} 
by counting
homomorphisms into a blowup of $H$.
Having reduced to a counting problem,
we may bring entropy tools to bear. 
However, the blowup trick only works
if the weight function $\lam$ takes
only \emph{rational} values. 
We deal with general $\lam$
via a limiting argument.

\begin{claim}\label{clhatchi}
If $\lam: V(H)\to \R_{>0}$ takes only rational values, then
\begin{align}
\frac{\sum_{f\in\chi(\rho, T)} \prod_{v\in V}\lam_{f(v)}}
 {\eta_\lam(H)^{m^n/2}}
&\le 
\exp_2\left\{-\sum_{t\in D}H\left({t}/{d'}\right)-\frac{r|D_\psi|}{5}+\frac{3qg}{d'}\right\}\, .
\end{align}
\end{claim}

\begin{proof}[Proof of Claim~\ref{clhatchi}]
First we introduce the blowup trick then
we introduce the entropy arguments. 
Since $\lam$ takes only rational values, 
letting $V(H)=[q]$ we may write,
\[
\lam(i)=\frac{a_i}{b}
\]
for each $i\in[q]$, where $a_i , b$ are positive integers.
For the sake of uniqueness, 
we assume $b$ is chosen as small as possible.

Define the \emph{blowup}
 $H[\lam]= H(a_1,\ldots, a_q)$, 
to be
the graph on vertex set 
$W_1\cup\ldots\cup W_q$
where the $W_i$ are disjoint, 
$|W_i|=a_i$ for all $i$ and
$x\sim y$ if and only if
$x\in W_i$, $y\in W_j$
for some pair $\{i,j\}\in E(H)$.
For a vertex $v\in V(H[\lam])$,
we let $\overline {v}$ denote the unique $i\in [q]$
for which $v\in W_i$.
We extend this notation to subsets 
$Y\subset H[\lam]$,
letting 
$\overline Y\bydef  \{\overline y : y\in Y\}$.

We let $A[\lam]=\bigcup_{i\in A}W_i$ and
define $B[\lam]$ similarly.
Let $\chi(\rho,T,\lam)$ denote the set of $f\in \hom(G,H[\lam])$ such that 
$f$ disagrees with $(A[\lam], B[\lam])$ at each $v\in \rho$ 
and agrees at each $v\in V\bs(\rho\cup T)$. We have
\begin{align}
|\chi(\rho, T, \lam) | = 
\sum_{f\in\chi(\rho, T)} \prod_{v\in V}a_{f(v)}\, 
\end{align}
and also
(recalling the definition~\eqref{eqetanon}) 
\begin{align}\label{eqetatran}
\eta(H[\lam])=b^2 \eta_\lam(H)\, ,
\end{align}
so that
\begin{align}\label{eqhatchi}
\frac{ |\chi(\rho, T, \lam) |}{ \eta(H[\lam])^{m^n/2}}= 
\frac{\sum_{f\in\chi(\rho, T)} \prod_{v\in V}\lam_{f(v)}}
 {\eta_\lam(H)^{m^n/2}}\, .
\end{align}
We have therefore related the quantity we want to bound
to the size of the set $\chi(\rho, T, \lam)$.
We bound $|\chi(\rho, T, \lam)|$ using an entropy argument.

Suppose that $f\in \chi(\rho, T, \lam)$ is chosen uniformly at random.
For ease of notation we write $N_u$ instead of $N'(u)$ for a vertex $u\in V$ in the following.
Recall that for a colouring $f$ and set $X\subset V$,
we write $f_X$ for the restriction of $f$ to $X$. 

By Lemma~\ref{lementropybasics}
we may decompose the entropy of $f$ as
\begin{align}\label{eqentropydecomp}
H(f)= H(f_\cE) + H(f_\cO | f_\cE)\, .
\end{align}
We bound these terms separately.
Let $U=N'(\rho)\cup T\subset \cO$.
With an application of Shearer's lemma (Lemma~\ref{lemShearer}) in mind
we define the function $\varphi: 2^{\cE}\to \R_{\ge0}$ by
\begin{align}
\varphi(S)=
\begin{cases} 
      1/d' & \text{if }S=N_u \text{ for some } u\in U,\\
      1-d'_U(u)/d' & \text{if } S=\{u\} \text{ for some } u\in \cE,\\
      0 & \text{otherwise.} 
   \end{cases}
\end{align}
Observe that $\sum_{S \ni u} \varphi(S)=1\, \text{ for each } u\in\cE$ and so
we may apply 
Lemma~\ref{lemShearer} to obtain
\begin{align}\label{eqHfE}
H(f_\cE) \le \sum_{u\in U} \frac{1}{d'} H(f_{N_u}) 
+ \sum_{u\in\cE}\left(1-\frac{d'_U(u)}{d'}\right)H(f_u)\, .
\end{align}
Turning to the second term of~\eqref{eqentropydecomp},
we have by Lemma~\ref{lementropybasics},
\begin{align}\label{eqHfO}
H(f_\cO | f_\cE)\le \sum_{u\in U} H(f_u | \overline{f(N_u)})
+ \sum_{u\in\cO\bs U}H(f_u) \, . 
\end{align}
For a vertex $u\in V$, let
\begin{align}
I(u)=\frac{1}{d'}H(f_{N_u})+H(f_u | \overline{f(N_u)})
=\frac{1}{d'}\left[H(f_{N_u}|\overline{ f(N_u) })+H(\overline{ f(N_u) })\right]+H(f_u | \overline{ f(N_u) })\, .
\end{align}
Then by~\eqref{eqentropydecomp}, \eqref{eqHfE} and \eqref{eqHfO}
we may write
\begin{align}\label{eqentropymain}
H(f)\le
 \sum_{u\in U}I(u)+ 
\sum_{u\in\cE}\left(1-\frac{d'_U(u)}{d'}\right)H(f_u)+
\sum_{u\in\cO\bs U}H(f_u).
\end{align}
Again we bound each term individually.
For a subset $c\subset V(H)$,
let $a_c\bydef \sum_{i\in c}a_i$.
For the second sum
we note that $f(u)\in A[\lam]$ for each $u\in \cE \bs \rho$,
$d'_U(u)=d'$ for each $u\in \rho$ and $|A[\lam]|=a_A$ so that
\begin{align}
 \sum_{u\in\cE}\left(1-\frac{d'_U(u)}{d'}\right)H(f_u)
 \le \sum_{u\in\cE\bs\rho}\left(1-\frac{d'_U(u)}{d'}\right)\log a_A=(m^n/2-|U|) \log a_A\, .
\end{align}
For the third sum in~\eqref{eqentropymain} we simply use 
$f(u)\in B[\lam]$ for each $u\in \cE\bs U$ so that
\begin{align}
\sum_{u\in\cO\bs U}H(f_u)\le (m^n/2-|U|) \log a_B.
\end{align}

It remains to bound the first sum in~\eqref{eqentropymain}. 
For any $u\in V$ 
we may bound $I(u)$
as follows:
\begin{align}
I(u)
&\le
\max_{c\subset V(H)}
\left\{\frac{1}{d'}H(f_{N_u} | \overline{ f(N_u) }=c)+H(f_u | \overline{ f(N_u) }=c)\right\}
+\frac{1}{d'}H(\overline{ f(N_u) }) \label{eqIbd1}\\
&\le
\max_{c\subset [q]}
\left\{
\log a_c + \log a_{n(c)}
\right\}
+\frac{q}{d'}\\
&\le
\log \eta(H[\lam]) + \frac{q}{d'}  \label{eqIbd3}\, .
\end{align}
For the second inequality we used~\eqref{equniformbound},
noting that there are at most $a_c^{d'}$ possible colourings of
$N_u$ given that $\overline{ f(N_u) }=c$,
and that there are at most 
$2^q$ choices for $c$.

For $u\in N'(\rho)$ we will improve on the bound~\eqref{eqIbd3}
 by considering vertices in $N_u \cap \rho$,
 on which $f$ disagrees with the pattern $(A[\lam], B[\lam])$.
Let $d_u$ denote $d'_\rho(u)$.
Since $f\in\chi(\rho,T, \lam)$ we know that 
$f(N_u\cap\rho)\subset (A[\lam])^c$ and
$f(N_u\bs\rho)\subset A[\lam]$. 
By~\eqref{eqIbd1}, we then have
\begin{align}\label{eqintmax}
I(u)
&\le
\max_{c\subset V(H)}
\left\{
\frac{d_u}{d'} \log a_{c \cap A^c} + \left(1-\frac{d_u}{d'}\right) \log a_{c \cap A} + \log a_{ n(c) }
\right\}
+\frac{q}{d'}\, .
\end{align}
Holding $c$ fixed 
and considering $a_{c \cap A^c}$, $a_{c \cap A}$ as 
continuous variables which sum to $a_c$,
the above expression is maximised when
$a_{c \cap A^c} = \tfrac{d_u}{d'}a_c $ and
$a_{c \cap A} = \left(1-\tfrac{d_u}{d'}\right) a_c $ 
and so
\begin{align}
I(u)
&\le
-H(d_u/d')+
\max_{c\subset V(H)}\{ \log a_c + \log a_{n(c)}\}+
\frac{q}{d'}\\
&\le
-H(d_u/d')+
\log \eta(H)
+\frac{q}{d'}\label{eqIH}\, .
\end{align}
The appearance of the $-H(d_u/d')$ terms here is key. 
These `entropy penalties' will exactly balance
the binomial terms in our polymer count, Lemma~\ref{lemrecon}
(see also \eqref{eqexpH}). 

Recalling the definition of $r$ \eqref{eqdefr}, we have
 \[
 r=r(H,\lam)=\frac{\min_i a_i}{a_1+\ldots+a_q}\, .
 \]
 If $d_u\leq \psi$, we can do better in~\eqref{eqintmax} by 
considering $a_{c \cap A^c}$, $a_{c \cap A}$ as 
integer variables which sum to $a_c$. 
In this case, since $\psi<rd'$ by~\eqref{eqpsibound}, the expression in~\eqref{eqintmax} is
maximised when $a_{c\cap A^c}=\min_i a_i$
 and $a_{c\cap A}=a_c-\min_i a_i$. We then have
\begin{align}
I(u)
&\le
\max_{c\subset V(H)}\{ \log (a_c-\min_i a_i) + \log a_{n(c)}\}+
{q}/{d'}\\
&\le
\max_{c\subset V(H)}\{ \log a_c + \log a_{n(c)}+ \log(1-r)\}+
{q}/{d'}\\
&\le \log(1-r)+\log \eta(H[\lam])+
{q}/{d'}\\
&\le -H(d_u/d')-r/5+\log \eta(H[\lam])+
{q}/{d'}\, . \label{eqintgain}
\end{align}
For the final inequality we used $d_u/d'\leq \psi/d'\leq r^2/2$ 
and $H(x)\leq -\log(1-r)-r/5$ for $0\leq x\leq r^2/2$.

By \eqref{eqIbd3}, \eqref{eqIH} and~\eqref{eqintgain},
\begin{align}
 \sum_{u\in U}I(u)
 \le |U|\left(\log \eta( H[\lam])
+\frac{q}{d'}\right)
-\sum_{t\in D}H(t/d')-r|D_\psi|/5\, .
\end{align}

Combining our bounds on the three sums in~\eqref{eqentropymain}, 
we have
\begin{align}
\log | \chi(\rho, T, \lam) | = H(f)\le \frac{m^n}{2} \log \eta(H[\lam])-\sum_{t\in D}H\left(\frac{t}{d'}\right)-\frac{r|D_\psi|}{5}+\frac{q|U|}{d'}\, .
\end{align}
 Recall that $R$ is a fingerprint of a polymer $\gam\in\cH(g)$ and so
 $|N'(\rho)|\le g$.
 Moreover, since $T\subset S$ 
 where $S$ is a part of a $\psi$-approximating pair $(F,S)$ of $\gam$
 we have $|T|\le 2g$ by~\eqref{eqSbd} and so $|U|\le 3g$.
 The claim follows.
\end{proof}

It remains to deal with the case
where $\lam$ takes irrational values.
We proceed by a limiting argument.
Choose a sequence
$(\lam^k)$ where
 $\lam^k : V(H)\to \Q_{>0}$
and
$
\lam^k \to \lam
$
pointwise.
Note that, as $k\to \infty$,
\begin{align}\label{eqlamlim}
\eta_{\lam^k}(H)\to \eta_\lam(H)
\end{align}
and (recalling~\eqref{eqdefr})
\begin{align}\label{eqrlim}
r(H,\lam^k)\to r(H,\lam)\, .
\end{align}
With Claim~\ref{clhatchi}
applied to $H[\lam^k]$
and letting $\psi_k\bydef r^2(H,\lam^k)/3$
we obtain

\begin{align}
\frac{\sum_{f\in\chi(\rho, T)} \prod_{v\in V}\lam^k(f(v))}
{ \eta_{\lam^k}(H)^{m^n/2}}
\le 
\exp_2\left\{-\sum_{t\in D}H\left({t}/{d'}\right)-\frac{r_k|D_{\psi_k}|}{5}+\frac{3qg}{d'}\right\}\, .
\end{align}
Taking the limit $k\to \infty$,
using~\eqref{eqlamlim} and~\eqref{eqrlim},
 and returning to \eqref{eqchibd} we obtain
\begin{align}
 \sum_{\gam\in\cP(R)}w(\gam)\le 
\exp_2\left\{-\sum_{t\in D}H\left({t}/{d'}\right)-\frac{r|D_\psi|}{5}+\frac{3qg}{d'}\right\}
 \, .
 \end{align}

\end{proof}

By summing over possible fingerprints, 
we can bound the total weight of polymers in $\cH(g)$.

\begin{proof}[Proof of Lemma~\ref{lemgweight}]
Let $\Delta$ denote the collection of multisets $D$
for which $\cR(g,D)\ne\emptyset$ 
(defined in the discussion following Definition~\ref{defprint}) .
For $D\in \Delta$,
the non-zero elements of $D$
must be an integer partition
of some multiple of $d'$ which is $\le gd'$
and so
by Theorem~\ref{thmHR}
\begin{align}\label{eqDeltabd}
|\Delta|\le 2^{O(\sqrt{gd})}\, .
\end{align}

Recall from Lemma~\ref{lemfweight} that
for $D\in\Delta$, $D_\psi$ is the multiset of elements in $D$ that are $\le \psi$.
\begin{claim}\label{eqDr}
If $(F,S)\in 2^V \times 2^V$ is such that
 $\cR_\cE(g,D,F,S)\ne \emptyset$
then
\[
 |D|-|F\cap\cO|\le |D_\psi|\, .
\]
The same statement holds with $\cO, \cE$ swapped.
\end{claim}
\begin{proof}
Let $(\rho,T)\in \cR_\cE(g,D,F,S)$ so that $(F\cap \cO, S\cap \cE)$
is an approximating pair for $\rho$.
By~\eqref{eqapprox1} and~\eqref{eqapprox3} in the definition of a $\psi$-approximating pair,
we have $d_\rho(v)\le d_{S\cap \cE}(v)\le \psi$ for each $v\in \cO\bs F$.
Since $D$ is the neighbourhood distribution of $\rho$, 
we thus have at least $|N(\rho)\bs (F\cap\cO)|=|D|-|F\cap\cO|$ 
elements of $D$ which are $\le \psi$.
The claim holds with $\cO, \cE$ swapped by symmetry.
\end{proof}

Recall from~\eqref{eqelldef0} that
we chose $k\in\N$ such that
\begin{align}\label{eqelldef}
\ell= 2^k=\Theta(\sqrt{d}\log^4 d)
\end{align}
and that $d-2\ell \le d' \le d$ from~\eqref{eqdprime}.

Let $\Delta_1\subset \Delta$ be the 
collection of multisets $D\in  \Delta$
such that $|D_\psi|> g(\log d)^2/\ell$.
Let $\Delta_0=\Delta\bs \Delta_1$.
Suppose that $D\in \Delta_1$.
Then 
\begin{align}
\sum_{R\in\cR(g,D)}\sum_{\gam\in \cP(R)}w(\gam)
&\le
2\sum_{(F,S)\in\cU(g)}\sum_{R\in\cR_\cE(g,D,F,S)}\sum_{\gam\in \cP(R)}w(\gam)\\
&\le
2\sum_{(F,S)\in\cU(g)}\sum_{R\in\cR_\cE(g,D,F,S)}
\exp_2\left\{ -{\sum_{t\in D}H(t/d')-r|D_\psi|/5+3qg/d'}\right\}\\
&\le
 2^{O(g (\log d)/\ell)+|D_\psi|\psi/d'}
\prod_{t\in D}
\binom{d'}{t}^{1/d'}
\exp_2\left\{ -{\sum_{t\in D}H(t/d')-r|D_\psi|/{5}}\right\}\\
&\le
2^{-\Omega(g(\log d)^2/\ell)}\, \label{eqDrbig}\, .
\end{align}
For the first inequality we use Corollary~\ref{corcontainers}
where the factor of $2$ accounts for the sum over 
$\cR_\cO(g,D,F,S)$ by symmetry.
For the second inequality we use Lemma~\ref{lemfweight}.
For the third inequality we used 
Corollary~\ref{corcontainers}, Lemma~\ref{lemfprint}
and Claim~\ref{eqDr}.
For the final inequality we use $r\le 1/2$
in collecting the $D_\psi$ terms.

Let us assume now that $D\in\Delta_0$.
It will be useful to refine the notion
of approximation for polymers $\gam\in\cH(g)$
with a fingerprint in $\cR_\cE(g,D)$
by taking advantage of the fact that
any approximating pair for $\gam$
must specify almost all of $N(\gam\cap\cE)$.
Indeed suppose that $\gam\in\cH(g)$
such that $|N(\gam\cap\cE)|\ge |N(\gam\cap\cO)|$.
We call $(F,S)\in 2^V \times 2^V$
a \emph{super-approximating pair}
for $\gam$ if $(F,S)$ is an approximating pair
for $\gam$ and moreover
\begin{align}\label{eqsup}
N(\gam\cap\cE)=F\cap\cO \text{ and } N(S\cap \cE)=F\cap\cO\, .
\end{align}

\begin{claim}\label{cluast}
There exists a family $\cU^\ast=\cU^\ast(g)\subset 2^V \times 2^V$
such that 
\begin{align}\label{equastbd}
|\cU^\ast| \le 2^{O(g(\log d)^3/\ell)}
\end{align}
with the following property:
for each $D\in \Delta_0$
and $\gam\in\cH(g)$ with fingerprint
in
$\cR_\cE(g,D)$,
there exists $(F,S)\in \cU^\ast$ such that 
$(F,S)$ is a super-approximating pair for $\gamma$.
\end{claim}

We first show how our lemma follows from Claim~\ref{cluast}.
 
 Let $D\in \Delta_0$ and suppose that $(F,S)\in \cU^\ast$
 is such that $\cR_\cE(g,D,F,S)\neq \emptyset$.
Recall that polymers are defined to be 
$G^2$-connected subsets $\gamma\subset V$ such that 
$|N(\gamma\cap \cE)|, |N(\gamma\cap \cO)|<(1-\alpha)m^n/2$ and so
$|F\cap\cO|<(1-\alpha)m^n/2$ by~\eqref{eqsup}.
Since $N(S\cap\cE)=F\cap\cO$ (again by~\eqref{eqsup}),
we have
by Lemma~\ref{lembipexp},
\begin{align}\label{eqgiso}
|S\cap\cE|= \left(1-\Omega({1}/{\sqrt{n})}\right)|F\cap\cO|\, .
\end{align}
Noting that $|F\cap\cO|\ge g/2$,  
by~\eqref{eqsup} and the 
 definition of a fingerprint (Definition~\ref{defprint}),
 we thus have
\begin{align}\label{eqdegreg}
e_{G'}(F\cap\cO, (S\cap\cE)^c)= d'(|F\cap\cO|- |S\cap\cE|)
=
\Omega(g \sqrt{d})\, .
\end{align}
We pause here to highlight the fact
that we used vertex isoperimetry in $G$
to obtain~\eqref{eqgiso} and then degree regularity in $G'$
to deduce~\eqref{eqdegreg}, an edge isoperimetric inequality in $G'$. 
We now use this edge expansion to 
more effectively bound the product in Lemma~\ref{lemfprint}.
We remark that edge expansion in $G$ alone
would not suffice for this purpose
since boundary edges in $G$ might be missing in $G'$.

Returning to the proof, note that 
 since $|D|=|F\cap \cO|$ by~\eqref{eqsup},
  we have by Lemma~\ref{lemfprint},
\begin{align}\label{eqRDFSbd}
|\cR_\cE(g,D,F,S)|\le
 2^{O(g (\log d)/\ell)}
\prod_{v\in F\cap\cO}
\binom{d'_{S\cap\cE}(v)}{x(v)}^{1/d'}\, ,
\end{align}
where $x=x_{ F\cap\cO, S\cap\cE, D}$
(as defined in~\eqref{eqxdef}).

Recall that $D\in \Delta_0$
so that all but at most $g(\log d)^2/\ell$
elements of $D$ are greater that $\psi$.
Suppose that $x(v)\ge \psi$,
 then for
$x(v)\le x \le d'$ we have
\[
\left. \binom{x-1}{x(v)} \middle/ \binom{x}{x(v)}\right. 
=1-x(v)/x \le 1-\psi/d'.
\]
Recalling the value of $\ell$ from~\eqref{eqelldef}
we have
\begin{align}
\prod_{v\in F\cap \cO}
\binom{d'_{S\cap\cE}(v)}{x(v)}
&\le
(1-\psi/d')^{e_{G'}(F\cap\cO, (S\cap\cE)^c)-gd(\log d)^2/\ell}
\prod_{t\in D}
\binom{d'}{t}\\
&=
(1-\psi/d')^{\Omega(g\sqrt{d})}
\prod_{t\in D}
\binom{d'}{t}\, .
\end{align}

Returning to \eqref{eqRDFSbd} we therefore have
\[
|\cR_\cE(g,D,F,S)|
\le
 2^{-\Omega(g/\sqrt{d})}
\prod_{t\in D}
\binom{d'}{t}^{1/d'}\, .
\]

Thus by Lemma~\ref{lemfweight}
\begin{align}
\sum_{R\in\cR_\cE(g,D,F,S)}\sum_{\gam\in \cP(R)}w(\gam)
&\le
 2^{-\Omega(g/\sqrt{d})}\prod_{t\in  D} 
\binom{d'}{t}^{1/d'}
\exp_2\left\{ -\sum_{t\in  D}H(t/d')+3qg/d'\right\}\\
&\le
2^{-\Omega(g/\sqrt{d})}\, . \label{eqdelta0}
\end{align}

By Claim~\ref{cluast} we have

\begin{align}
\sum_{\gamma\in\cH(g)}w(\gamma)
\le
2\sum_{D\in \Delta_1}
\sum_{R\in\cR_\cE(g,D)}
\sum_{\gam\in \cP(R)}w(\gam)
&+
2
\sum_{D\in \Delta_0}
\sum_{(F,S)\in\cU^\ast }
\sum_{R\in\cR_\cE(g,D,F,S)}
\sum_{\gam\in \cP(R)}w(\gam)\, .
\end{align}
The factors of $2$ 
account for the sums over $R_\cO(g,D)$ by symmetry.
By the bounds~\eqref{eqDeltabd}, 
\eqref{eqDrbig}, \eqref{equastbd} and~\eqref{eqdelta0}
we have that the right hand side is at most $2^{-\Omega(g/(\sqrt{d}\log^2 d))}$
(recall that $g\ge d^4$ by assumption).

It remains to verify Claim~\ref{cluast}.

\begin{proof}[Proof of Claim~\ref{cluast}]
Let $(F,S)\in \cU(g)$ where $\cU(g)$ is the set from Corollary~\ref{corcontainers}
and let 
\begin{align}
\cV_{F,S}\bydef 
\left\{(F\cap\cO )\cup Q: Q\in \binom{N(S)}{\le g(\log d)^2/\ell} \right\}\, \subset 2^\cO\, .
\end{align}
For $X\subset V$ let $B(X)\bydef \{v\in V: N(v)\subset X\}$.
Define
 \[
 \cU^\ast\bydef  \bigcup_{(F,S)\in \cU(g)} \bigcup_{X\in \cV_{F,S}}((F\cap\cE)\cup X, (S\cap\cO)\cup B(X))\, .
 \]
We now show that  $\cU^\ast$ has the desired properties. 
Let
$D\in \Delta_0$ 
and let $\gam\in\cH(g)$ 
such that 
 $|N(\gam\cap\cE)|\ge |N(\gam\cap\cO)|$.
 Let $\rho\bydef \gam\cap\cE$ have neighbourhood distribution $D$.
 By Corollary~\ref{corcontainers} there exists $(F,S)\in \cU(g)$
 such that $(F,S)$ is an approximating pair for $\gamma$.
 In particular by~\eqref{eqapprox1}, $F\cap\cO\subset N(\rho) \subset N(S)$.
 By Claim~\ref{eqDr}, $|N(\rho)\bs (F\cap \cO)|\le g(\log d)^2/\ell$ and so
 $N(\rho)\in \cV_{F,S}$.
 It follows that the pair
 \[
(F',S')=((F\cap\cE)\cup N(\rho), (S\cap\cO)\cup B(N(\rho)))
 \]
belongs to $\cU^\ast$.
It is easy to verify that $(F',S')$ is a super-approximating pair for $\gamma$.
It remains to bound on the size of $\cU^\ast$.

For $(F,S)\in \cU(g)$ we have $|N(S)| \le 2dg$
by~\eqref{eqNSbd}
 and so $|\mathcal V_{F,S}|\le \binom{2dg}{\le g(\log d)^2/\ell}=2^{O(g(\log d)^3/\ell)}$.
 By Corollary~\ref{corcontainers} we thus have
 \[
 |\cU^\ast|\le 2^{O(g (\log d) / \psi+g(\log d)^3/\ell)} =2^{O(g(\log d)^3/\ell)} \qedhere
 \]
\end{proof} 
\end{proof}

\section{Verifying the Koteck\'y-Preiss condition}\label{secVKP}
In this section we prove Lemma~\ref{lemKP}. That is, we verify
the Koteck\'y-Preiss condition~\eqref{eqKPcond} 
for each of our polymer models with suitable functions $f$ and $g$.
We then use this lemma to establish tail bounds 
on the cluster expansions of the log partition functions $\ln \Xi_{A,B}$. 
In particular we prove a more detailed version Lemma~\ref{lemexpL11}.

\begin{proof}[Proof of Lemma~\ref{lemKP}] 
With 
$\xi$ as in Lemma~\ref{lemgweight}
and $r$ as defined in~\eqref{eqdefr}, let
\begin{align} \label{eqfg}
f(\gam)=
|\gam|/d  \text {\,  and \, }
g(\gam)=
\begin{cases} 
       r |N(\gam)|/3& \text{ if }|N(\gam)|\le d^4 \\
       \xi |N(\gam)|/(4\sqrt{d} \log^2 d) & \text{ otherwise .}
   \end{cases}
\end{align}

Let us fix a dominant pattern $(A,B)\in\cD_\lam(H)$.
We denote $w_{A,B}$ simply by $w$.
We want to show that
the condition~\eqref{eqKPcond2} holds with the above choice of $f$ and $g$. 
That is, we want to show that
\begin{align}\label{eqKPsimple}
\sum_{\gamma':d(\gamma', \gamma)\leq2} 
w(\gamma')
 e^{f(\gam')+g(\gam')}\le f(\gam) 
\end{align}
for all polymers $\gamma\in \cP$.

We will in fact show that for each $v\in V(G)$
\begin{equation}\label{eqvertex}
\sum_{\gamma':\gamma'\ni v }w(\gam') \cdot e^{f(\gam')+g(\gam')}
\le \frac{1}{d^3}= \frac{f(\gam)}{d^2|\gamma|}\, ,
\end{equation}
for all $\gamma\in \cP$.  
By summing this inequality over all $v$ at distance at most $2$ from $\gam$ in $G$ 
(noting that there are $\le d^2|\gamma|$ such vertices) we establish~\eqref{eqKPsimple}. 

We split the sum appearing in~\eqref{eqvertex} according to the size of $|N(\gam')|$.
First we consider those $\gam'$ for which $|N(\gam')|\ge d^4$.
By Lemma~\ref{lemgweight} we have

\begin{align}
\sum_{g=d^4}^{(1-\alpha)m^n/2}
\sum_{\substack{\gam\in\cH(g)}}
w(\gam) \cdot e^{f(\gam)+g(\gam)}
&\le
\sum_{g=d^4}^{(1-\alpha)m^n/2}
e^{\xi g /(2\sqrt{d} \log^2 d)}
\sum_{\gam\in\cH(g)}
w(\gam)\\
&\le
\sum_{g=d^4}^{\infty}e^{-\xi g /(2\sqrt{d} \log^2 d)}\\
&\le
\frac{1}{2d^3}\, .\label{eqlargesum}
\end{align}
for $d$ sufficiently large. 

Next we consider the contribution to the sum in~\eqref{eqvertex}
coming from polymers $\gamma\in \cH(g)$ where $g\le d^4$.
To this end, we need the following bound on the weight of a polymer.
Recall that for $S\subset V$
we let $\p(S)=N(S)\bs S$ and $S^+=N(S)\cup S$.
\begin{claim}\label{lemwbound}
For $(A,B)\in\cD_\lam(H)$ and $\gam\in\cP$ we have 
\[
w(\gam)\le
(1-r)^{|\p \gam|} r^{-|\gam|}
\, .
\]
\end{claim}
\begin{proof}[Proof of Claim~\ref{lemwbound}]
By~\eqref{eqconv} we have
\begin{align}
 w(\gam)= \frac{\sum_{f\in\hat\chi_{A,B}(\gam)} \prod_{v\in \gam^+}\lam_{f(v)}}{\lam_A^{|\gamma^+\cap\cO|} \lam_B^{|\gamma^+\cap\cE|}}\, ,
\end{align}
where 
$\hat\chi_{A,B}(\gam)$
denotes the set of possible restrictions $f_{\gamma^+}$
where $f\in \chi_{A,B}(\gam)$.

Given any $v\in A^c$ there must exist a $w\in B$ such that $v \nsim w$
else $A\cup\{v\}\sim B$ contradicting the fact that $(A,B)$ is a dominant pattern.
Similarly for each $v\in B^c$ there is a $w\in A$ such that $v \nsim w$.
It follows that 
\[
\sum_{f\in\hat\chi_{A,B}(\gam)} \prod_{v\in \gam^+}\lam_{f(v)}
\le 
\lam_{A^c}^{|\gam\cap\cO|}\left(\lam_B- \min_{i\in B}\lam_i\right)^{|\p \gam \cap\cE|}
 \lam_{B^c}^{|\gam\cap\cE|} \left(\lam_A- \min_{i\in A}\lam_i\right)^{|\p \gam \cap\cO|}\, ,
\]
and so
\begin{align}
\label{eqwbdsimple}
w(\gam)\le
\left(1-\frac{ \min_{i\in A}\lam_i}{\lam_A}\right)^{|\p \gam\cap \cO|}
\left(1-\frac{ \min_{i\in B}\lam_i}{\lam_B}\right)^{|\p \gam\cap \cE|}
\left(\frac{\lam_{A^c}}{\lam_A}\right)^{|\gam\cap\cO|}
\left(\frac{\lam_{B^c}}{\lam_B}\right)^{|\gam\cap\cE|}\, .
\end{align}
The result follows by comparing the above quotients to $r$.
\end{proof}

If $|N(\gam)|=g\le d^4$ then it follows from Lemma~\ref{thmviso} that 
$|\gam|\le g/(Cd)$.
We then have by Claim~\ref{lemwbound}
\begin{align}
w(\gam)\le
\left(1-r\right)^{g/2}
r^{-g/(Cd)}\, ,
\end{align}
for $d$ sufficiently large.

Note that the graph $G^2$ has maximum degree at most $d^2$
and so by Lemma~\ref{lemConCount} the number of $G^2$-connected
sets of size $t$ containing a fixed vertex $v$ is at most $(ed^2)^{t-1}$.
Note also if $\gamma\neq \emptyset$ then $|N(\gamma)|\ge d$.
It follows that
\begin{align}\label{eqsmallsum}
\sum_{g=d}^{d^4}
\sum_{\substack{\gam\in\cH(g):\\ \gam \ni v}}
w(\gam) \cdot e^{f(\gam)+g(\gam)}
\le
\sum_{g=d}^{d^4}(e^2d^2/r)^{g/(Cd)}\left(1-r\right)^{g/2}e^{gr/3}\le \frac{1}{2d^3}
\end{align}
for $d$ sufficiently large. 
By summing~\eqref{eqlargesum} and \eqref{eqsmallsum},
we obtain \eqref{eqvertex} and hence also \eqref{eqKPsimple}.
This concludes the proof.
\end{proof}

We end this section by using Lemma~\ref{lemKP}
to establish strong tail bounds on the cluster expansion of $\ln \Xi_{A,B}$.

Recall that for a  dominant pattern $(A,B)\in\cD_{\lam}(H)$
\begin{align}
\delta_{A,B}\bydef  \max\left\{\max_{v\in B^c}\frac{\lam_{N(v)\cap A}}{\lam_A},\max_{u\in A^c} \frac{\lam_{N(u)\cap B}}{\lam_B}   \right\}\, ,
\end{align}
and recall that $\delta_{A,B}<1$ (Lemma~\ref{lemdelta1}).
Recall also that
for $k\ge 1$ we define
\begin{align}
L_{A,B}(k) = \sum_{\substack{\Gamma \in \cC : \\ \|\Gamma\|= k}}  w_{A,B}(\Gamma)    \, .
\end{align}

The following lemma extends Lemma~\ref{lemexpL11}.
\begin{lemma}\label{lemexpL1}
Let $(A,B)\in \cD_\lam(H)$ and let $\delta=\delta_{A,B}$.
\begin{align}\label{eqL1bd}
L_{A,B}(1)=\Theta(m^n \delta^d)\, .
\end{align}
Moreover
 for 
$t\ge 1$ fixed
\begin{align}\label{eqLsum}
\sum_{k=t}^\infty |L_{A,B}(k)|= O(m^n d^{2(t-1)}\delta^{dt})\, ,
\end{align}
and for $s\ge 0$ fixed 
\begin{align}\label{eqkssum}
\sum_{k=1}^\infty k^s |L_{A,B}(k)|= (1+o(1))L_{A,B}(1)\, \text{ as } n\to\infty\, .
\end{align}
\end{lemma}
\begin{proof}
Let us fix $(A,B)\in \cD_\lam(H)$.
For ease of notation we let
$L_k=L_{A,B}(k)$ and $w=w_{A,B}$.

The first claim follows by recalling the formula
\begin{align}
\label{eqL1lam}
L_1=\left[\frac{1}{\lam_A\lam_B^d}\sum_{v\in A^c}\lam_v \lam_{N(v)\cap B}^d + \frac{1}{\lam_B\lam_A^d}\sum_{v\in B^c}\lam_v \lam_{N(v)\cap A}^d\right]\frac{m^n}{2}
\end{align}
and using the definition of $\delta$~\eqref{eqdeltadef}.

Suppose that $k\in \N$ is fixed.
We now bound $L_k$ from above.
Let $\Gamma$ be a cluster with $\|\Gamma\|=k$. 
Since $V(\Gamma)\bydef \bigcup_{S\in\Gamma}S$ is a $G^2$-connected set of size at most $k$, 
there are $O(m^n d^{2(k-1)})$ possibilities for $V(\Gamma)$ by Lemma~\ref{lemConCount}. 
Given a set $X\subseteq V(Q_n)$ of size at most $k$, 
there are at most a constant number of clusters $\Gamma$ of size $k$ such that $V(\Gamma)=X$. 
The number of clusters of size $k$ is therefore $O(m^n d^{2(k-1)})$. 

Suppose now $\gam$ is a polymer 
with $|\gam\cap\cO|=i$ and $|\gam\cap\cE|=j$
where $i+j\leq k$.
Using the fact that
 $|\p(\gam\cap\cO)|\ge di-O(1)$
 and
$|\p(\gam\cap\cE)|\ge dj-O(1)$
by Lemma~\ref{thmviso} and~\eqref{eqconv},
we have 
\begin{align}
w(\gam)
&=
O\left(
\frac{1}{\lam_A^{dj}\lam_B^{di}}
\sum_{\substack{v_1,\ldots, v_{i}\in A^c\\ u_1,\ldots, u_{j}\in B^c}}
\lam_{N(u_1)\cap A}^d\cdots\lam_{N(u_j)\cap A}^d\cdot 
\lam_{N(v_1) \cap B}^d\cdots\lam_{N(v_i) \cap B}^d
\right)
\\
&=
O\left(
\delta^{d(i+j)}
\right)\, ,
\end{align}
(this is a refinement of Claim~\ref{lemwbound} in the case where the size of the polymer is constant).
Thus if $\Gamma$ is a cluster of size $k$ then 
$w(\Gam)=O\left(
\delta^{dk}
\right)
$
(note that $k$ is fixed so that the Ursell function $\phi(\Gamma)$ is bounded by a constant).
We therefore have 
\begin{align}\label{eqLkconst}
L_k=O(m^n d^{2(k-1)} \delta^{dk})\, .
\end{align}

We now show how the lemma follows
 from the following claim.

\begin{claim}\label{eqkstail}
For $s\ge0$, $t\ge1$ fixed, there exists a constant $K$ such that
for $n$ sufficiently large,
\begin{align}
\sum_{k>K}k^s L_k\le \delta^{td} \, .
\end{align}
\end{claim}

Combining Claim~\ref{eqkstail} with~\eqref{eqLkconst} (for each $k\le K$),
we have 
\[
\sum_{k=1}^\infty k^s L_k= L_1 + \sum_{k=2}^K O\left(m^n d^{2(k-1)} \delta^{dk}\right) + O(\delta^{td})
=L_1+ O(m^nd^2 \delta^{2d})\, ,
\]
and also
\[
\sum_{k=t}^\infty L_k= \sum_{k=t}^K O(m^n d^{2(k-1)} \delta^{dk}) + O(\del^{td})= 
 O(m^n d^{2(t-1)} \delta^{dt}) \, .
\]
The lemma follows. 
It remains to establish Claim~\ref{eqkstail}.
\begin{proof}[Proof of Claim~\ref{eqkstail}]
Having established the Koteck\'y-Preiss condition (Lemma~\ref{lemKP})
for our polymer model (with $f,g$ as in \eqref{eqfg}), we may 
apply Theorem~\ref{thmKP}. 
In particular, applying \eqref{eqKPtail}
where $\gamma=\{v\}$ is a single vertex polymer 
we have
\begin{align}\label{eqnsim}
\sum_{\substack{\Gamma \in \cC \\  \Gamma \nsim \{v\}}} \left |  w_{A,B}(\Gamma) \right| e^{g(\Gamma)} \le 1/d \,.
\end{align}
Recall that we write $\Gamma \nsim \{v\}$ 
if there exists $\gam \in \Gamma$ so that $\{v\} \nsim \gam$.
For each $\Gamma \in \cC$,  
$\Gamma \nsim \{v\}$ for some $v\in V$
and so 
by summing~\eqref{eqnsim} over all $v\in V$
we obtain
\begin{align}
\label{eqtail}
\sum_{\Gamma\in\cC} |w_{A,B}(\Gamma)| e^{g(\Gamma)}\le m^n/d\, ,
\end{align}
(recall that $g(\Gam)\bydef  \sum_{\gam\in\Gam}g(\gam)$). 
By \eqref{eqfg}, the definition of the function $g(\cdot)$,
we have for any cluster $\Gamma$ and $n$ sufficiently large,
\begin{align}\label{eqeg}
e^{g(\Gamma)/2}\ge \|\Gamma\|^s\, .
\end{align}
Moreover,
there is a constant $K=K(m, H, \lam)$ such that
if $\Gamma$ is a cluster of size $> K$, 
then for $n$ large
\begin{align}\label{eqeg2}
e^{g(\Gamma)/2} \ge m^n\cdot \delta^{-td}\, .
\end{align}
Thus
\[
\sum_{k>K}k^s L_k\le
\sum_{\substack{\Gamma \in  \cC\\ \|\Gamma\|> K}} |w(\Gamma)| e^{g(\Gamma)/2} \le \delta^{td} \, ,
\]
where for the first inequality we used~\eqref{eqeg}
and for the second inequality we used~\eqref{eqnsim} and~\eqref{eqeg2}\, .
\end{proof}
\end{proof}

\section{Large deviations for polymer configurations}\label{secLD}
In this section we use the Koteck\'y-Preiss condition (Lemma~\ref{lemKP})
established in the previous section to 
prove a large deviation result for the total size 
of all polymers in a random sample 
from the polymer measure $\nu_{A,B}$ (as defined in~\eqref{eqnudef}). 
As a corollary, we show that a typical 
sample from $\hat \mu_{H,\lam}$ has very few defect vertices 
and is `highly-balanced'
in a sense that we will make precise later in the section.

In the next section we use these results to 
prove Lemmas~\ref{lemH01} and~\ref{lemH21} 
and hence also Theorem~\ref{mainTV}.
The results of this section will also play a key role in 
Sections~\ref{secgenstruc} and~\ref{seckbd}.

The proof of the following theorem and the proofs in Section~\ref{secdefect}, exploit a connection between the cluster expansion and cumulant generating functions.
Recently Cannon and Perkins~\cite{cannon2020counting} made novel use of this connection to prove correlation decay results for the hard-core model on unbalanced bipartite graphs. Our treatment parallels that of the first author and Perkins~\cite{jenssen2019independent} who prove a large deviation result and central limit theorems for the hard-core model on $Q_n$. 

Suppose $X$ is a random variable
 whose moment generating function 
 $\E e^{tX}$ is defined for $t$ in a neighbourhood of $0$. 
 We will use the 
 \textit{cumulant generating function} of  $X$, 
 defined as 
\begin{align}\label{eqcumulant}
h_t(X) &= \ln \E e^{t X}  \,,
\end{align}
that is, the logarithm of the moment generating function.

\begin{theorem}
\label{lempolyLD}
Let $(A,B)\in\cD_\lam(H)$ and 
let $\delta=\delta_{A,B}$ (as defined in~\eqref{eqdeltadef}).
Let $\mathbf \Gamma$ be a random configuration drawn from the distribution $\nu_{A,B}$. 
There exist constants $C, c>0$, depending on $(H,\lam)$ and $m$, such that 
if $t\ge C\cdot d\delta^d$,
 then
\begin{align*}
\P(\|\mathbf\Gamma\| &\ge  t m^n)\leq  e^{  - c\cdot tm^n/n   } \, .
\end{align*}
\end{theorem}
\begin{proof}
Let $w$ denote $w_{A,B}$ and let $\Xi$ denote $\Xi_{A,B}$.
We introduce an auxiliary polymer model with modified polymer weights:
\begin{align*}
\tilde w(S) &= w(S) e^{ |S|/d}  \,.
\end{align*}
Let $\tilde \Xi$ be the associated polymer model partition function.  
Recall that $\Omega$ denotes the collection of 
all sets of mutually compatible polymers.
We have
\begin{align}
\frac{\tilde \Xi }{\Xi}=
 \frac{1}{\Xi}\sum_{\Gamma \in \Omega}\prod_{\gam\in \Gamma}\tilde w(\gam)
 =
  \frac{1}{\Xi}\sum_{\Gamma \in \Omega} e^{\|\Gamma\|/d}
  \prod_{\gam\in \Gamma} w(\gam)
  =  \E e^{ \|\mathbf\Gamma\|/d}
\end{align}
where $\mathbf \Gamma$ is a random polymer configuration
drawn according to $\nu_{A,B}$ (the unmodified polymer model).
In other words
\begin{align}\label{eqinother}
h_{1/d}(\|\mathbf\Gamma\|)=\ln \tilde \Xi -  \ln \Xi\, .
\end{align}
We claim that the cluster expansion of $\ln \tilde \Xi$ converges absolutely. 
Indeed let $v\in V$ be a fixed vertex. 
The sum in the Koteck\'y-Preiss condition 
\eqref{eqvertex} for the polymer model 
with modified weights $\tilde w (\cdot)$ 
with $f$ and $g$ as in \eqref{eqfg} is
\begin{align}
\sum_{\gam \ni v}  |\tilde w(\gam)| e^{f(\gam) +g(\gam)}
=
\sum_{\gam \ni v}  |w(\gam)| e^{f(\gam)+g(\gam)+|\gam|/d}\, .
\end{align}
Returning to inequalities~\eqref{eqlargesum} and \eqref{eqsmallsum},
we observe that they both equally hold with
 $f(\gam)$
replaced by $f(\gam)+|\gam|/d$. 
We conclude that
\begin{align}\label{eqmodKP}
\sum_{\gam \ni v}  |\tilde w(\gam)| e^{f(\gam) +g(\gam)}\le 1/d^3 .
\end{align}
In other words we have verified
the Koteck\'y-Preiss condition 
\eqref{eqvertex} for the modified polymer model 
and so the cluster expansion of $\ln \tilde \Xi$ converges absolutely.
Using~\eqref{eqinother} and the bound $\Xi \ge 1$, then
 applying Lemma~\ref{lemexpL1}
  (noting that the proofs remain valid with the modified weights),
  gives
\begin{align}
h_{1/d}(\|\mathbf\Gamma\|)&\le
 \ln \tilde \Xi\\
  &\le \sum_{\Gamma \in \cC} | \tilde w(\Gamma)| \\
& = O(m^n \delta^d)  \,. \label{eqhbd}
\end{align}
By Markov's inequality we have,
\begin{align*}
\P[\|\mathbf\Gamma\| > t m^n]
&= \P\left[e^{ \|\mathbf\Gamma\|/d} > e^{t m^n/d}\right]\\
 &\le e^{-tm^n/d} \E e^{| \mathbf \Gamma|/d } \,.
\end{align*}
Applying the bound~\eqref{eqhbd} gives
\begin{align*}
\P[\|\mathbf\Gamma\| > t m^n] &\le \exp \left [ -\frac{t m^n}{d}  + O(m^n \delta^d) \right ]\, .
\end{align*}   
The result follows.
\end{proof}

For $s>0$, and weighted graph $(H,\lam)$,
we say that a colouring $f\in \hom(G,H)$ is
\emph{$(s,\lam)$-balanced} with respect to $(A,B)$
 if for each
$k\in A$ ,
the proportion of vertices of $\cE$ 
coloured $k$ is within $s$
of $\lam_k/\lam_A$ and
for each
$\ell \in B$,
the proportion of vertices of $\cO$ 
coloured $\ell$ is within $s$
of $\lam_\ell/\lam_B$.

The following
lemma shows that with high probability,
 a colouring sampled
from the measure $\hat\mu_{H,\lam}$ (as defined in Definition~\ref{defmuhat})
is well-balanced with respect to some dominant pattern. 

Recall that we call the set of vertices at which a colouring $f\in\hom(\Z_m^n,H)$
differs from its closest dominant colouring (breaking ties arbitrarily if necessary)
the \emph{defect vertices} of $f$.

\begin{lemma}\label{lembalanced}
Let $\mathbf f$ denote a random element of $\hom(G,H)$
sampled according to $\hat \mu_{H,\lam}$, let 
$\mathbf{D}$ denote the random dominant pattern 
selected at Step \ref{step1} in the definition of $\hat \mu_{H,\lam}$ (Definition~\ref{defmuhat})
and let $\mathbf \Gamma$ be the random polymer configuration
selected at Step \ref{step2}.
Let $X$ denote the number of defect vertices of $\mathbf f$.
The following hold:
\begin{enumerate}[label=(\roman*)]
\item \label{itemdefect}
\[
\P(X=\|\mathbf \Gamma\|)\geq 1- e^{-\Omega(m^n/n^2)}\, .
\]
\item \label{itemdefectlarge} If $(A,B)\in \cD_{\lam}(H)$ and $Cd\delta_{A,B}^d\leq t \leq 1/(4d)$,
with $C$ as in Theorem~\ref{lempolyLD}, then
\begin{align}
\P(X\geq tm^n \mid \mathbf{D}= (A,B))\leq  e^{ - \Omega(tm^n/n)}\, .
\end{align}
\item \label{itembal} If $(A,B)\in \cD_{\lam}(H)$ and $s\geq 10Cd^2 \delta_{A,B}^d$, then
\begin{align}
\P(\mathbf{f}\text{ is }(s, \lam)\text{-balanced wrt }(A,B) \mid \mathbf{D}=(A,B))\ge 1- e^{-\Omega(s^2m^n)}-e^{-\Omega(sm^n/n^2)}\, .
\end{align}
\end{enumerate}
\end{lemma}
\begin{proof}
We begin by proving \ref{itemdefect} and \ref{itemdefectlarge} in conjunction. 
Let $(A,B)\in \cD_{\lam}(H)$ and $Cd\delta_{A,B}^d\leq t \leq 1/(4d)$.
By Theorem~\ref{lempolyLD},  
 we have  
\begin{align}\label{eqGammaleq}
\P(\| \mathbf  \Gamma \| \le  t m^n \mid \mathbf D = (A,B)) \geq 
1 - e^{ - \Omega(tm^n/n)}\, .
\end{align}

Let $\mathbf{\Gamma}^{+}$ denote 
$\bigcup_{\gamma\in \mathbf \Gamma}\gamma^{+}$ and 
let $\|\mathbf{\Gamma}^{+}\|\bydef\sum_{\gamma\in \mathbf \Gamma}|\gamma^{+}|$.

For $k\in A$, let $Z_k$ denote the
number of vertices in $\cO\bs \mathbf \Gamma^{+}$ that receive colour $k$
conditioned on the event $\mathbf D = (A,B)$.
Since the vertices of $V \bs \mathbf \Gamma^{+}$ are
coloured at Step \ref{step4} in the 
definition of $\hat \mu_{H,\lam}$ we see that 
\begin{align}
Z_k\sim \bin(|\cO\bs \mathbf \Gamma^{+}|, \lam_k/\lam_A)\, .
\end{align}
For $k\in B$ we define $Z_k$ similarly and note that
$Z_k\sim\bin(|\cE\bs \mathbf \Gamma^{+}|, \lam_\ell/\lam_B)$.
Since $t\leq 1/(4d)$ and $\|{\mathbf\Gamma}^{+}\|\leq (d+1) \|\mathbf\Gamma\|$,
if $\| \mathbf\Gamma \| \le  t m^n$, then $|\cO\bs \mathbf\Gamma^{+}|, |\cE\bs \mathbf\Gamma^{+}|\geq m^n/6$.

For $k\in A$, let $Y_k$ denote the number of vertices of $\cO$ that receive colour $k$ in the colouring $\mathbf f$ conditioned on the event $\mathbf D=(A,B)$. Define $Y_k$ similarly for $k\in B$. Note that for $k\in A\cup B$,
\begin{align}\label{eqYZ}
0\leq Y_k-Z_k\leq \|\mathbf \Gamma^{+}\|\, .
\end{align}

By Chernoff's bound (applied to the random variables $Z_k$) and a union bound we then have
\begin{align}\label{eqminY}
\P\left(\min_{k\in A\cup B}Y_k=\Omega(m^n) \;\middle\vert\; \| \mathbf  \Gamma \| \le  t m^n, \mathbf D=(A,B)\right)\geq1- e^{-\Omega(m^n)}\, .
\end{align}
Suppose that the events $\mathbf D=(A,B)$, $\min_{k\in A\cup B}Y_k=\Omega(m^n)$,
 and $\| \mathbf  \Gamma \| \leq  t m^n$ all hold.
 If $(C,D)\in \cD_{\lam}(H)$ is a dominant pattern distinct
 from $(A,B)$, then either $A\bs C\neq \emptyset$ or $B \bs D\neq \emptyset$
 and so, since $\min_{k\in A\cup B}Y_k=\Omega(m^n)$, 
 $\mathbf{f}$ must disagree with $(C,D)$ on $\Omega(m^n)$ vertices.
 On the other hand, since $\mathbf D=(A,B)$,
 $\mathbf{f}$ disagrees with $(A,B)$ on $\| \mathbf  \Gamma \| \leq  t m^n$
 vertices.
We conclude that  
$(A,B)$ must be the closest dominant pattern to $\mathbf f$
and $X=\|\mathbf\Gamma\|$.
Conclusion~\ref{itemdefectlarge} therefore follows from~\eqref{eqGammaleq}
 and~\eqref{eqminY} and~\ref{itemdefect} follows by taking 
 $t=1/(4d)$ in ~\eqref{eqGammaleq}
 and~\eqref{eqminY} for each $(A,B)\in \cD_{\lam}(H)$. 

We now turn to \ref{itembal}.
By Theorem~\ref{lempolyLD} again
we have that if $s\ge 10C\cdot d^2\delta^d$,
then
\begin{align}\label{eqGammaplus}
\P\left(\| \mathbf  \Gamma^{+} \| \le  s\frac{|\cO|}{4} \;\middle\vert\; \mathbf D = (A,B)\right)\geq 1 - e^{ - \Omega(sm^n/n^2)}     \, .
\end{align}
For $k\in A$, by~\eqref{eqYZ} and Chernoff's bound (applied to $Z_k$), we also have that
\begin{align}\label{eqYk}
\P\left(\left|\frac{Y_k}{|\cO|}-\frac{\lam_k}{\lam_A}\right|\leq s \;\middle\vert\; \| \mathbf  \Gamma^{+} \| \le  s\frac{|\cO|}{4}, \mathbf D = (A,B)\right)\geq1- e^{-\Omega(s^2m^n)}\, .
\end{align}
The analogous statement holds for $k\in B$. 
conclusion \ref{itembal} follows from 
~\eqref{eqGammaplus},~\eqref{eqYk} and a union bound.
\end{proof}

\section{Capturing by polymer models}\label{seccapture}
In this section we prove Lemmas~\ref{lemH01} and~\ref{lemH21}, 
which together show that almost all elements of $\hom(G,H)$ are captured by 
precisely $1$ polymer model.

In Section~\ref{secover}, we showed how
Lemmas~\ref{lemH01} and~\ref{lemH21} imply 
both
Lemma~\ref{lempolymerapprox} and our main result
Theorem~\ref{mainTV}.
This will therefore conclude the proof of 
these results. We also showed how Lemmas~\ref{lemH01},~\ref{lemH21} 
and~\ref{lemexpL11} imply Theorem~\ref{cormain2} which provides asymptotic formulae 
for the partition function $Z_G^H(\lam)$. 
Since we established Lemma~\ref{lemexpL11} by proving Lemma~\ref{lemexpL1}, 
this will also conclude the proof of Theorem~\ref{cormain2} (a refinement of Theorem~\ref{cormain}).

We begin by recalling some terminology from
Sections~\ref{secHcol} and~\ref{secover}.
Recall from Definition~\ref{defpoly} 
that a polymer is a $G^2$-connected set $\gam$
such that
$|N(\gam\cap\cO)|, |N(\gam\cap\cE)|<(1-\alpha)m^n/2$
where $\alpha=\alpha(H,\lam)$ will be specified later in the section
(see~\eqref{eqadef}).
We say a colouring $f\in\hom(G,H)$ 
is {captured}
by $(A,B)\in\cD_{\lam}(H)$
if each of the $G^2$-connected components of
$(f^{-1}(A^c)\cap \cO)\cup(f^{-1}(B^c)\cap \cE)$  
is a polymer (cf.\  Definition~\ref{defcap}).
$\hom_0(G,H)$, $\hom_1(G,H)$, $\hom_2(G,H)$ denote the sets of all colourings
which are captured by $0$, precisely $1$, and $\geq2$ dominant patterns respectively.

We first prove Lemma~\ref{lemH21} which
we restate for convenience.

\begin{lemma}\label{lemH2}
 \begin{align}
Z_2\bydef \sum_{f\in\hom_2(G,H)}\prod_{v\in V}\lam_{f(v)}\le e^{-\Omega(m^n/n^2)}
\tilde Z_G^H(\lam)\, .
\end{align}
\end{lemma}

\begin{proof}
As before,
let $\mathbf{f}$ denote a random element of $\hom(H,G)$ 
selected according to $\hat \mu_{H,\lam}$.
Let $\mathbf{D}$ denote the random dominant pattern 
selected at Step \ref{step1} in the definition of
 $\hat \mu_{H,\lam}$ (Definition~\ref{defmuhat}).

Let us fix two distinct dominant patterns $(A,B)$ and $(C,D)$
(we may assume two dominant patterns exist else the conclusion of the lemma holds trivially).
Let $\cF\subset \hom(H,G)$ denote the subset of colourings 
that are captured by both $(A,B)$ and $(C,D)$.
There exists a constant $s>0$ (depending only on $(H,\lam)$)
such that a colouring cannot be $(s,\lam)$-balanced 
with respect to both $(A,B)$ and $(C,D)$.
We may therefore partition $\cF$ as 
$\cF=\cF_{A,B}\cup \cF_{C,D}$
where $\cF_{A,B}$ consists of those elements of $\cF$
that are \emph{not} $(s,\lam)$-balanced
with respect to $(A,B)$
(and similarly for $\cF_{C,D}$).

It follows from Lemma~\ref{eqpcount}~\ref{itemD} and Lemma~\ref{lembalanced}~\ref{itembal}  that 
\[
\P(\mathbf{f}\in\cF_{A,B} | \mathbf{D}=(A,B)) =
 \frac{\sum_{f\in \cF_{A,B}}\prod_{v\in V} \lam_{f(v)}}{\eta_\lam(H)^{m^n/2} \cdot \Xi_{A,B}} \le
 e^{-\Omega(m^n/n^2)}\, .
\]

Since there are only a constant number of pairs of dominant patterns we have
\[
\sum_{f\in\hom_2(G,H)}\prod_{v\in V}\lam_{f(v)}\le e^{-\Omega(m^n/n^2)}
\eta_\lam(H)^{m^n/2}
\sum_{(A,B)\in \cD_\lam(H)}\Xi_{A,B}\, .
\]
\end{proof}

We now prove Lemma~\ref{lemH01}
which is a more 
cumbersome task. 
To help in this endeavour we adapt an artful entropy argument
of Engbers and Galvin \cite{engbers2012h}.
We note that the proof of Lemma~\ref{lemH01} 
does not rely on the convergence of the cluster expansion,
and may therefore be considered as independent of Sections~\ref{seciso}--\ref{secLD}.
Again we restate the lemma. 

\begin{lemma}\label{lemH0}
There exists $\zeta=\zeta(H,\lam)<\eta_\lam(H)$ so that 
\begin{align}
Z_0\bydef \sum_{f\in\hom_0(G,H)}\prod_{v\in V}\lam_{f(v)}
\le \zeta^{m^n/2}\, .
\end{align}
\end{lemma}
Before giving the proof
we require some preliminaries and notation. 
Let 
\[
V^\ast = \{x\in V: x_n=0, x\in \cE\}\, .
\]
For $v\in V$, and $0\le i \le m-1$, let 
\[
v^i\bydef  v+(0, \ldots, 0, i)\, 
\]
and let 
\[
C(v)\bydef  \{v^0,\, \ldots, v^{m-1}\}\, .
\]
For $v\in V$, set 
\[
M_v \bydef  N(v)\bs \{v^1, v^{m-1}\}\, 
\]
and 
\[
M_{C(v)}\bydef  \bigcup_{u\in C(v)} M_u\, .
\]
We note that the subgraph of $\Z_m^n$ induced
by $M_{C(v)}$ is a disjoint union of 
$2n-2$ cycles of length $m$ (when $m\ge 4$)
or $n-1$ disjoint edges (when $m=2$). 

To carry out a judicious application of Shearer's Lemma,
Engbers and Galvin construct a partial order on $V$
with certain desirable properties.
We record one of these properties here in the 
form of a lemma.

\begin{lemma}[Engbers and Galvin \cite{engbers2012h}]\label{lemporder}
There exists a partial order $\prec$ on $V$
such that for all $v\in V^\ast$
\[
M_{C(v)}\subset \{u: u \prec C(v)\}.
\]
\end{lemma}

We are now in a position to prove our lemma.

\begin{proof}[Proof of Lemma~\ref{lemH0}]
We proceed with an entropy argument and blowup trick similar
to that used in the proof of Lemma~\ref{lemfweight}.
First for each $i\in [q]$, let us fix a sequence of rationals $(\lam_{i,k})_{k\in\N}$
such that $\lam_{i,k}\to \lam_i$. For each $k$ we may write
\[
(\lam_{1,k},\ldots, \lam_{q,k})=(a_{1,k}/b_k,\ldots, a_{q,k}/b_k)\, ,
\]
where $a_{i,k}, b_k\in \N$.
Let $H_k\bydef H(a_{1,k},\ldots,a_{q,k})$,
that is the graph on vertex set 
$U_{1,k}\cup\ldots\cup U_{q,k}$
where the $U_{i,k}$ are disjoint, 
$|U_{i,k}|=a_{i,k}$ for all $i$ and
$x\sim y$ if and only if
$x\in U_{i,k}$, $y\in U_{j,k}$
for some pair $\{i,j\}\in E(H)$.
For a vertex $v\in V(H_k)$,
we let $\overline v$ denote the unique $i\in [q]$
for which $v\in U_{i,k}$.
For $Y\subset V(H_k)$, 
we let $\overline Y\bydef \{\overline y: y\in Y\}$.
For convenience let us set $H_0\bydef H$.

For a vertex $v\in V^\ast$ and colouring $f\in\hom(G,H_k)$,
let 
\[
P_v(f)\bydef (\overline {f(M_{v^0})}, \ldots, \overline{f(M_{v^{m-1}})})\, .
\]
We write $(A_0, \ldots, A_{m-1})$ to indicate 
tuples of sets $A_i \subset V(H)$.
Write $\text{alt}(A,B)$ to denote the alternating tuple 
$(A,B,\ldots, A,B)$.
We say that $\cA=(A_0, \ldots, A_{m-1})$ is \emph{ideal}
if  $\cA=\text{alt}(A,B)$ 
for some $(A,B)\in \cD_\lam(H)$.
For $f\in \hom (G, H_k)$,
 we call a vertex $u\in V^\ast$ \emph{ideal} (wrt $f$)
 if $P_v(f)$ is ideal.
  For $X\subset V^\ast$, let 
 $\hom_X(G,H_k)$ denote the set 
 of all $f\in \hom(G, H_k)$ for which 
 the set of ideal vertices of $V^\ast$ is precisely $X$.
 Note that
 \begin{align}\label{eqhomxlam}
 \sum_{f\in\hom_X(G,H)} \prod_{v\in V}\lam_{f(v),k}=
 |\hom_X(G,H_k)|/b_k^{m^n}\, .
 \end{align}
 
 For $\beta\in (0,1]$, let
\begin{align}
 \hom^\beta(G,H_k)\bydef \bigcup_{X\in \binom{V^\ast}{\le \beta |V^\ast|}} \hom_X(G,H_k)\, ,
 \end{align}
 that is $\hom^\beta(G,H_k)$ is the set
 of all $f\in\hom(G,H_k)$ for which $\le\beta |V^\ast|$
 vertices of $V^\ast$ are ideal. 
 
Establishing the following claim will complete the proof. 
 \begin{claim}\label{clzeta}
 There exists $\zeta=\zeta(H,\lam)<\eta_\lam(H)$ and
  $\beta=\beta(H,\lam)>0$ such that 
 for $n$ sufficiently large
 \[
\sum_{f\in \hom^\beta(G,H)}\prod_{v\in V}\lam_{f(v)}\le 
\zeta^{m^n/2}\, .
\]
 \end{claim}
 
Indeed if $f\in \hom(G,H)$ has $>\beta |V^\ast|$ ideal vertices in $V^\ast$,
then by pigeonholing we can find a set $V^{\ast\ast}\subset V^\ast$ 
such that 
$|V^{\ast\ast}|\ge4^{-q}\beta |V^\ast|$ and 
$P_v(f)=\text{alt}(A,B)$ for all $v\in V^{\ast\ast} $
 for some fixed $(A,B)\in \cD_\lam(H)$.
 Letting $U=\bigcup_{v\in V^{\ast\ast} } C(v)$,
 we have 
 $|U\cap\cO|=|U\cap\cE|=m|V^{\ast\ast}|/2\ge 2^{-2q-1}\beta m^n/2 $.
 Moreover $f_{U^+}$ agrees with the pattern $(A,B)$.
 We set 
 \begin{align}\label{eqadef}
 \alpha=\alpha(H,\lam)\bydef 2^{-2q-1}\beta\, 
 \end{align}
 and declare this to be 
 the parameter appearing in Definition~\ref{defpoly}
 of a polymer
 (a polymer is a $G^2$-connected subset $\gam\subset V$
 such that $|N(\gam\cap\cE)|, |N(\gam\cap\cO))|<(1-\alpha)m^n/2$). 
 
As in Definition~\ref{defcap}, let $S(f)=(f^{-1}(A^c)\cap \cO)\cup(f^{-1}(B^c)\cap \cE)$, the set of vertices at which $f$ disagrees with $(A,B)$.
Since $f_{U^+}$ agrees with the pattern $(A,B)$,
we have $S(f)\cap U^+=\emptyset$ and so 
$S(f)^+\cap U=\emptyset$. 
By the choice of $\alpha$,
both $U\cap\cO$ and $U\cap\cE$ have size at least 
$\alpha m^n/2$ and so each of the $G^2$-connected components of
$S(f)$ is a polymer i.e.\ $f$ is captured by $(A,B)$. 
We conclude that 
 $\hom_0(G,H)\subset \hom^\beta(G,H)$ and
the lemma follows. It remains to prove Claim~\ref{clzeta}.
 
 \begin{proof}[Proof of Claim~\ref{clzeta}]
Fix $X\in \binom{V^\ast}{\le \beta |V^\ast|}$ and 
choose $f\in \hom_X(G,H_k)$ uniformly at random.
Following \cite{engbers2012h}
we will upper bound $H(f)$ by Shearer's Lemma 
(Lemma~\ref{lemShearer}). 
We take as our covering family 
$\{M_{C(v)}: v \in V^\ast\}$
together with $(1+\ind_{\{m>2\}})(n-1)$ 
copies of $C(v)$ for each
$v\in V^\ast$. 
Each vertex of $G$ is covered 
$(1+\ind_{\{m>2\}})(n-1)$ times 
and so by Lemmas~\ref{lemShearer} and \ref{lemporder} 
we have
\begin{align}\label{eqgalvinshearer}
H(f)\le 
\sum_{v\in V^\ast} H(f_{C(v)} | f_{M_{C(v)}})
+
\left(
\frac{1+\ind_{\{m=2\}}}{2n-2}
\right)
\sum_{v\in V^\ast} H(f_{M_{C(v)}})\, .
\end{align}

For a tuple $\cA=(A_0,\ldots, A_{m-1})$
where each $A_i\subset V(H)$,
we let $z_k(\cA)$
denote the number of ways of choosing
$(x_0,\ldots, x_{m-1})$
with $x_i\in U_{y_i,k}$
where
$y_i\in A_i$ 
 for each $i$
and $y_0\sim \ldots \sim y_{m-1}\sim y_0$.
Let $z(\cA)$ 
denote the sum of $\prod_{i=0}^{m-1}\lam_{x_i}$
over all tuples $(x_0,\ldots, x_{m-1})$
with $x_i\in A_i$ for each $i$
and with $x_0\sim \ldots \sim x_{m-1}\sim x_0$.
Observe that as $k\to\infty$
\begin{align}\label{eqzlim}
z_k(\cA)/b_k^m \to z(\cA)
\end{align}
for any tuple $\cA$.

We let $n(\cA)$ denote the tuple 
$(n(A_0),\ldots, n(A_{m-1}))$
where we recall that for $A\subset V(H)$, $n(A)$
denotes the common neighbourhood of $A$ in $H$.
We bound the terms in the first sum of~\eqref{eqgalvinshearer} as follows:

\begin{align}
 H(f_{C(v)} | f_{M_{C(v)}})
 &\le
  H(f_{C(v)} | P_v(f))\\
  &\le 
  \sum_{\cA}
 \P[P_v(f)= \cA] H(f_{C(v)} | P_v(f)= \cA)\\
 &\le
  \sum_{\cA}
 \P[P_v(f)=\cA] \log z_k(n(\cA))\, , \label{eqHfCv}
\end{align}
where the sum is over all tuples $\cA=(A_0,\ldots, A_{m-1})$
where each $A_i\subset V(H)$.
 The final inequality follows from the fact that 
$C(v)$ induces a cycle of length $m$ in $G$
and so if $P_v(f)= \cA$ we must have
for $i\in\{0,\ldots,m-1\}$,
$f(v^i)\in  U_{y_i,k} $ for some $y_i \in n(A_i)$ 
 where the $y_i$
form a cycle in $H$.

To bound the terms in the second sum of \eqref{eqgalvinshearer} we write

\begin{align}
H(f_{M_{C(v)}}) 
 &=
  H(f_{M_{C(v)}}| P_v(f))+ H(P_v(f))\\
  &\le 
 \left(
 \frac{2n-2}{1+\ind_{\{m=2\}}}
\right)\sum_{\cA}
 \P[P_v(f)= \cA] \log z_k(\cA)
+ qm\, . \label{eqHMCv}
\end{align}
The inequality follows from the fact that 
$M_{C(v)}$ induces $2n-2$ cycles (when $m\ge 4$)
or $n-1$ disjoint edges (when $m=2$) and
that there are at most $2^{qm}$
possible values that $P_v(f)$ can take. 
Summing the bounds \eqref{eqHfCv} and \eqref{eqHMCv} 
gives 
\begin{align}\label{eqHfsum}
H(f_{C(v)} | f_{M_{C(v)}})
+
\left(
\frac{1+\ind_{\{m=2\}}}{2n-2}
\right)
 H(f_{M_{C(v)}})
\le  
\sum_{\cA}
 \P[P_v(f)= \cA]
\log(z_k(\cA)&z_k(n(\cA))
\\&+
\frac{qm }{n-1}\, . 
\end{align}
Summing this inequality
over $v\in V^\ast$, 
recalling that the vertices of 
$X\subset V^\ast$ are ideal
and the vertices of $V^\ast\bs X$
are not,
and using \eqref{eqgalvinshearer}
we obtain
\begin{align}
&H(f)\\ 
&\le 
|X|
 \max_{ {\cA}}
\log(z_k(\cA)z_k(n(\cA))
+ 
(|V^\ast|-|X|)
 \max_{ \substack{\cA\\ \text{not ideal}}}
\log(z_k(\cA)z_k(n(\cA)) )
+
\frac{qm }{n-1} |V^\ast|\\
&\le
\frac{m^{n-1}}{2}
\left[
\beta
 \max_{\cA}
\log(z_k(\cA)z_k(n(\cA))
+ 
(1-\beta)
 \max_{ \substack{\cA\\ \text{not ideal}}}
\log(z_k(\cA)z_k(n(\cA)) )
+
\frac{qm}{n-1}
\right]\, .
\end{align}
Subtracting $m^n\log b_k$ from both sides of the 
above inequality, taking the limit $k\to\infty$
and using~\eqref{eqhomxlam} and~\eqref{eqzlim}
gives
\begin{align}\label{eqhomx}
&\log \left(
\sum_{f\in \hom_X(G,H)}\prod_{v\in V}\lam_{f(v)}
\right)\\
&\le
\frac{m^{n-1}}{2}
\left[
\beta
 \max_{\cA}
\log(z(\cA)z(n(\cA))
+ 
(1-\beta)
 \max_{ \substack{\cA\\ \text{not ideal}}}
\log(z(\cA)z(n(\cA)) 
+
\frac{qm}{n-1}
\right]\, . 
\end{align}

Now, 
for $\cA=(A_0,\ldots,A_{m-1})$
we have the inequality
\[
z(\cA)\le \prod_{i=0}^{m-1}\lam_{A_i}
\]
and so
\begin{align}\label{eqprodbd}
z(\cA)z(n(\cA))\le \prod_{i=0}^{m-1}\lam_{A_i}\lam_{n(A_i)}\, .
\end{align}
Observe that for $A\subset V(H)$ we have 
$\lam_{A}\lam_{n(A)}\le \eta_\lam(H)$
with equality if and only if $(A,n(A))$ is a dominant pattern.
Returning to~\eqref{eqprodbd} we see that
\begin{align}\label{eqmaxcA}
\max_{\cA}
z(\cA)z(n(\cA))
\le
\eta_\lam(H)^m\, ,
\end{align}
with equality only if 
$(A_i, n(A_i))\in \cD_\lam(H)$
for each $i$.
We will now show that 
if $\cA$ is not ideal then 
\begin{align}\label{eqzzsmall}
z(\cA)z(n(\cA))
<
\eta_\lam(H)^m\, .
\end{align}

By \eqref{eqmaxcA} we may assume that
$\cA=(A_0,\ldots, A_{m-1})$ where
$(A_i, n(A_i))\in \cD_\lam(H)$ 
for all $i$. 
Observe that if $(A,B)\in \cD_\lam(H)$
then 
\begin{align}\label{eqnAB}
n(B)=A \text{ and } n(A)=B\, .
\end{align}
Indeed it is clear that $A\subset n(B)$
and the inclusion cannot be strict else
$(n(B), B)$ is a pattern with $\lam_A\lam_B< \lam_{n(B)}\lam_B$ 
contradicting the fact that $(A,B)$ is dominant. 
The equality $n(A)=B$ follows similarly.

Since $\cA$ is not ideal,
we may assume without 
loss of generality 
that $(A_0, A_1)$ is not a dominant pattern 
and so
$A_1\neq n(A_0)$.
For $A,B\subset V(H)$, let $p(A,B)$
be the set of pairs $(a,b)\in A\times B$
such that $a \nsim b$.
We have
\[
z(\cA)
\le
\left(
\lam_{A_0} \lam_{A_1}-\sum_{\{x,y\}\in p(A_0,A_1)}\lam_x\lam_y
\right)
 \prod_{i=2}^{m-1}\lam_{A_i}
\]
and 
\[
z(n(\cA))
\le
\left(
\lam_{n(A_0)} \lam_{n(A_1)}-\sum_{\{x,y\}\in p(n(A_0),n(A_1))}\lam_x\lam_y
\right)
 \prod_{i=2}^{m-1}\lam_{n(A_i)}\, .
\]

If both $ p(A_0,A_1)$ and $p(n(A_0),n(A_1))$ are empty
then $A_0\sim A_1$ and $n(A_0)\sim n(A_1)$
hence $A_1\subset n(A_0)$ and $n(A_0)\subset n(n(A_1))$.
Since $(A_1, n(A_1))$ is a dominant pattern 
we have $A_1= n(n(A_1))$ by~\eqref{eqnAB}
and so $n(A_0)=A_1$
contrary to assumption.
We conclude that one of 
$ p(A_0,A_1)$ and $p(n(A_0),n(A_1))$
is non-empty and so $z(\cA)z(n(\cA))<\eta_\lam(H)^m$ which proves~\eqref{eqzzsmall}. 
Letting 
\[
 \eta'\bydef \left(\max_{ \substack{\cA\\ \text{not ideal}}}
z(\cA)z(n(\cA))\right)^{1/m}<\eta_\lam(H)
\]
and returning to \eqref{eqhomx}
we have
\begin{align}\label{eqhomxbd}
\log\left(
\sum_{f\in \hom_X(G,H)}\prod_{v\in V}\lam_{f(v)}
\right)
&\le
\frac{m^{n}}{2}
\left[
\beta
\log \eta_\lam(H)
+ 
(1-\beta)
\log \eta'
+
\frac{q}{n-1}
\right]\, . 
\end{align}

It follows from \eqref{eqhomxbd} that
\begin{align}
\log \left(
\sum_{f\in \hom^\beta(G,H)}\prod_{v\in V}\lam_{f(v)}
\right)
\le
\frac{m^{n}}{2}
\left[
\beta
\log \eta_\lam(H)
+ 
(1-\beta)
\log \eta'
+
\frac{H(\beta)}{m}
+
\frac{q}{n-1}
\right]\, . 
\end{align}
The claim follows by choosing $\beta=\beta(H,\lam)$ sufficiently small and $n$
sufficiently large. 
\end{proof}
\end{proof}

\section{A structure theorem for $\hom(\Z_m^n,H)$}\label{secgenstruc}
In this section we show how the results proved thus far can be used to prove Theorem~\ref{thmgenstruc0} and
resolve conjectures of Engbers and Galvin \cite[Conjectures 6.1, 6.2, 6.3]{engbers2012h}.

Recall that
for $s>0$ and a weighted graph $(H,\lam)$,
we say that an $H$-colouring of $\Z_m^n$ is
\emph{$(s,\lam)$-balanced} with respect to $(A,B)$
 if for each
$k\in A$ ,
the proportion of vertices of $\cE$ 
coloured $k$ is within $s$
of $\lam_k/\lam_A$ and
for each
$\ell \in B$,
the proportion of vertices of $\cO$ 
coloured $\ell$ is within $s$
of $\lam_\ell/\lam_B$.

Theorem~\ref{thmgenstruc0} was inspired
by the following theorem of Engbers and Galvin.

\begin{theorem}[Engbers and Galvin~\cite{engbers2012h}]
\label{galvindecomp}
Fix an even integer $m\ge2$ and a weighted graph $(H,\lam)$ 
where $\lam$ takes only rational values. 
There is a partition of $\hom(\Z_m^n,H)$ into
$|\cD_\lam(H)|+1$ classes
\[
\hom(\Z_m^n,H)= 
F(0)\cup
\bigcup_{(A,B)\in \cD_\lam(H)}F(A,B)
\]
with the following properties (we write $\mu$ for $\mu_{H,\lam}$).
\begin{enumerate}
\item 
$\mu(F(0)) \le e^{-\Omega(n)}$.
\item 
For $(A,B)\in\cD_\lam(H)$, 
each
$f\in F(A,B)$ is $(e^{-\Omega(n)},\lam)$-balanced
w.r.t. $(A,B)$.
\item \label{itemmeas0}
If $A \neq B$ is such that 
$(A,B), (B,A)\in \cD_\lam(H)$ then
\[
\mu(F(A,B))=\mu(F(B,A))\left(1\pm e^{-\Omega (n)}\right)\, .
\]
\item \label{itemcond0}
If $(A,B), (\tilde A, \tilde B)\in \cD_\lam(H)$ are such that
$\varphi(A)=\tilde A$ and $\varphi(B)=\tilde B$ for some
weight preserving automorphism $\varphi$ of $(H,\lam)$, then 
\[
\mu(F(A,B))=\mu(F(\tilde A,\tilde B))\left(1\pm e^{-\Omega (n)}\right)\, .
\]
\item
For each $(A,B)\in \cD_\lam(H)$, $x\in \cO$, $y\in \cE$, $k\in A$ and $\ell \in B$,
\[
\P_\mu(f(x)=k | f\in F(A,B))=\frac{\lam_k}{\lam_A}\left(1\pm e^{-\Omega (n)}\right)
\]
and
\[
\P_\mu(f(y)=\ell | f\in F(A,B))=\frac{\lam_\ell}{\lam_B}\left(1\pm e^{-\Omega (n)}\right)\, .
\]

\end{enumerate}
\end{theorem}

 Engbers and Galvin note that Theorem~\ref{galvindecomp} 
 does not make a general statement about the
 measures of the sets $F(A,B)$,
 however they conjecture an explicit 
 formula for the asymptotics of 
 $\ln \mu(F(A,B))$  \cite[Conjectures 6.1, 6.2]{engbers2012h}.
They go on to conjecture \cite[Conjecture 6.3]{engbers2012h}
that there is a decomposition of $\hom(\Z_m^n,H)$ satisfying
the conclusions of Theorem~\ref{galvindecomp} such that 
$\mu_{H,\lam}(F(0))\le e^{-m^n/p(n)}$
for some polynomial $p(n)$ whose degree depends only on $(H,\lam)$ and $m$.
We remark that although Theorem~\ref{galvindecomp} 
includes the restriction that $\lam$ takes rational values,
both of the aforementioned conjectures are made for an arbitrary 
weighted graph $(H,\lam)$. 
The following theorem resolves these conjectures in a strong form.

Recall the definition of $\tilde Z_G^H(\lam)$
from Lemma~\ref{lempolymerapprox}.
We have the following strengthening of Theorem~\ref{thmgenstruc0}.

\begin{theorem}
\label{thmgenstruc}
Fix a weighted graph $(H,\lam)$ and $m\ge2$ even. 
There exists $\xi\in(0,1)$ 
such that if $\xi^d \leq s \leq 1$ then
there is a partition of $\hom(\Z_m^n,H)$ into
$|\cD_\lam(H)|+1$ classes
\[
\hom(\Z_m^n,H)= 
F(0)\cup
\bigcup_{(A,B)\in \cD_\lam(H)}F(A,B)
\]
with the following properties. With $\mu=\mu_{H,\lam}$ and $t\bydef \min\{s^2, s/n^2\},$
\begin{enumerate}
\item \label{itemzero}
$\mu(F(0)) \le e^{-\Omega(tm^n)}$.
\item \label{itembalanced}
For $(A,B)\in \cD_\lam(H)$, each
$f\in F(A,B)$ is $(s,\lam)$-balanced
with respect to $(A,B)$.

\item \label{itemmeas}
For each $(A,B)\in \cD_\lam(H)$,
\[
\mu(F(A,B))= 
\frac{1}{ Z_G^H(\lam)} 
 \eta_\lam(H)^{\frac{m^n}{2}} 
 \Xi_{A,B}\left(1\pm e^{-\Omega(tm^n)}\right)\, .
\]
\item\label{itemcond}
For each $(A,B)\in \cD_\lam(H)$, $x\in \cO$, $y\in \cE$, $k\in A$ and $\ell \in B$,
\[
\P_{\mu}(f(x)=k | f\in F(A,B))=\frac{\lam_k}{\lam_A}\left(1\pm e^{-\Omega(n)}\right)
\]
and
\[
\P_{\mu}(f(y)=\ell | f\in F(A,B))=\frac{\lam_\ell}{\lam_B}\left(1\pm e^{-\Omega(n)}\right)\, .
\]
\end{enumerate}
\end{theorem}

\begin{proof}
Let $\mu, \hat \mu$ denote $\mu_{H,\lam}, \hat \mu_{H,\lam}$ respectively. 
Recall that $\hom_0(G,H)$, $\hom_1(G,H)$, $\hom_2(G,H)$ denote the sets of all colourings
which are captured (cf.\ Definition~\ref{defcap}) by $0$, precisely $1$, and $\geq2$ dominant patterns respectively.
Let $E(0)=\hom_0(G,H)\cup \hom_2(G,H)$
and note that by Lemmas~\ref{lemH2} and \ref{lemH0},
\begin{align}\label{eqE0}
\hat\mu(E(0))\le e^{-\Omega(m^n/n^2)}\, .
\end{align}
We now consider the partition
\[
\hom_1(G,H)=\bigcup_{(A,B)\in \cD_\lam(H)}E(A,B)\, ,
\]
where $E(A,B)$ denotes the set of colourings $f\in \hom_1(G,H)$
which are captured by the $(A,B)$ polymer model.

As in Lemma~\ref{lemH2}, we let
\[
Z_2=\sum_{f\in\hom_2(G,H)}\prod_{v\in V}\lam_{f(v)}\, .
\]
Fix $(A,B)\in \cD_{\lam}(H)$.
Since each element of $E(A,B)$ is captured by the $(A,B)$ polymer model
and no others we have by Lemma~\ref{eqpcount}~\ref{itemf} that 
\begin{align}\label{eqEmeas}
\frac{1}{\tilde Z_G^H(\lam)}\left(\eta_\lam(H)^{m^n/2} \Xi_{A,B}-4^q Z_2\right)\leq
\hat\mu(E (A,B))
\le\frac{1}{\tilde Z_G^H(\lam)}\eta_\lam(H)^{m^n/2} \Xi_{A,B}\, .\quad\quad
\end{align}
Here $4^q$ is used as a crude upper bound for $|\cD_{\lam}(H)|$.

Let
\[
\rho\bydef \max_{(A,B)\in \cD_{\lam}(H)} \delta_{A,B}\, ,
\]
with $\delta_{A,B}$ as defined in~\eqref{eqdeltadef}.
Note that $\rho<1$ since $\delta_{A,B}<1$ for each $(A,B)$ (Lemma~\ref{lemdelta1}).
Pick 
\begin{align}\label{eqxibd}
\max\{m^{-1/4}, \rho^{1/2}\}< \xi <1\, .
\end{align}
Suppose that $s \geq \xi^d$.
Let us fix $(A,B)\in \cD_\lam(H)$
and let $F(A,B)$ denote the set of 
elements of $E(A,B)$ that are $(s,\lam)$-balanced 
with respect to $(A,B)$.
Let $t\bydef \min\{s^2, s/n^2\}$.
By Lemma~\ref{lembalanced}~\ref{itembal} we have
\begin{align}\label{eqEbsF}
\hat \mu (E(A,B)\bs F(A,B))
\le e^{-\Omega(tm^n)}\, .
\end{align}
Let
$F(0)= E(0)\cup \bigcup_{(A,B)\in \cD_\lam(H)}(E(A,B)\bs F(A,B))$, and consider the partition  
\[
\hom(G,H)=  F(0)\cup \bigcup_{(A,B)\in \cD_\lam(H)} F(A,B)\, .
\]
By~\eqref{eqE0},\eqref{eqEbsF} 
and Theorem~\ref{mainTV},
the above partition
satisfies conclusions  \ref{itemzero} and \ref{itembalanced} of our theorem.
It remains to verify conclusions \ref{itemmeas} and \ref{itemcond}. 
By~\eqref{eqEmeas},~\eqref{eqEbsF} and Lemma~\ref{lemH2},
\begin{align}\label{eqmuF}
\hat\mu(F (A,B))=\left(1\pm e^{-\Omega(tm^n)}\right)\frac{1}{\tilde Z_G^H(\lam)}\eta_\lam(H)^{m^n/2} \Xi_{A,B}\, .
\end{align}
By Theorem~\ref{mainTV}, the same holds for $\mu(F (A,B))$
and so, applying Lemma~\ref{lempolymerapprox}, we obtain conclusion \ref{itemmeas}.
Before we turn to conclusion  \ref{itemcond},
we use~\eqref{eqmuF} to prove the following simple lower bound 
on $\hat\mu(F (A,B))$.
\begin{claim}\label{claimmubd}
\begin{align}\label{eqFABlb}
\hat\mu(F (A,B))\ge e^{-O(m^n \rho^d + q)}\, .
\end{align}
\end{claim}
\begin{proof}[Proof of Claim~\ref{claimmubd}]
By Lemma~\ref{lemexpL1} and the cluster expansion,
\begin{align}\label{eqclusterXi}
\ln \Xi_{A,B} = \sum_{j=1}^{\infty} L_{A,B}(j) = O(m^n \delta_{A,B}^d)\, .
\end{align}

Using the crude bound $|\cD_{\lam}(H)|\leq 4^q$
we have
\[
\tilde Z_G^H(\lam)\bydef \eta_\lam(H)^{m^n/2}
\sum_{(A,B)\in\cD_\lam(H)}\Xi_{A,B}\leq 4^q  \eta_\lam(H)^{m^n/2} 
e^{O(m^n \rho^d)}\, .
\]
The claim follows from~\eqref{eqmuF}, noting that $\Xi_{A,B}\geq1$
and $tm^n\to\infty$ by the definition of $t$.
\end{proof}

We now turn to conclusion \ref{itemcond}.
As before, let $\mathbf D$ denote the random pattern 
selected at Step \ref{step1} in the definition of $\hat \mu$ 
(Definition~\ref{defmuhat}).
First we bound the probability 
\[
\P_{\hat \mu}(f(x)=k \mid  \mathbf D=(A,B))\, .
\]
Let 
$\mathbf \Gamma$ be the random polymer configuration
selected at Step \ref{step2} in the definition of $\hat \mu$ 
conditioned on the event $\mathbf D=(A,B)$.
In particular $\mathbf \Gamma$ has 
distribution $\nu_{A,B}$.
Let $\mathbf\Gamma^{+}$ denote the union $\bigcup_{\gamma\in \mathbf\Gamma}\gamma^{+}$.
By symmetry, $\P_{\hat \mu}(v\in \mathbf\Gamma^{+})$ is the same for each $v\in \cO$.
Let us denote this probability by $p_{\cO}$ and define $p_\cE$ similarly.
Then by Lemma~\ref{lemexpL1} and~\eqref{eqcumulant} we have

\[
(p_{\cO}+p_{\cE})m^n/2=\E \left |\mathbf\Gamma^{+}\right|\leq (d+1) \E \left |\mathbf\Gamma\right| = O (d m^n \delta_{A,B}^d)\, ,
\]
and so 
\[
p_{\cO}+p_{\cE}= e^{-\Omega(n)}\, .
\]
It follows that 
\begin{align}
\P_{\hat \mu} \left(f(x)=k \mid \mathbf D=(A,B)\right)
&=
\P_{\hat \mu} \left(f(x)=k \mid \mathbf D=(A,B), x\notin \mathbf \Gamma^{+} \right)+ e^{-\Omega(n)}\\
&= \lam_k/\lam_A + e^{-\Omega(n)}\, , \label{eqlamk}
\end{align}
where for the final equality we used that if $x\notin \mathbf \Gamma^{+}$, 
then $x$ is coloured at Step \ref{step4}  in the definition of $\hat \mu$.

By~\eqref{eqE0}~and~\eqref{eqEbsF}  the symmetric difference of the events 
$\{\mathbf D =(A,B)\}$ and $\{f \in F(A,B)\}$
has $\hat \mu$-measure $e^{-\Omega(tm^n)}$.
By Claim~\ref{claimmubd},
\begin{align}
\P_{\hat\mu}(f\in F(A,B))\ge 4^{-q-1}e^{-O(m^n \rho^d)}\, .
\end{align}
By~\eqref{eqxibd} we have $tm^n\to\infty$ and also $t/\rho^d\to\infty$.
It then follows that 
 \[
 \P_{\hat \mu} \left(f(x)=k \;\middle\vert\; f\in F(A,B)\right) =  \P_{\hat \mu} \left(f(x)=k \mid \mathbf D=(A,B)\right)\left(1 \pm e^{-\Omega(tm^n)}\right)
 \]
and so by~\eqref{eqlamk}
 \[
 \P_{\hat \mu} \left(f(x)=k \;\middle\vert\; f\in F(A,B)\right) = \lam_k/\lam_A + e^{-\Omega(n)}\, .
 \]
 Finally using the fact that $\|\mu-\hat\mu\|_{TV}\le e^{-\Omega(m^n/n^2)}$
 (Theorem~\ref{mainTV}),
 conclusion \ref{itemcond} follows. 
 \end{proof}
 
 We note that if $(A,B), (\tilde A, \tilde B)\in \cD_\lam(H)$ are such that
$\varphi(A)=\tilde A$ and $\varphi(B)=\tilde B$ for some
weight preserving automorphism $\varphi$ of $(H,\lam)$, then 
$\Xi_{A,B}=\Xi_{\tilde A, \tilde B}$. 
Similarly if $(A,B)\in \cD_\lam(H)$ where $A\neq B$
then $\Xi_{A,B}=\Xi_{B,A}$.
Conclusion \ref{itemmeas} of Theorem~\ref{thmgenstruc} therefore strengthens
 conclusions  \ref{itemmeas0}  and \ref{itemcond0} of Theorem~\ref{galvindecomp}
 significantly. 
 Moreover by the cluster expansion $\ln \Xi_{A,B}=\sum_{j=1}^{\infty}L_{A,B}(j)$ and 
 Lemma~\ref{lemexpL1} we can obtain detailed explicit asymptotic formulae
 for each $\mu(F(A,B))$ in Theorem~\ref{thmgenstruc} (see Section~\ref{secAlg} for an algorithm to compute the terms $L_{A,B}(j)$). 
 
We end this section by showing that the tradeoff
in Theorem~\ref{thmgenstruc} between
the size of $\mu(F(0))$ and the degree to which the 
elements of $F(A,B)$ are balanced (see conclusions  \ref{itemzero}  and  \ref{itembalanced})
is essentially optimal. 

\begin{prop}
Let $\xi$ be as in Theorem~\ref{thmgenstruc}, $\xi^d\leq s = o(1)$, and let
\[
\hom(\Z_m^n,H)= 
F(0)\cup
\bigcup_{(A,B)\in \cD_\lam(H)}F(A,B)
\]
be a partition of $\hom(\Z_m^n,H)$ 
such that for $(A,B)\in \cD_\lam(H)$, each
$f\in F(A,B)$ is $(s,\lam)$-balanced
with respect to $(A,B)$.
Then
\[
\mu(F(0))=e^{-O(s^2 m^n)}\, .
\]
\end{prop}
\begin{proof}[Proof sketch]
Fix $(A,B)\in \cD_{\lam}(H)$, 
and let $F$ denote the set of $f\in\hom(\Z_m^n, H)$
such that $f(\cO)\subset A$, $f(\cE)\subset B$
and $f$ is \emph{not} $(s,\lam)$-balanced with respect to $(A,B)$. 
Since $s=o(1)$, 
each $f\in F$ is not $(s,\lam)$-balanced 
with respect to any dominant pattern 
(in fact we only require $s$ to be less than a sufficiently small constant) . 
Thus $F\subset F(0)$.

Consider the random experiment where we 
independently assign each $v\in\cO$ the colour $k\in A$ with probability $\lam_k/\lam_A$ and each $v\in\cE$ the colour $\ell\in B$ with probability $\lam_\ell/\lam_B$. 
By tightness of the Chernoff bound,
 \begin{align}\label{eqtightchern}
\frac{ \sum_{f\in F}\prod_{v\in V}\lam_{f(v)}}{\eta_{\lam}(H)^{m^n/2}}=e^{-O(s^2 m^n)}.
 \end{align}
As before, let $\rho=\max_{(A,B)\in\cD_{\lam}(H)}\delta_{A,B}$. 
By Theorem~\ref{cormain2} we then have
 \[
\mu(F(0))\geq\mu(F)= \frac{ \sum_{f\in F}\prod_{v\in V}\lam_{f(v)}}{Z_G^H(\lam)}=
\frac{ \sum_{f\in F}\prod_{v\in V}\lam_{f(v)}}{\eta_{\lam}(H)^{m^n/2}}\cdot e^{-O(m^n\rho^d+q)}=e^{-O(s^2 m^n)},
 \]
 where for the final equality we used~\eqref{eqtightchern} and $\xi>\max\{m^{-1/4}, \rho^{1/2}\}$ (see~\eqref{eqxibd}).
\end{proof}

\section{Torpid mixing via conductance}\label{sectorpid}
In this section we show how Theorem~\ref{thmslowmix}
can be deduced from our decomposition result Theorem~\ref{thmgenstruc}.
We will make use of a well-known {conductance argument},
using a form of the argument presented in \cite{dyer2002counting}.

Let $\cM$ be an ergodic (connected and aperiodic) Markov chain
on a finite state space $\Omega$, 
with transition probabilities
$P(\omega, \omega')$, $\omega, \omega'\in\Omega$
and stationary distribution $\pi$.

Let $X\subset \Omega$ and $Y\subset \Omega \bs X$
satisfy $\pi(X)\le 1/2$ 
and 
$P(\omega, \omega')=0$
for all $\omega\in X, \omega'\in \Omega \bs (X \cup Y)$.
Then from \cite[Claim 2.3]{dyer2002counting} we have
\begin{align}
\label{eqconductance}
\tau_\cM \geq  \frac{\pi(X)}{8 \pi(Y)}\, .
\end{align}
We will apply this bound when $\pi(Y)$
is small. 
Intuitively, $Y$ acts as a bottleneck:
for the chain to leave the set $X$ it must pass through $Y$.
If $\pi(Y)$ is small, the chain is unlikely to pass through $Y$, 
which drives up the mixing time.

\begin{proof}[Proof of Theorem~\ref{thmslowmix}]
We use the conductance argument
outlined above. 
Let $\mu=\mu_{H,\lam}$ and let
\[
\beta \bydef  \min_{(A,B)\in \cD_{\lam}(H)}\min\left\{\frac{\min_{k\in A}\lam_k}{\lam_A}, \frac{\min_{\ell \in B}\lam_\ell}{\lam_B}\right\}\, .
\]

Consider
the decomposition from Theorem~\ref{thmgenstruc}
with $s=\beta/(2q)$
\[
\hom(\Z_m^n,H)= 
F(0)\cup
\bigcup_{(A,B)\in \cD_\lam(H)}F(A,B)\, .
\]
Suppose that $(H,\lam)$ is non-bipartite and non-trivial and so has 
at least two dominant patterns.
Suppose also that $(A,B)$
is such that $\mu(F(A,B))$ 
is minimal and so in particular $\mu(F(A,B))\le 1/2$. 
Let $X=F(A,B)$ and $Y=F(0)$.
By \eqref{eqFABlb} and Theorem~\ref{mainTV}
we also have
\begin{align}\label{eqFAB}
\mu(X)\ge e^{-O(m^n \rho^d+q)}\, ,
\end{align}
where 
\(
\rho= \max_{(A,B)\in \cD_{\lam}(H)}\delta_{A,B}<1\, .
\)

Suppose now that $(C, D)$ is a dominant pattern distinct from $(A,B)$
and note that either $A\bs C\neq \emptyset$ or $B \bs D\neq \emptyset$.
Without loss of generality suppose $A\bs C\neq \emptyset$.
Let $f_1\in F(A,B), f_2 \in F(C,D)$.
By conclusion  \ref{itembalanced}  of Theorem~\ref{thmgenstruc},
$f_1$ is $(s,\lam)$-balanced with respect to $(A,B)$
and $f_2$ is $(s,\lam)$-balanced with respect to $(C,D)$. 
Letting $k\in A\bs C$ we then have
\[
|f_1^{-1}(k) \bs f_{2}^{-1}(k)| \geq \left(\beta-s-|C|s\right)\frac{m^n}{2}\geq \frac{\beta m^n}{4}
\]
and so $f_1$ and $f_2$ differ in at least $\beta m^n/4$ vertices.

Now let $\cM$ be a $(\beta/5)$-local ergodic Markov chain 
on $\hom(\Z_m^n,H)$
with transition probabilities
$P(f_1, f_2)$, $f_1, f_2\in\hom(\Z_m^n,H)$
and stationary distribution $\mu$.
By the above, $P(f_1, f_2)=0$ for all 
$f_1\in X$, $f_2\in (X\cup Y)^c$
(i.e.\  for the chain to leave $X$
it must pass through $Y$).
By conclusion  \ref{itemzero}  of Theorem~\ref{thmgenstruc}, 
the conductance bound \eqref{eqconductance} and \eqref{eqFAB}, we therefore have
\begin{align}
\tau_\cM
\ge   \frac{\mu(X)}{ 8\mu(Y)} 
= e^{\Omega(m^n/n^2)}\, .
\end{align}
The proof of the second statement of the theorem, dealing with bipartite $(H,\lam)$,
goes through with trivial modifications and so we omit it. 
\end{proof}

We end this section by remarking that the assumption of ergodicity in Theorem~\ref{thmslowmix} is a non-trivial restriction. 
We illustrate this with the example where $H=K_4$, i.e.\ $4$-colourings of $\Z_m^n$.
Let $\mathcal H$ denote the graph on vertex set $\hom(\Z_m^n,K_4)$
where two vertices $f_1, f_2$ are adjacent if and only 
$f_1(v)\neq f_2(v)$ for precisely one vertex $v\in V(\Z_m^n)$.
It is well-known that $\cH$ has isolated vertices for any $m\geq2$ (these are referred to as \emph{frozen colourings}: updating the colour of any single vertex results in a non-proper colouring). 
One can use Theorem~\ref{mainTV} to show that the component structure
of $\cH$ undergoes a phase transition at $m=256$ in the following sense:
If $m<256$ is even then $\cH$ has a giant connected component occupying $1-o(1)$ of the vertices whereas if $m>256$ the largest component of $\cH$ occupies an exponentially small fraction of the total number of vertices. If $m=256$, then the largest component in $\cH$ is polynomially small. This transition is driven by the appearance of copies of the three dimensional hypercube $Q_3$ which is $4$-coloured in such a way that no single vertex update is valid (i.e.\ a frozen colouring of $Q_3$). We omit the details for brevity.

\section{The defect distribution}\label{secdefect}
In this section we prove our central limit theorem for the distribution of the number of  vertices in a sample from the defect distributions $\nu_{A,B}$ (Theorem~\ref{thmasymdist}).

We begin with some preliminaries on cumulants of random variables. 
Recall from~\eqref{eqcumulant}, the cumulant generating function of a random variable $X$, 
$h_t(X) = \ln \E e^{tX}$. 
The $\ell$th \textit{cumulant} of $X$
 is defined by taking derivatives of $h_t(X)$ and evaluating at $0$:
\begin{align*}
\kappa_\ell(X) &= \frac{ \partial^\ell h_t(X)}{\partial t^\ell} \Bigg|_{t=0} \,.  
\end{align*}
In fact the cumulants of $X$ 
are related to the moments of $X$ 
by a non-linear change of basis (see e.g.\ \cite{leonov1959method}). 
 In particular, $\kappa_1(X) = \E X$ and $\kappa_2(X) = \var(X)$. 
 Moreover, if a random variable $X$
 has a distribution determined by its moments,
 and if for a sequence of random variables $X_n$ we have
 $\lim_{n \to \infty} \kappa_\ell(X_n) = \kappa_\ell(X)$ 
 for all $\ell \ge 1$, 
 then $X_n \overset{d}{\longrightarrow} X$
 (recall that $\overset{d}{\longrightarrow}$ denotes convergence in distribution). 
 We will use this in conjunction with the following fact.

\begin{fact}
\label{factPoisNormal}
 If $X$ has a Poisson distribution with mean $m$, then $\kappa_r(X) = m$ for all $r$. 
If $X$ has a standard normal distribution (mean $0$, variance $1$) then $\kappa_1(X) =0$, $\kappa_2(X) = 1$, and $\kappa_\ell(X)= 0$ for all $\ell \ge 3$. 
\end{fact}

We will also use the following basic property of cumulants.

\begin{fact}
\label{factadd}
If $X$ is a random variable, $a,b\in \R$ and
$\ell\ge 2$, then $\kappa_\ell(aX+b)=a^\ell \kappa_\ell(X)$.
\end{fact}

We are now in a position to prove Theorem~\ref{thmasymdist}.

\begin{proof}[Proof of Theorem~\ref{thmasymdist}]
For brevity let us denote the weight function 
$w_{A,B}$ (see~\eqref{eqwdef}) by $w$ and 
denote the partition function $\Xi_{A,B}$ (see~\eqref{eqpfdef}) by $\Xi$.
For $t\ge 0$, let us define the modified weight function
\[
w_t(\gamma) = w(\gamma)e^{t|\gamma|}\, 
\]
and let $\Xi_t$ denote the accompanying modified partition function.
We then have 
\begin{align}
h_t(\|\mathbf\Gamma\|)= \log \Xi_t - \log \Xi\, .
\end{align}
As was shown in the proof of Theorem~\ref{lempolyLD},
the cluster expansion of $\log \Xi_t$ converges absolutely
for $t\le 1/d$ (see inequality~\eqref{eqmodKP}).
\[
\log \Xi_t=\sum_{\Gamma\in\cC}w_t(\Gamma)=\sum_{\Gamma\in\cC}w(\Gamma)e^{t\|\mathbf\Gamma\|}\, .
\]

It follows that for fixed $s\in \N$
\begin{align}\label{eqcumulant}
\kappa_s(\|\mathbf\Gamma\|) &= \frac{ \partial^s \log \Xi_t}{\partial t^s} \Bigg|_{t=0}\\
&= \sum_{\Gamma\in\cC}\|\mathbf\Gamma\|^s w(\Gamma)\\
&=\sum_{k=1}^\infty k^s L_{A,B}(k)\\
&= (1+o(1))L_{A,B}(1)\, , \label{eqLABlast}
\end{align} 
where for the last inequality we used~\eqref{eqkssum}
 from Lemma~\ref{lemexpL1}.
 
Let $L_1=L_{A,B}(1)$.
If $L_1\to 0$, then
$\kappa_\ell(\|\mathbf\Gamma\|)\to 0$
for each $\ell$ so that $\|\mathbf\Gamma\|=0$ whp.
If $L_1\to \rho>0$ then 
$\kappa_\ell(\|\mathbf\Gamma\|)\to \rho$
for each $\ell$ so that $\|\mathbf\Gamma\|\overset{d}{\longrightarrow} \pois(\rho)$.
Finally suppose that $L_1\to \infty$ 
and let $Y$ denote $(\|\mathbf\Gamma\|- \E \|\mathbf\Gamma\|)/\sqrt{\var \|\mathbf\Gamma\|}$.
We have $\kappa_1(Y)=\E Y=0$.
By Fact~\ref{factadd} and~\eqref{eqLABlast}
we have, for $\ell\ge 2$,
\begin{align}
\kappa_\ell (Y)= \frac{\kappa_\ell(\|\mathbf\Gamma\|)}{(\var \|\mathbf\Gamma\|) ^{\ell/2}}
= \frac{\kappa_\ell(\|\mathbf\Gamma\|)}{\kappa_2( \|\mathbf\Gamma\|) ^{\ell/2}}
=(1+o(1))L_{A,B}(1)^{1-\ell/2}
\end{align}
so that $\kappa_2 (Y)\to 1$ and
 $\kappa_\ell (Y)\to 0$
 for $\ell\ge3$.
We conclude that $Y \overset{d}{\longrightarrow} N(0,1)$ by Fact~\ref{factPoisNormal}.
\end{proof}

\begin{proof}[Proof of Corollary~\ref{corqcol}]
We specialise to the $q$-colouring model so that
$(H,\lam)=(K_q, \dot\iota)$ where $\dot\iota\equiv 1$.
We denote $\hat\mu_{H,\lam}, \mu_{H,\lam}$ by $\hat\mu, \mu$ respectively.
 Let $\hat X$ denote the number of defect vertices in a sample from $\hat\mu$.
 Note that since $\|\mu-\hat\mu\|\le e^{-\Omega(m^n/n^2)}$ by 
Theorem~\ref{mainTV},
it suffices to establish the statement of the corollary with
 $\hat\mu$ and $\hat X$ in place of $\mu$ and $X$. 

Let $\mathbf \Gamma$ denote the random 
polymer configuration selected at Step \ref{step2} in the definition of $\hat \mu$.
By Lemma~\ref{lembalanced}~\ref{itemdefect} we have that $\hat X= \|\mathbf\Gamma\|$
with probability $\geq 1-e^{-\Omega(m^n/n^2)}$.
The result follows from Theorem~\ref{thmasymdist} and 
the symmetry of the dominant patterns in $(K_q, \dot\iota)$.
 \end{proof}

 \section{Computing terms of the cluster expansion}\label{secAlg}
In this section we give a general description of the terms in the cluster expansion 
of $\Xi_{A,B}$ and provide an algorithm for computing these terms as an explicit function of $n$. Combined with Theorem~\ref{cormain}, this provides, for fixed $(H,\lam)$ and $m\geq2$ even,
a finite time algorithm to compute an asymptotic formula (correct up to a multiplicative factor $(1+o(1))$) for the partition function $Z_G^H(\lam)$ (such as those presented in Corollary~\ref{corcurious}).
 
 For the following lemma we recall that for $(A,B)\in \cD_{\lam}(H)$
\[
L_{A,B}(k) =  \sum_{\substack{\Gamma \in \cC : \\ \|\Gamma\|= k}}  w_{A,B}(\Gamma)  
\]
and 
\begin{equation}\label{eqdeltadef2}
\delta_{A,B}=  \max\left\{\max_{v\in B^c}\frac{\lam_{N(v)\cap A}}{\lam_A},\max_{u\in A^c} \frac{\lam_{N(u)\cap B}}{\lam_B}   \right\}\, .
\end{equation}

 \begin{lemma}\label{lemLkform}
Fix a weighted graph $(H,\lam)$, a dominant pattern $(A,B)\in \cD_{\lam}(H)$, $m\ge2$ even and $k\in \N$. We have
\[
L_{A,B}(k)=m^n \sum_{i\in I}p_i(n)\cdot \alpha_i ^n\, ,
\]
where $I$ is an index set of size at most $e^{O(k\log k)}$
and for each $i\in I$,
$\alpha_i\in [0,\delta_{A,B}^k)$ depends only on $(H,\lam)$ and
$p_i$ is a polynomial of degree at most $2(k-1)$ whose coefficients depend only on $(H,\lam)$.
Moreover the
$\alpha_i$ and the coefficients of the polynomials $p_i$
can all be computed in time $e^{O(k \log k)}$.
 \end{lemma}
 
 \begin{proof}
For $v\in V$ we say $i\in [n]$ is an \emph{active coordinate} of $v$ if $v_i\neq0$. For $S\subset V$ we define the active coordinates of $S$ to be the union of the sets of active coordinates of the elements of $S$. For a cluster $\Gamma$, we define the active coordinates of $\Gamma$ to be the active coordinates of the set $V(\Gamma)=\bigcup_{S\in \Gamma}S$.

 For $j, k, a\in \N$, let $\cG_{j, k, a}$ denote the set of all clusters $\Gamma$ containing $\mathbf 0$ with $\|\Gamma\|=k$ and $|V(\Gamma)|=j$ and whose set of active coordinates is precisely $[a]$.
Note that if $S\subset V$ is a $G^2$-connected set of size $j$ containing $\mathbf 0$ then $S$ has at most $2(j-1)$ active coordinates.  Therefore, if $\cG_{j, k, a}\neq \emptyset$ we must have $a\leq 2(j-1)$.

By symmetry of coordinates and vertex transitivity of $\Z_m^n$ we have
\begin{align}\label{eqtrans}
L_{A,B}(k)=m^n\sum_{j=1}^k \frac{1}{j}\sum_{a=1}^{2(j-1)}\binom{n}{a}\sum_{\Gamma\in \cG_{j, k, a}}w_{A,B}(\Gamma)\, .
\end{align}

We first show how to efficiently generate the sets $\cG_{j,k,a}$,
then we show how to efficiently compute the weights $w_{A,B}(\Gamma)$.

\begin{claim}\label{claimcG}
For $k\in \N$, $j\in[k]$, and $a\in [2k]$ the set $\cG_{j, k, a}$ has size $e^{O(k\log k)}$ and can be
generated in time $e^{O(k\log k)}$.
\end{claim}
\begin{proof}[Proof of Claim~\ref{claimcG}]
First we construct the list $\cL_j$ of all 
$G^2$-connected subsets of $V$ of size $j$ 
which contain $\mathbf 0=(0,\ldots,0)$ and
whose set of active coordinates are a subset of $[2k]$. 
We do so iteratively. Let $1\leq t<j$ and suppose we have constructed the list $\cL_t$. For each $S\in \cL_t$ and $v\in S$, we run through the list of vertices $w$ at distance $\leq2$ from $v$ such that the active coordinates of $w$ are a subset of $[2k]$, and we add $S\cup \{w\}$ to the list $\cL_{t+1}$ if $w\notin S$. Note that there are at most $4\binom{2k}{2}+2k\leq 8k^2$ choices for $w$.  This procedure generates the whole list $\cL_{t+1}$ and shows that $|\cL_{t+1}|\le 8tk^2|\cL_t|$ and so $|\cL_j|\leq j!(8k^2)^j=e^{O(k \log k)}$.

Let $\cL^a_j$ denote the subset of $\cL_j$ consisting of those sets whose active coordinates are precisely $[a]$. Note that we can generate the list $\cL^a_j$ in time $e^{O(k \log k)}$ by checking the elements of $\cL_j$ one by one. 

We now generate the list $\cG_{j, k, a}$. To do so we run through each $S\in \cL^a_j$ and create the list of clusters $\Gamma$ with $\|\Gamma\|=k$ and $V(\Gamma)=S$. We claim that this can be done in time $e^{O(k \log k)}$. Recall that a cluster $\Gamma$ with $\|\Gamma\|=k$ is an ordered set of polymers $\Gamma=(\gamma_1,\ldots, \gamma_\ell)$ such that $\sum_{i=1}^\ell |\gamma_i|=k$. Let us fix $S\in \cL^a_j$. Since there are at most $2^k$ ordered integer partitions of $k$, it suffices to show that for a fixed such partition $(j_1,\ldots, j_\ell)$ (so that $\sum_i j_i=k$) we may find, in time $e^{O(k \log k)}$, all clusters $(\gamma_1,\ldots, \gamma_\ell)$ for which $|\gamma_i|=j_i$ for all $i$, and $\bigcup_i \gamma_i=S$. To do this we can simply check each element of $\binom{S}{j_1}\times\ldots \times \binom{S}{j_\ell}$ (a set of size at most $e^{O(k\log k)}$) to see if it constitutes a legitimate cluster. 
\end{proof}

We now turn to computing the weights $w_{A,B}(\Gamma)$ in~\eqref{eqtrans}. For $S\subset V$, let $\tilde \chi_{A,B}(S)$ denote the set of all colourings of $G[S]$
which disagree with $(A,B)$ at every vertex of $S$.

\begin{claim}\label{claimwcalc}
 If $S\subset V$ is a $G^2$-connected set
 of size at most $k$ whose active coordinates lie
 in $[2k]$, then we can write
\[
w_{A,B}(S)=\sum_{f\in\tilde\chi_{A,B}(S)}\beta_f \cdot \alpha_f^d
\]
where $0\leq \alpha_f<\delta_{A,B}^{|S|}$, $\beta_f\geq 0$ and both can be computed in time polynomial in $k$.
\end{claim}

Note that if $|S|\leq k$, 
then $|\tilde\chi_{A,B}(S)|=e^{O(k)}$
and we can list the elements of 
$\tilde\chi_{A,B}(S)$ by brute force in $e^{O(k)}$ time. 

We recall that for a cluster $\Gamma$, $w_{A,B}(\Gamma)\bydef \phi(I_\gamma)\prod_{\gamma\in \Gamma}w_{A,B}(\gamma)$ where $\phi(I_\gamma)$ is the Ursell function of the graph $I_\gamma$ (as defined in~\eqref{eqUrsell}). The Ursell function $\phi(I_\gamma)$ is an evaluation of the Tutte polynomial of $I_\gamma$ and therefore can be computed in time $e^{O(k)}$ by an algorithm of Bj\"orklund, Husfeldt, Kaski, and Koivisto \cite[Theorem 1]{bjorklund2008computing}.

The lemma therefore follows from \eqref{eqtrans} and Claims~\ref{claimcG}~and~\ref{claimwcalc}.
It remains to prove Claim~\ref{claimwcalc}.

\begin{proof}[Proof of Claim~\ref{claimwcalc}]
Recall from~\eqref{eqconv} that 
\begin{align}\label{eqwagain}
w_{A,B}(S)=
 \frac{\sum_{f\in\hat\chi_{A,B}(S)} \prod_{v\in S^+}\lam_{f(v)}}{\lam_A^{|S^+\cap\cO|} \lam_B^{|S^+\cap\cE|}}
\, ,
\end{align}
where $\hat\chi_{A,B}(S)$ is the set of all
colourings of $G[S^+]$ which disagree with $(A,B)$
precisely on $S$.

For $f\in \tilde \chi_{A,B}(S)$ and $v\in \p S\cap \cO$,
let $A(f,v)$ denote the subset of $A$ available to $v$ given the colouring $f$ of $G[S]$.
More precisely $A(f,v)=A\cap\bigcap_{u\in N_G(v)\cap S} N_H(f(u))$.
For $v\in \p S\cap \cE$, define $B(f,v)$ similarly.
Then 
\begin{align}\label{eqtildechi}
\sum_{f\in\hat\chi_{A,B}(S)} \prod_{v\in S^+}\lam_{f(v)}=
\sum_{f\in\tilde\chi_{A,B}(S)} \prod_{v\in S}\lam_{f(v)}\prod_{u\in \p S\cap\cO} \lam_{A(f,u)}\prod_{w\in \p S\cap\cE} \lam_{B(f,w)}\, .
\end{align}

Let $\bar S$ be the set of vertices in $V$ which are either contained in $S$ or have $\geq2$ neighbours in $S$. By grouping elements of $\partial S\bs \bar S$ according to their unique neighbour in $S$ we may rewrite the right hand side as 
\[
\sum_{f\in\tilde\chi_{A,B}(S)}
 \prod_{v\in S}\lam_{f(v)}\prod_{u\in \bar S\cap\p S\cap\cO} \lam_{A(f,u)}\prod_{w\in \bar S\cap \p S\cap\cE} \lam_{B(f,w)}
\prod_{x\in S\cap\cE}\lam_{A\cap N(f(x))}^{d-d_{\bar S}(x)}
\prod_{y\in S\cap\cO}\lam_{B\cap N(f(y))}^{d-d_{\bar S}(y)}\, .
\]
We now consider the denominator in~\eqref{eqwagain}.  
Note that
\begin{align}\label{eqSplus}
|S^+\cap\cO|=|\bar S\cap\cO|+\sum_{v\in S\cap\cE}(d-d_{\bar S}(v))\, ,
\end{align}
and similarly for $|S^+\cap\cE|$.
Thus 
\begin{align}
w_{A,B}(S)=\sum_{f\in\tilde\chi_{A,B}(S)}\Bigg[
&\frac{1}{\lam_A^{|\bar S\cap \cO|}\lam_B^{|\bar S\cap \cE|}}
 \prod_{v\in S}\lam_{f(v)}\prod_{u\in \bar S\cap\p S\cap\cO} \lam_{A(f,u)}\prod_{w\in \bar S\cap \p S\cap\cE} \lam_{B(f,w)} \\
&\prod_{x\in S\cap\cE}\left(\frac{\lam_{A\cap N(f(x))}}{\lam_A}\right)^{d-d_{\bar S}(x)}
\prod_{y\in S\cap\cO}\left(\frac{\lam_{B\cap N(f(y))}}{\lam_B}\right)^{d-d_{\bar S}(y)}\Bigg]\, . 
\end{align}

For each $f\in\tilde\chi_{A,B}(S)$ the corresponding term in the above sum
takes the form $\beta_f\cdot \alpha_f^d$ where $\alpha_f<\delta_{A,B}^{|S|}$ by~\eqref{eqdeltadef2},
the definition of $\delta_{A,B}$. We now show that the set $\bar S$ can be generated in time polynomial in $k$ and that $|\bar S|\leq k^2$ so that  $\alpha_f, \beta_f$ can be 
 computed in time polynomial in $k$. 
 
 Since $\Z_m^n$ has maximum codegree $2$, 
 by considering the joint neighbourhood of each pair of elements in $S$,
 we have $|\bar S|\leq k+2\binom{k}{2}=k^2$.
 Now let $u, v\in S$ be distinct. 
Since the active coordinates of $S$ are contained in $[2k]$,
each element of $N(u)\cap N(v)$ must differ from $u$ and $v$ in a 
coordinate in $[2k]$. 
We may therefore generate the set $N(u)\cap N(v)$,
and hence also $\bar S$, in time polynomial in $k$.
\end{proof} 
\end{proof}

\section{Counting proper $q$-colourings}\label{seclabelqcol}
In the previous section we gave a general description of 
the terms of the cluster expansion of the log partition functions
$\ln \Xi_{A,B}$ along with an algorithm for how to compute these terms.
In this section we specialise to the case of the 
uniform measure on $q$-colourings and present a
 more detailed picture of the cluster expansion. 
 In particular, we prove a strengthening of
Theorem~\ref{conjqcol} and prove Corollary~\ref{corcurious}, 
thus resolving \cite[Conjecture 5.2]{kahn2018number}.

Throughout this section
we let $(H,\lam)=(K_q, \dot\iota)$
where $\dot\iota\equiv 1$.
We let $V(K_q)=[q]$.
The dominant patterns 
of $(K_q, \dot\iota)$
are the pairs $(A,B)$
where $A\cup B= [q]$
and
$\{|A|,|B|\}=\{\lceil q/2 \rceil, \lfloor q/2 \rfloor\}$.

\subsection{Typical polymers}
 
 Define the \textit{type} of a $G^2$-connected set $S\subset V$ to be the isomorphism class of the induced subgraph $G^2[S]$.
For any fixed type $T$, we denote by $n_T$ the number of  $G^2$-connected sets of type $T$.
Let us call a set $S\subset V$ \emph{typical} if it is
$G^2$-connected,
its type is a tree, 
$S$ is a subset of one of the partition classes $\cE, \cO$ and 
for every edge $\{u,v\}$ in $G^2[S]$, the codegree of $u$ and $v$ in $G$ is $2$
(we note that this last condition is redundant in the case where $m=2$).
The reason for considering such sets is that almost all
$G^2$-connected sets of a fixed size are typical and the
weights of these sets take a particularly simple form (see Lemma~\ref{lemtreeweight}). 
For a tree $T$, let $n'_T$ denote the number of typical sets of type $T$.
\begin{lemma}\label{lemnT}
Let $t\in \N$ and let $T$ be a type with $t$ vertices. 
If $T$ is a tree then
\[
n'_T=(c_T+o(1))m^nn^{2t-2}\, \text{ and\,  }n_T-n'_T=O(m^nn^{2t-3})
\]
where $c_T=\left(\frac{1+3\ind_{m>2}}{2}\right)^{t-1}|\textup{Aut}(T)|^{-1}$. 
If $T$ is not a tree then
\[
n_T=O(m^nn^{2t-3})\, .
\]
\end{lemma}
\begin{proof}

Since $T$ is a connected graph we may fix an ordering
 $(x_1,\ldots, x_t)$ of the vertices of $T$ so that 
 $T_i\bydef T[\{x_1,\ldots, x_i\}]$ is connected for all $i\in[t]$.
 We let $d_i$ denote the degree of the vertex $x_i$ in the graph $T_i$.

Let $h_i$ denote the number of 
 injective graph homomorphisms  
from $T_i$ to $G^2$.
We note that $h_1=m^n$ and 
\begin{align}\label{eq:Aut}
n_T=h_t/|\text{Aut}(T)|\, .
\end{align}
We will construct an injective graph homomorphism 
 $\varphi: T\to G^2$
 recursively as follows.
Suppose that we have constructed an injective graph homomorphism
 $\varphi_i: T_i\to G^2$
 for some $i\leq t-1$.
 We will extend $\varphi_i$ to an injective graph homomorphism
  $\varphi_{i+1}: T_{i+1}\to G^2$.
 Let $E_i$ denote 
the set of possible choices for $\varphi_{i+1}(x_{i+1})$.
We consider two cases. 

 If $d_{i+1}>1$, then $\varphi_{i+1}(x_{i+1})$ 
  must lie in the joint neighbourhood of $\varphi_i(x)$ and $\varphi_i(y)$ for some $x,y\in V(T_i)$.
  For any pair of vertices $u,v\in V$ their codegree in $G^2$
  is at most $4n$ and so
  \begin{align}\label{mbd1}
   |E_i|\leq 4n\, .
  \end{align}

Suppose now that $d_{i+1}=1$.
 We note that $u\in E_i$ if and only if $u$ is adjacent to $\varphi_{i}(x_{i})$ and non-adjacent to $\varphi_{i}(x_{j})$ for $j<i$ in $G^2$.
 Again using the fact that the maximum codegree in $G^2$ is at most $4n$ we have
 \begin{align}\label{mbd2}
|E_i|=(1+3\ind_{m>2})\binom{n}{2}+O(n)\, .
\end{align}

 If $T$ is not a tree then $d_{i+1}>1$ for some $i\le t-1$. 
 It follows by ~\eqref{mbd1} and \eqref{mbd2}
 that $h_t=O(m^n n^{2(t-1)-1 })=  O(m^n n^{2t-3 })$. 
 The bound $n_T=O(m^n n^{2t-3 })$ follows from~\eqref{eq:Aut}.

Suppose now that $T$ is a tree so that 
$d_{i+1}=1$ for all $i\le t-1$.
Let $h'_i$ denote the number of 
 injective graph homomorphisms  
from $T_i$ to $G^2$ whose image is typical.
Let $E_i'\subset E_i$ be
the set of possible choices for $\varphi_{i+1}(x_{i+1})$
 so that $\varphi_{i+1}(x_i), \varphi_{i+1}(x_{i+1})$
 have codegree $2$ in $G$ 
 (in particular $x_i, x_{i+1}$ both lie in $\cE$ or $\cO$).
 Given $u\in E_i$, we have $u\notin E_i'$ if and only if
 either $u$ is adjacent to $\varphi_{i}(x_{i})$ in $G$
 or the codegree of $u$ and $\varphi_{i}(x_{i})$ is 1
 (the latter is only possible when $m>2$).
 It follows that 
\begin{align}\label{mbd3}
|E_i\bs E'_i|\leq 4n\, .
\end{align}
 By \eqref{mbd2} and \eqref{mbd3}, we have
 \[
h'_t=(1+o(1))(1+3\ind_{m>2})^{t-1}  2^{-(t-1)} m^n  n^{2(t-1)} 
\]
 and $h_t-h_t'= O(m^n n^{2t-3 })$.
 The result follows by~\eqref{eq:Aut} and the analogous identity $n'_T=h'_t/|\text{Aut}(T)|$.
\end{proof}

\begin{lemma}\label{lemtreeweight}
Fix $k\in \N$ and let $\gamma$ be a polymer of size $k$. 
We have
\[
w_{A,B}(\gam)= O\left(\left(1-\frac{1}{\lceil q/2 \rceil}\right)^{dk}\right)\, .
\]
Moreover if $\gamma$ is typical of type $T$,
 then if $\gamma\subset\cO$
\[
w_{A,B}(\gamma)=\frac{|B|^k}{|A|^k}\left(1-\frac{1}{|B|}\right)^{dk-k+1}
\]
and if $\gamma\subset\cE$
\[
w_{A,B}(\gamma)=\frac{|A|^k}{|B|^k}\left(1-\frac{1}{|A|}\right)^{dk-k+1}\, .
\]
\end{lemma}
\begin{proof}
We let $w$ denote $w_{A,B}$.
By \eqref{eqwbdsimple}, 
noting that in this setting $\lam_i=1$ for all $i$ and 
$\{\lam_A, \lam_B\}=\{\lceil q/2 \rceil, \lfloor q/2 \rfloor\}$,
\begin{align}
w(\gam)= O\left(\left(1-\frac{1}{\lceil q/2 \rceil}\right)^{|\p \gam|}\right)\, .
\end{align}
Since the size of $\gamma$ is $k$, a constant, 
we have by Lemma~\ref{thmviso} that  $|\p\gam|=dk-O(1)$
and so
\begin{align}
w(\gam)
=O\left(\left(1-\frac{1}{\lceil q/2 \rceil}\right)^{dk}\right)\, .
\end{align}

Suppose now that
$\gamma$ is typical and of type $T$ 
where $T$ is a tree on $k$ vertices.
We may abuse notation slightly
 and let $T$ denote the graph $G^2[\gamma]$
(as well as the isomorphism class of this graph).
Suppose that 
$\gam\subset \cO$
(the proof where $\gam\subset \cE$ is identical). 
Let $a=|A|$ and let $b=|B|$.
Since $\gamma$ is typical 
we have 
$|N(\gamma)|=dk-2(k-1)$.

By \eqref{eqwagain} 
\begin{align}\label{eqwtree}
w_{A,B}(\gamma)=\frac{|\hat \chi_{A,B}(\gamma)|}{a^{|\gam|} b^{|N(\gamma)|}}
=\frac{|\hat \chi_{A,B}(\gamma)|}{a^{k} b^{dk-2(k-1)}}\, 
\end{align}
where $\hat\chi_{A,B}(\gam)$ is the set of all
colourings of $G[\gamma^+]$ which disagree with $(A,B)$
precisely on $\gamma$. Since $\gamma\subset\cO$, $\hat\chi_{A,B}(\gam)$
is the set of proper
colourings $f: \gamma^+\to B$.
Fix a map $c: \gam \to B$ and suppose that
we want to extend it to an element of $\hat \chi_{A,B}(\gamma)$.
For every vertex $v\in N(\gamma)$ we have $b-1$ choices for its colour
except if $v\in N(u)\cap N(w)$
for some $u,w\in \gam$ where $c(u)\neq c(v)$. 
For such a vertex we have $b-2$ choices for its colour.
We therefore have
\begin{align}
|\hat \chi_{A,B}(\gamma)|&=
\sum_{c: \gam\to B}(b-1)^{dk-2(k-1)}\left(\frac{b-2}{b-1} \right)^{\sum_{\{u,v\}\in E(T)}\ind_{c(u)\neq c(v)}}\\
&= (b-1)^{dk-3(k-1)}(b-2)^{k-1} \sum_{c: \gam\to B}\left(\frac{b-1}{b-2} \right)^{\sum_{\{u,v\}\in E(T)}\ind_{c(u)= c(v)}}\, .
\end{align}
Though not essential for the proof, we note that this final sum
is an evaluation of the $b$-state Potts model partition function of the graph $T$.
It is well known that the Potts model partition function 
is the same for all trees $T$ on $k$ vertices
and indeed a simple induction on $k$ reveals that
\[
\sum_{c: \gam\to B}x^{\sum_{\{u,v\}\in E(T)}\ind_{c(u)= c(v)}}\equiv b(x+b-1)^{k-1}\, .
\]
We therefore have
\[
|\hat \chi_{A,B}(\gamma)|=b(b-1)^{dk-k+1}\, .
\]
The result now follows from \eqref{eqwtree}.
\end{proof}

\subsection{Terms of the cluster expansion}
In this section we show how Lemmas~\ref{lemnT} and~\ref{lemtreeweight} can be used to provide a detailed picture of the cluster expansion. 
Recall that we are in the specialised setting $(H,\lam)=(K_q, \dot\iota)$ where $\lam\equiv 1$,
so that the expression $L_{A,B}(k)$ (as defined in~\eqref{eqLkDef}) is the same for all dominant patterns $(A,B)$ by symmetry.  
We therefore denote $L_{A,B}(k)$ simply by $L_k$. 

Let $\cT_k$ denote the set of trees on $k$ vertices.
\begin{lemma}\label{lemLkbdnew}
For $k\in \N$ fixed, as $n\to\infty$
\[
L_k=(1+o(1))c_k\left(1-\frac{1}{\lceil q/2 \rceil}\right)^{dk}m^nn^{2k-2}
\]
where
\[
c_k= \frac{(1+\ind_{q \textup{ even}})(1+3\ind_{m>2})^{k-1}}{2^k(\lceil q/2 \rceil-1)^{k-1}}
\frac{\lceil q/2 \rceil^k}{\lfloor q/2 \rfloor^k}\sum_{T\in \cT_k} \frac{1}{|\textup{Aut}(T)|}\, .
\]
\end{lemma}
\begin{proof}
Let $\Gamma$ be a cluster of size $k$.
By Lemma~\ref{lemtreeweight}
\begin{align}\label{eqwgamcol}
w(\Gam)=
\phi(I_{\Gamma})\prod_{\gam\in\Gam}w(\gam)
=O\left(\left(1-\frac{1}{\lceil q/2 \rceil}\right)^{dk}\right)\, 
\end{align}
where $\phi(I_{\Gamma})$ is the Ursell function of
the incompatibility graph of $\Gamma$ 
as defined in~\eqref{eqUrsell}.

We say a cluster $\Gamma$ is of type $T$
if the set $V(\Gamma)\bydef \bigcup_{\gamma\in\Gamma}\gamma$ is of type $T$.
Similarly we say $\Gamma$ is typical if $V(\Gamma)$ is typical.

Note that given a set $X\subseteq V$ of size at most $k$, 
there are at most a constant (dependent on $k$) number of clusters 
$\Gamma$ of size $k$ such that $V(\Gamma)=X$.
The number of clusters of size $k$ and type $T$ is
therefore $O(n_T)$ (recall that $n_T$ is the number of  $G^2$-connected sets of type $T$).

For $T\in \cT_k$, let $\cC_k(T)$ denote the set of all typical clusters $\Gamma$ 
of type $T$ with $\|\Gamma\|=k$. 
 Then
by Lemma~\ref{lemnT} and \eqref{eqwgamcol}
\begin{align}\label{eqLktree}
L_k=\sum_{\Gamma\in \cC_k} w(\Gamma)=\sum_{T\in \cT_k}\sum_{\Gamma\in\cC_k(T)}w(\Gamma) + O\left(\left(1-\frac{1}{\lceil q/2 \rceil}\right)^{dk}m^n n^{2k-3}\right)\, .
\end{align}
Let us now fix $T\in \cT_k$.
Since $T$ is a tree, for $1\leq\ell\leq k$, deleting $\ell-1$ edges
from $T$ results in a graph with precisely $\ell$ connected components. 
Therefore the number of vertex partitions $V(T)=V_1\cup\ldots \cup V_\ell$
for which $T[V_i]$ is connected for each $i$ is $\binom{k-1}{\ell-1}$.
Thus if $X\subset V$ is of type $T$,
the number of clusters
$\Gamma\in \cC_k(T)$ consisting of 
$\ell$ polymers with $V(\Gamma)=X$ is $\ell!\binom{k-1}{\ell-1}$
(recall that a cluster is an ordered multiset of polymers).
Moreover for any such cluster $\Gamma$,
the incompatibility graph $I_{\Gamma}$ is a tree on $\ell$ vertices and 
so has Ursell function $\phi(I_{\Gamma})=(-1)^{\ell-1}/\ell!$.
If $X=V(\Gamma)$ is typical with $X\subset\cE$,
then by Lemma~\ref{lemtreeweight}, letting $a=|A|, b=|B|$,
\[
w(\Gamma)=\frac{1}{\ell !}(-1)^{\ell-1} \frac{a^k}{b^k}\left(1-\frac{1}{a}\right)^{dk-k+\ell}\, ;
\]
we obtain the same formula with $a, b$ swapped if $X\subset\cO$.
Recalling that $n'_T$ denotes the number of typical sets of type $T$, 
we then have
\begin{align}
\sum_{\Gamma\in\cC_k(T)}w(\Gamma)
&=
\frac{n'_T}{2} \sum_{\ell=1}^{k}\binom{k-1}{\ell-1}(-1)^{\ell-1}
\frac{a^k}{b^k}\left(1-\frac{1}{a}\right)^{dk-k+\ell}\\
&+\frac{n'_T}{2} \sum_{\ell=1}^{k}\binom{k-1}{\ell-1}(-1)^{\ell-1}
\frac{b^k}{a^k}\left(1-\frac{1}{b}\right)^{dk-k+\ell}\\
&=\frac{n'_T}{2}
\left[
\frac{a^k}{b^k(a-1)^{k-1}}
\left(1-\frac{1}{a}\right)^{dk}
+
\frac{b^k}{a^k(b-1)^{k-1}}
\left(1-\frac{1}{b}\right)^{dk}
\right]\\
&=
(1+o(1))(1+\ind_{q \text{ even}})
\frac{n'_T}{2} 
\frac{\lceil q/2 \rceil^k}{\lfloor q/2 \rfloor^k(\lceil q/2 \rceil-1)^{k-1}}
\left(1-\frac{1}{\lceil q/2 \rceil}\right)^{dk}\, .
\end{align}
The result follows from \eqref{eqLktree} and Lemma~\ref{lemnT}.
\end{proof}

The above lemma shows that each term $L_k$ is 
positive for large $n$.
Theorem~\ref{cormain} then 
shows that in order to compute 
an expression for $c_q(\Z_m^n)$ that is correct 
up to a multiplicative $(1+o(1))$ factor, 
one necessarily has to compute 
$L_1, \ldots, L_{k-1}$ where $k$
is the least integer for which $L_k=o(1)$.
By Lemma~\ref{lemLkbdnew}, this is the least integer $k$
such that $m(1-1/\lceil q/2 \rceil)^{(1+\ind_{m>2})k}<1$.

\subsection{The hypercube}

To end this section we specialise
to the case where $m=2$
and fully determine $L_1, L_2$
in order to arrive at the expressions
for $c_q(Q_n)$ stated in Corollary~\ref{corcurious}.
The calculations required to compute 
the expressions for $i(\Z_m^n)$ in Corollary~\ref{corcurious}
are similar and so we omit them. 

The following is a strengthening of Theorem~\ref{conjqcol}.
 
 \begin{theorem}
 \label{thmL1L2}
 For $q\ge4$
\begin{align}
c_q(Q_n)= (1+\ind_{\{q \text{ odd}\}})
\binom{q}{\left\lfloor q/2 \right\rfloor}
\left\lfloor \frac{q}{2} \right\rfloor^{2^{n-1}} \left\lceil \frac{q}{2} \right\rceil^{2^{n-1}} 
\cdot
\exp\left\{
L_1 + L_2 + \eps
\right\}\, ,
\end{align}
where
\begin{align}
L_1=
\frac{\lceil q/2 \rceil}{2 \lfloor q/2 \rfloor}
\left(2-\frac{2}{\lceil q/2 \rceil}\right)^n +
\frac{\lfloor q/2 \rfloor }{2 \lceil q/2 \rceil}
\left(2-\frac{2}{\lfloor q/2 \rfloor}\right)^n
\, ,
\end{align}
\begin{align}
L_2= 2^n\Bigg[&\frac{q-1}{2(\left\lfloor q/2 \right\rfloor-1)(\left\lceil q/2 \right\rceil-1)}\cdot n\left(1-\frac{1}{\left\lfloor q/2 \right\rfloor}\right)^n\left(1-\frac{1}{\left\lceil q/2 \right\rceil}\right)^n\\
&+\frac{\left\lfloor q/2 \right\rfloor^2}{8\left\lceil q/2 \right\rceil^2(\left\lfloor q/2 \right\rfloor-1)^3}(n^2-n-2(\left\lfloor q/2 \right\rfloor-1)^3)\left(1-\frac{1}{\left\lfloor q/2 \right\rfloor}\right)^{2n}\\
&+\frac{\left\lceil q/2 \right\rceil^2}{8\left\lfloor q/2 \right\rfloor^2(\left\lceil q/2 \right\rceil-1)^3}(n^2-n-2(\left\lceil q/2 \right\rceil-1)^3)\left(1-\frac{1}{\left\lceil q/2 \right\rceil}\right)^{2n}\Bigg]\, ,
\end{align}
and 
\[
\eps=O\left(\left(1-\frac{1}{\lceil q/2 \rceil}\right)^{3n}2^nn^{4}\right)\, .
\]
 \end{theorem} 
 
 \begin{proof}
The result will follow from Theorem~\ref{cormain}
once we verify the expressions for $L_1, L_2$.

Let $(A,B)$ be a dominant pattern.
For ease of notation we let 
$a=|A|= \lfloor q/2 \rfloor$ and 
$b=|B|= \lceil q/2 \rceil$.

\paragraph{Polymers.}

There are two types of polymer of size $1$: 
a single vertex in $\cO$ and a single vertex in $\cE$. 
There are $2^{n-1}$ of the first type and each has weight 
$\tfrac{b}{a}(1-\tfrac{1}{b})^n$.  
There are $2^{n-1}$ of the second type and each has weight
$\tfrac{a}{b}(1-\tfrac{1}{a})^n$.

There are three types of polymer of size $2$:
\begin{enumerate}[label=(\roman*)]
\item two vertices in $\cO$ at distance $2$ in $Q_n$,
\item two vertices in $\cE$ at distance $2$ and
\item an edge of $Q_n$.
\end{enumerate}
There are $2^{n-3} n (n-1)$ of type (i) 
and each has weight 
\[
\frac{b(b-1)}{a^2}\cdot \frac{(b-2)^2(b-1)^{2n-4}}{b^{2n-2}}+\frac{b}{a^2} \cdot \frac{(b-1)^{2n-2}}{b^{2n-2}}=\frac{1}{a^2}(b^2-3b+3)(1-\tfrac{1}{b})^{2n-3}\, .
\]
The first term in the sum on the LHS accounts for colourings where the two vertices of the polymer receive distinct colours in $B$ and the second term accounts for colourings where the two vertices receive the same colour.
Similarly, there are $2^{n-3} n (n-1)$ of type (ii) 
and each has weight 
$\tfrac{1}{b^2}(a^2-3a+3)(1-\tfrac{1}{a})^{2n-3}$.
There are $n2^{n-1}$ of type (iii) and
each has weight $(1-\tfrac{1}{a})^{n-1}(1-\tfrac{1}{b})^{n-1}$.

\paragraph{Clusters.}

There are two cluster types of size $1$, each consisting of single polymer of size $1$, with Ursell function $1$ and count and weight given above. 
Thus
\begin{align}
L_1=\frac{a}{2b}\left(2-\frac{2}{a}\right)^n+
\frac{b}{2a}\left(2-\frac{2}{b}\right)^n\, .
\end{align}
There are four types of clusters of size $2$: 
\begin{enumerate}
\item An ordered pair of incompatible polymers of size $1$ both from $\cO$, 
\item An ordered pair of incompatible polymers of size $1$ both from $\cE$, 
\item An ordered pair of incompatible polymers of size $1$,
 one from $\cO$ and one from $\cE$,
 \item One polymer of type (i), (ii) or (iii) above.
\end{enumerate}

There are $2^{n-1} + 2^{n-2} n (n-1)$ of type 1,
with Ursell function $-1/2$ and 
weight $\tfrac{b^2}{a^2}(1-\tfrac{1}{b})^{2n}$.
There are $2^{n-1} + 2^{n-2} n (n-1)$ of type 2,
with Ursell function $-1/2$ and 
weight $\tfrac{a^2}{b^2}(1-\tfrac{1}{a})^{2n}$.
There are $n 2^n$ of type 3,
with Ursell function $-1/2$ and 
weight $(1-\tfrac{1}{a})^{n}(1-\tfrac{1}{b})^{n}$.
The rest have Ursell function 1 with counts and weights given above. 

All together this gives:
\begin{align}
L_2= 2^n\Bigg[&\frac{a+b-1}{2(a-1)(b-1)}\cdot n\left(1-\frac{1}{a}\right)^n\left(1-\frac{1}{b}\right)^n\\
&+\frac{a^2}{8b^2(a-1)^3}(n^2-n-2(a-1)^3)\left(1-\frac{1}{a}\right)^{2n}\\
&+\frac{b^2}{8a^2(b-1)^3}(n^2-n-2(b-1)^3)\left(1-\frac{1}{b}\right)^{2n}\Bigg]\, .
\end{align}
 \end{proof}

The expressions for $c_q(Q_n)$ in Corollary~\ref{corcurious}
follow once we observe that $\eps_n=o(1)$ for $q\in\{5,6,7,8\}$.

\section{$k$-bounded functions and height functions}\label{seckbd}
In this section we apply our detailed understanding 
of the set $\hom(Q_n, H)$ (for appropriately chosen $H$)
in order to prove Theorem~\ref{thmkbd} and
then Theorem~\ref{thmBHM}.

\subsection{$k$-bounded functions}
Recall from the introduction that
a function $f: 2^{[n]}\to\N$ is \emph{$k$-bounded}
if $f(\emptyset)=0$ and
\[
0\leq f(A\cup\{x\})-f(A)\leq k
\]
for all $A\subset [n]$ and $x\in [n]\bs A$. 
We let $\cB_k(n)$ denote the set of all 
$k$-bounded functions on $2^{[n]}$. 
Recall also that 
 \[
\mathcal F(n)=
 \{f:V(Q_n)\to\mathbb Z: f(\mathbf 0)=0 \text{ and } u\sim v \implies |f(u)-f(v)|=1  \}\, .
 \]
 
As observed by Mossel (see \cite{kahn2001range}),
there is a bijection from $\cB_1(n)$ (rank functions)
to the set  $\mathcal F(n)$.
The next lemma generalises this observation, 
giving a bijection from $\cB_k(n)$ to an 
appropriate class of functions $f: V(Q_n)\to \Z$.

Given a subset $\cS\subset \N$, let 
\[
\lip(Q_n; \cS)\bydef  \{f:V(Q_n)\to\mathbb Z: f(\mathbf 0)=0 \text{ and } u\sim v \implies |f(u)-f(v)|\in \cS  \}\, .
\]
In particular $\mathcal F(n)=\lip(Q_n; \{1\})$\, .

\begin{lemma}\label{lembkbij}
There exists a bijection
\begin{align}
\varphi:\cB_k(n) \to \lip (Q_n; \cS_k)\, ,
\end{align}
where
\[
\cS_k=\{0,\ldots,k\}\cap(k+2\Z)\, .
\]
\end{lemma}
\begin{proof}
Throughout we
identify $2^{[n]}$ with $V(Q_n)=\{0,1\}^n$ in the usual way.
For $v\in V(Q_n)$, we let $|v|\bydef \sum_{i}v_i$.
Given $f\in \cB_k(n)$, 
let $\varphi(f)$ denote the map $V(Q_n)\to\Z$
given by 
\[
v\mapsto 2f(v)-k|v|\, .
\]
We claim that $\varphi(f)\in   \lip (Q_n; \cS_k) $.
Clearly $\varphi(f)(\mathbf 0)=0$ since $f(\mathbf 0)=0$. 
Now suppose $u\sim v$ in $Q_n$ and
without loss of generality let $|u|=|v|+1$. 
Then 
\[
|\varphi(f)(u)-\varphi(f)(v)|=|2f(u)-2f(v)-k|\in\cS_k
\]
since $f\in\cB_k(n)$ and so $0\leq f(u)-f(v)\leq k$. 

For $g\in  \lip (Q_n; \cS_k) $, let $\varphi'(g)$ denote the map $V(Q_n)\to\Q$
given by 
\[
v\mapsto (g(v)+k|v|)/2\, .
\]
It suffices to show that $\varphi'(g)\in \cB_k(n)$ since then
clearly $\varphi, \varphi'$ are inverse to each other. 
Note first that $\varphi'(g)(\mathbf 0)=0$ since $g(\mathbf 0)=0$.
If $k$ is even then $g(v)$ is even for all $v\in V(Q_n)$.
If $k$ is odd then $g(v)$ and $|v|$ have the same parity for all $v\in V(Q_n)$.
In either case the image of $\varphi'(g)$ is a subset of $\Z$.
Suppose now that $u\sim v$ in $Q_n$ where
$|u|=|v|+1$.
Then 
\[
\varphi'(g)(u)-\varphi'(g)(v)=\frac{g(u)-g(v)+k}{2}\in\{0,\ldots,k\}
\]
since $g\in  \lip (Q_n; \cS_k)$.
This completes the proof. 
\end{proof}

We now show that for any finite set $\cS\subset \N$,
there is a bijection between the sets $\lip(Q_n; \cS)$ and
$\hom(Q_n, H)$ for an appropriately chosen graph $H$.
To this end we define the following class of Cayley graphs. 
Given $N\in\mathbb N$ and $\cS\subset \Z_N$,
let $C(N; \cS)$ denote the graph on vertex set 
$\Z_N$ where 
$u\sim v$ if and only if $u-v= \pm x$ for some $x\in \cS$.
Note that $C(N; \cS)$ has loops if $0\in \cS$. 
Let 
\[
\hom_{\mathbf 0 }(Q_n,C(N; \cS))\bydef \{f\in \hom(G,C(N; \cS)): f(\mathbf 0)=0\}.
\]

Suppose now $f: V(Q_n)\to \Z$. 
Let $\text{Mod}_N(f)$ denote the map 
$V(Q_n)\to \{0,\ldots, N-1\}$ where
$\text{Mod}_N(f)(v)\equiv f(v) \pmod N$.
The following lemma
is an adaptation of \cite[Proposition 2.1]{feldheim2018rigidity}.

\begin{lemma}\label{lemmodlip}
Fix a finite set $\cS\subset\N\cup\{0\}$. 
For $N\geq 4\cdot \max \cS+1$
\[
\textup{Mod}_N: \lip(Q_n; \cS)\to \hom_{\mathbf 0 }(Q_n, C(N; \cS)) 
\]
is a bijection.
\end{lemma}
\begin{proof}
We first show that $\textup{Mod}_N$ is injective.
Indeed suppose that
$f,g\in \lip(Q_n; \cS)$ where $f\neq g$.
Since $f(\mathbf 0)=g(\mathbf 0)=0$, there must exist
$\{u, v\}\in E(Q_n)$ such that 
$f(u)= g(u)=:x$ whereas
$f(v)\neq g(v)$.  
Let $s\bydef \max \cS$.
Since $f,g\in \lip(Q_n; \cS)$, 
we have $f(v), g(v)\in [x-s, x+s]$.
Since $N>2s$, it follows that 
$\textup{Mod}_N(f)(v)\neq \textup{Mod}_N(g)(v)$.

It remains establish surjectivity.  
For $v\in V(Q_n)=\{0,1\}^n$, we let $|v|\bydef \sum_{i}v_i$.
We define a spanning tree $T$ of $Q_n$,
rooted at $\mathbf 0$, as follows.
If $|v|>0$,
with $j$ the first coordinate for which $v_j\neq0$, 
let $v^\ast\bydef (0,\ldots, 0, v_j-1, v_{j+1},\ldots, v_n)$
be the parent of $v$ in $T$.

Suppose now that $g\in \hom_{\mathbf 0}(Q_n, C(N; \cS))$.
We will construct $f\in \lip(Q_n; \cS)$
such that $g=\textup{Mod}_N(f)$ recursively as follows. 
Set $f(\mathbf 0)=0$. 
Suppose now that $|v|>0$ and that
we have defined $f(w)$ for all $w$ such that $|w|<|v|$.
We then let $f(v)$ be the unique integer $z$ such that
$|f(v^\ast)-z|\in \cS$ and $z\equiv g(v) \pmod N$.
It remains to check that $f$ is indeed an element of 
$\lip(Q_n; \cS)$. We need to check that 
\begin{align}\label{eqindhyp}
|f(u)-f(v)|\in \cS \text{ for all } \{u,v\}\in E(Q_n)\, . \tag{\(\dagger\)}
\end{align}
Note that by construction \eqref{eqindhyp} holds for all pairs $\{v,v^\ast\}$.
We proceed by induction on $|u|+|v|$.
If $|u|+|v|=1$ then $u=v^\ast$ or vice versa and so \eqref{eqindhyp} holds. 
Suppose now that $\{u,v\}\in E(Q_n)$
with $|u|+|v|>1$.

If $\min\{i: u_i\neq0\}>\min\{i: v_i\neq0\}$ 
then $u=v^\ast$ and so \eqref{eqindhyp} holds.
If
$\min\{i: u_i\neq0\}=\min\{i: v_i\neq0\}$,
then $u\sim u^\ast\sim v^\ast\sim v$.
By the induction hypothesis and the construction of $f$,
$|f(u)-f(u^\ast)|, |f(u^\ast)-f(v^\ast)|, |f(v)-f(v^\ast)| \in \cS$ and so 
$|f(u)-f(v)|\leq 3s$.
Moreover
$f(u)-f(v)\equiv g(u)-g(v)\equiv \pm x \pmod N$
for some $x\in \cS$. 
We deduce that $|f(u)-f(v)|\in \cS$ since $N>4s$.

\end{proof}

\begin{proof}[Proof of Theorem~\ref{thmkbd}]
By Lemmas~\ref{lembkbij} and \ref{lemmodlip},
\begin{align}\label{eqbkhom}
|\cB_k(n)|=|\hom_{\mathbf 0}(Q_n, C(N; \cS_k))|
\end{align}
where $\cS_k= \{0,\ldots,k\}\cap(k+2\Z)$ and $N\bydef 4k+1$.
By vertex transitivity of $C(N; \cS_k)$ we also have
\begin{align}\label{eqhom0hom}
|\cB_k(n)|=|\hom_{\mathbf 0}(Q_n, C(N; \cS_k))|=\frac{1}{N}|\hom(Q_n, C(N; \cS_k))|\, .
\end{align}
We now use Theorem~\ref{cormain} to obtain accurate
asymptotics for $|\hom(Q_n, C(N; \cS_k))|$
(and therefore also $|\cB_k(n)|$).
Let $H=C(N; \cS_k)$.
We begin by identifying the dominant patterns of $H$.
First let us establish some terminology. 
We call a set $\{y_1,\ldots, y_t\}\subset V(H)=\Z_N$,
an \emph{interval} if the elements of the set can be relabelled such that
$y_{i+1}= y_i+1$ for all $1\leq i\leq t-1$. 
Similarly, we call the set $\{y_1,\ldots, y_t\}$,
a \emph{skip interval} if the elements of the set can be renamed such that
$y_{i+1}= y_i+2$ for all $1\leq i\leq t-1$. 

Suppose now that $(A,B)$ is a pattern in $H$ and write 
$B=\{x_1,\ldots, x_\ell\}\subset \Z_N$.
Since $a\sim b$ in $H$ for all $a\in A, b\in B$,
we have 
\begin{align}
A\subset (x_1\pm \cS_k)\cap\ldots \cap(x_\ell\pm \cS_k)\subset \Z_N \, ,
\end{align}
where $x_i\pm \cS_k\bydef x_i+(-\cS_k\cup \cS_k)$. 
Now, if $\ell\leq k+2$,
\[
|(x_1\pm \cS_k)\cap\ldots \cap(x_\ell\pm \cS_k)|\leq |-\cS_k\cup \cS_k|-(\ell-1)=k+2-\ell
\]
with equality if and only if $B$
is a skip interval. 
Thus,
\[
|A||B|\leq \ell (k+2-\ell)\leq \left \lfloor \frac{k}{2}+1 \right\rfloor \left\lceil \frac{k}{2}+1 \right\rceil
\]
with equality if and only if $A,B$ are both skip intervals 
with
 $\{|A|, |B|\}=\left\{\left \lfloor \frac{k}{2}+1 \right\rfloor ,\left\lceil \frac{k}{2}+1 \right\rceil\right\}$.
 If $k$ is even and $B=\{x_1, \ldots, x_{k/2+1}\}$, then 
 \[
 A=(x_1\pm \cS_k)\cap\ldots \cap(x_{k/2+1}\pm \cS_k)=B\, .
 \]
 If $k$ is odd and $B=\{x_1, \ldots, x_{(k+3)/2}\}$
 where $x_{i+1}\equiv x_i+2 $ for $1\leq i \leq (k+1)/2$, then 
 \[
 A=(x_1\pm \cS_k)\cap\ldots \cap(x_{(k+3)/2}\pm \cS_k)=\{x_1+1, x_2+1,\ldots, x_{(k+1)/2}+1\}\, .
 \]
 
If $k$ is even, 
 we therefore have $N$ dominant patterns 
 (one for each choice of skip interval with $k/2+1$ elements).
 If $k$ is odd we have $2N$ dominant patterns:
 we choose a skip interval $B$ of length $(k+3)/2$,
 then we have dominant patterns $(A,B)$ and $(B, A)$. 
 
Given a dominant pattern $(A,B)$, 
let us calculate $L_{A,B}(1)$ (as defined in~\eqref{eqLkDef}).
Again we split into cases depending on the parity of $k$.
If $k$ is even,
by symmetry 
we may assume that $A=B=\{0,2,\ldots, k\}$.
Calculating $L_{A,B}(1)$ (e.g.\ by using the explicit expression \eqref{eqrefLAB1}) we have
\[
L_{A,B}(1)=2^{n-1}\left[\frac{2}{(k/2+1)^{n+1}}\sum_{v\in A^c}|N(v)\cap B|^n \right]\, .
\]

If $v=k+2t$, where $1\leq t \leq k/2$ then 
$N(v)\cap B= \{2t, 2t+2, \ldots, k\}$ and so
$|N(v)\cap B|=k/2-t+1$. 
Similarly if $v=-2t$ then $|N(v)\cap B|=k/2-t+1$.
For all other $v\in A^c$, $N(v)\cap B=\emptyset$. 
We therefore have 
\[
L_{A,B}(1)=\frac{2^{n+2}}{(k+2)}\sum_{t=1}^{k/2}\left(\frac{2t}{k+2}\right)^n=
\frac{2^{n+2}}{(k+2)}\left(\frac{k}{k+2}\right)^n+O\left(2^n\left(\frac{k-2}{k+2}\right)^n\right)
\, 
\]
and $\delta_{A,B}=k/(k+2)$
(with $\delta_{A,B}$ as defined in \eqref{eqdeltadef}).
By Lemma~\ref{lemexpL1} we therefore have
\[
\sum_{j=2}^{\infty}L_{A,B}(j)=O\left(n^2 2^n\left(\frac{k}{k+2}\right)^{2n}\right)\, .
\]
By Theorem~\ref{cormain2} we then have
 \[
 |\cB_k(n)|= \left(\frac{k}{2}+1\right)^{2^{n}}\exp\left\{ \frac{2^{n+2}}{k+2}\left(\frac{k}{k+2}\right)^n+O\left(n^2 2^n\left(\frac{k}{k+2}\right)^{2n}\right)\right\}
 \]
 (note we have $N$ dominant patterns which cancels the factor of $1/N$ in~\eqref{eqhom0hom}).
 
If $k$ is odd, by symmetry 
we may assume that $A=\{1,3,\ldots, k\}, B=\{0,2,\ldots, k+1\}$.
Then 
\[
L_{A,B}(1)=2^{n-1}\left[\frac{1}{|A||B|^n}\sum_{v\in A^c}|N(v)\cap B|^n + \frac{1}{|B||A|^n}\sum_{v\in B^c}|N(v)\cap A|^n \right]\, .
\]

If $v=k+2t$, where $1\leq t \leq (k+1)/2$ then 
$N(v)\cap B= \{2t, 2t+2, \ldots, k+1\}$ and so
$|N(v)\cap B|=(k+1)/2-t+1$. 
Similarly if $v=1-2t$ then $|N(v)\cap B|=(k+1)/2-t+1$.
For all other $v\in A^c$, $N(v)\cap B=\emptyset$. 

If $v=k+2t+1$, where $1\leq t \leq (k-1)/2$ then 
$N(v)\cap A= \{2t+1, 2t+3, \ldots, k\}$ and so
$|N(v)\cap A|=(k-1)/2-t+1$. 
Similarly if $v=-2t$ then $|N(v)\cap A|=(k-1)/2-t+1$.
For all other $v\in B^c$, $N(v)\cap A=\emptyset$.
Thus $\delta_{A,B}=(k+1)/(k+3)$ and
\begin{align}
L_{A,B}(1)&=\frac{2^{n+1}}{k+1}\sum_{t=1}^{(k+1)/2}\left(\frac{2t}{k+3}\right)^n
+\frac{2^{n+1}}{k+3}\sum_{t=1}^{(k-1)/2}\left(\frac{2t}{k+1}\right)^n\\
&= \frac{2^{n+1}}{k+1}\left(\frac{k+1}{k+3}\right)^n+
O\left(2^n\left(\frac{k-1}{k+1}\right)^{n}\right)
\, .
\end{align}
By Lemma~\ref{lemexpL1}
\[
\sum_{j=2}^{\infty}L_{A,B}(j)=O\left(n^2 2^n\left(\frac{k+1}{k+3}\right)^{2n}\right)\, 
\]
and so if $k\geq3$, by Theorem~\ref{cormain2},
 \[
 |\cB_k(n)|=2\left(\frac{k+1}{2}\right)^{2^{n-1}}\left(\frac{k+3}{2}\right)^{2^{n-1}}
 \exp\left\{\frac{2^{n+1}}{k+1}\left(\frac{k+1}{k+3}\right)^n
 +O\left(2^n\left(\frac{k-1}{k+1}\right)^{n}\right)\right\}\, .
 \]
 We have a leading factor of $2$ in the above since there are $2N$ dominant patterns and a factor of $1/N$ in~\eqref{eqhom0hom}.
 
If $k=1$ then the error term is dominated by  $L_{A,B}(2)$ and so
\[
 |\cB_1(n)|=\left(1+O\left(\frac{n^2}{ 2^{n}}\right)\right)2e2^{2^{n-1}}\, .
\]

For $k\geq 2$, we may summarise our results as follows:
\[
 |\cB_k(n)|=(1+\ind_{k\text{ odd}}) \left(\left \lfloor \frac{k}{2}+1 \right\rfloor \left\lceil \frac{k}{2}+1 \right\rceil\right)^{2^{n-1}}
  \exp\left\{\frac{(1+\ind_{k\text{ even}})}{\lfloor k/2+1\rfloor}\left(\frac{2 \lceil k/2 \rceil}{\lceil k/2+1\rceil}\right)^n
+\eps\right\}
\]
where
\[
\eps= \begin{cases} 
   O\left(n^2 2^n\left(\frac{k}{k+2}\right)^{2n}\right) & k\text{ even}\\
      O\left(2^n\left(\frac{k-1}{k+1}\right)^{n}\right) & k\text{ odd}\, .
   \end{cases}
\]
\end{proof}

\subsection{Height functions}
As a further application of our large deviation result 
(Theorem~\ref{lempolyLD}),
we now show how it can be used to prove Theorem~\ref{thmBHM}.
Recall that our goal is to estimate
the probability that a uniformly chosen height function $f\in \mathcal F(n)$
takes $>tn$ values (for $t$ fixed and $n$ large). 
We let $R(f)$ denote the size of the range of $f$.

The study of height functions (on $Q_n$ and more general graphs)
and their concentration properties
was initiated by 
 Benjamini, H{\"a}ggstr{\"o}m and Mossel \cite{benjamini2000random}.
 In~\cite{benjamini2000random} the authors conjecture that for any $t>0$,
if $f$ is chosen uniformly at random from $\mathcal{F}(n)$,
 then
 \(
\P(R(f)>tn)
 \)
converges to $0$ as $n\to\infty$.
Kahn~\cite{kahn2001range} resolved this conjecture in 
a strong form showing that there is
in fact a constant $b$ such that
  \(
 \P(R(f)>b)= e^{-\Omega(n)}\, .
 \)
Later, Galvin \cite{galvin2003homomorphisms} strengthened Kahn's 
 result further still showing that one can in fact take $b=5$.
Peled~\cite{peled2017high} subsequently proved a vast generalisation of 
Galvin's result which applies to a general class of tori including
$\Z_m^n$ (where $m$ is allowed to be large with respect to $n$)
and provides strong bounds on
\(
 \P(R(f)\geq k)
 \)
for arbitrary $k$. 
Since the upper bound on $\P(R(f)\geq tn)$ in Theorem~\ref{thmBHM}
is a special case of Peled's result \cite[Theorem 2.1]{peled2017high},
we sketch this part of the proof.

 First we note that by Lemma~\ref{lemmodlip}, 
 the map 
 $\text{Mod}_5$ is a bijection $\mathcal F(n) \to \hom_{\mathbf 0 }(Q_n, C_5)$
 where $C_5$ denotes a cycle of length $5$.
 A closer look at the proof of Lemma~\ref{lemmodlip} reveals that
 $\text{Mod}_3$ provides a bijection $\mathcal F(n)\to \hom_{\mathbf 0 }(Q_n, K_3)$
 (this bijection is attributed to Randall in \cite{galvin2003homomorphisms}).
In the following it is slightly more convenient to work with this latter bijection.

\begin{proof}[Sketch proof of Theorem~\ref{thmBHM}]
Let $f\in\mathcal F(n)$.
Note that since $f(u), f(v)$ have opposite parity for all $\{u,v\}\in E(Q_n)$
and $f(\mathbf 0)=0$ we have $f(\cE)\subset 2\Z$ and $f(\cO)\subset 1+2\Z$.

Let $t>0$ and 
suppose that $R(f)\geq tn$.
As remarked above, $\text{Mod}_3(f)\in \hom_{\mathbf 0}(Q_n, K_3)$.
Let $(A,B)$ be the dominant pattern of $(K_3, \dot\iota)$ (where $\dot\iota\equiv1$)
that agrees the most with  $\text{Mod}_3(f)$ 
(breaking ties arbitrarily if necessary).
We say an integer $x$ is \emph{good}
if either $x$ is even and $x\pmod 3\in A$
or $x$ is odd and $x\pmod 3\in B$.
We say an integer is \emph{bad} otherwise and
note that 
if $f(v)$ is a bad integer,
then $\text{Mod}_3(f)(v)$ 
disagrees with $(A,B)$ at $v$.

Let $p$ be the midpoint of the range of $f$.
Since good and bad integers occur in alternating intervals of length $3$,
we may choose $\{b,b+1\}\subset[p-3,p+3]$ such that $b$ and 
$b+1$ are both bad. 
Let 
\[
V_{\text{small}}\bydef \{v\in V(Q_n): f(v)\leq b\} \text{ and } V_{\text{large}}\bydef \{v\in V(Q_n): f(v)\geq b+1\}
\]
and note that $V_{\text{small}}, V_{\text{large}}$ partition $V(G)$.
Choosing $v,w\in V(Q_n)$ so that $f(v)-f(w)$ is maximal, we have 
$B_v(tn/2-4)\subset V_{\text{large}}$
and
$B_w(tn/2-4)\subset V_{\text{small}}$.
Suppose that 
$|V_{\text{small}}|\leq |V_{\text{large}}|$
(we argue similarly if $|V_{\text{small}}|> |V_{\text{large}}|$)
so that
\[
|B_{w}(tn/2-4)|\leq|V_{\text{small}}|\leq 2^{n-1}\, .
\]
By Harper's isoperimetric inequality~\cite[Theorem 1]{harper1966optimal},
\[
|\partial V_{\text{small}}|\geq |\p B_{w}(tn/2-4)|= \binom{n}{tn/2-3}\, .
\] 
Since $f(u)=b+1$ for all $u\in \p V_{\text{small}}$,
and $b+1$ is bad
we have that 
$\text{Mod}_3(f)\in\hom_{\mathbf 0}(Q_n, K_3)$
disagrees with $(A,B)$
at
each vertex of $\partial V_{\text{small}}$.

By Lemma~\ref{lembalanced}~\ref{itemdefectlarge} and Theorem~\ref{mainTV},
the probability that a uniformly chosen element of $\hom(Q_n, K_3)$
disagrees with its closest dominant pattern
at $\geq \binom{n}{tn/2-3}$ vertices is at most
\[
\leq \exp\left\{-\Omega\left(\frac{1}{n}\binom{n}{tn/2-3}\right) \right\}\, .
\]
The same is therefore also true of $\hom_\mathbf{0}(Q_n, K_3)$.
Since $\hom_{\mathbf{0}}(Q_n, K_3)$ and $\mathcal F(n)$ are in bijection,
it follows that if $f$ is a uniformly chosen element of $\mathcal F(n)$,
then
\[
\P(R(f)\geq tn)\leq \exp \left\{- \Omega \left(\frac{1}{n}\binom{n}{tn/2} \right) \right\}=\exp \left\{- 2^{H(t/2)n(1+o(1))}  \right\}\, .
\]

We now give a lower bound for $\P(R(f)\geq tn)$.
We construct a set of height functions $f\in \hom(Q_n, \Z)$ with
$R(f)\geq tn$ as follows. 
Let $\tau$ denote the least integer $\geq tn-1$
with the same parity as $n$ such that $\lfloor \tau/2 \rfloor$ is even.
Let $\mathbf 1=(1,\ldots, 1)\in V(Q_n)$.
Let $\cF$ denote the family of
$f: V(Q_n)\to \Z$ satisfying the following constraints:
\begin{enumerate}
\item $f(v)= d(v, \mathbf 0)$ for all $v\in B_\mathbf 0( \lfloor \tau/2 \rfloor )$,

\item $f(v)= \tau-d(v, \mathbf 1)$ { for all } $v\in B_\mathbf 1( \lceil \tau/2 \rceil )$,

\item $f(v)= \lfloor \tau/2 \rfloor$ { for all } $v\in \cE\bs (B_\mathbf 0(\lfloor \tau/2 \rfloor) \cup B_\mathbf 1(\lceil \tau/2 \rceil))$,

\item $f(v)\in \{\lfloor \tau/2 \rfloor-1, \lfloor \tau/2 \rfloor+1\}$ { for all } $v\in \cO\bs (B_\mathbf 0(\lfloor \tau/2 \rfloor) \cup B_\mathbf 1(\lceil \tau/2 \rceil))$.
\end{enumerate}

It is clear that $\cF \subset \hom(Q_n, \Z)$ and 
$R(f)=\tau+1\geq tn$ for all $f\in \cF$.
Moreover,
\[
|\cF|\geq \exp_2\left\{2^{n-1}-2|B_\mathbf 0( \lceil \tau/2 \rceil )|\right\}\geq \exp_2\left\{2^{n-1}-O\left(2^{H(t/2)n}\right)\right\}\, .
\]
Since $|\hom(Q_n, \Z)|\sim2e2^{2^{n-1}}$ by Theorem~\ref{cormain} and Lemma~\ref{lemmodlip} (or alternatively \cite[Corollary 1.5]{galvin2003homomorphisms}),
we have
\[
\P(R(f)\geq tn)\geq \exp_2\left\{-O\left(2^{H(t/2)n}\right) \right\}\, .
\]

\end{proof}

\section{Concluding Remarks}\label{secconc}
We conclude by mentioning some open 
problems and directions for future research. \\

\noindent \textbf{$n$ fixed, $m$ large.}
\begin{itemize}
\item The techniques of this paper were inspired by tools used to 
analyse spin models on the integer lattice $\Z^n$.
In this context, it is common to study the torus $\Z_m^n$
(viewed as the integer lattice with `periodic boundary conditions')
where $n$ is fixed and $m\to\infty$ .

The main obstacle in adapting the proofs of our paper to this setting
is that the isoperimetric properties of the torus $\Z_m^n$ become too weak
in the limit $m\to \infty$. 
One way to overcome this obstacle, inspired by~\cite{helmuth2018contours}, would be to replace our polymer models with \emph{contour models} from Pirogov-Sinai theory (see also \cite{friedli2017statistical} for an excellent introduction to this topic). 
Roughly speaking, our polymer models encode the regions on which a colouring disagrees with
a fixed dominant pattern whereas contour models can be used to encode the \emph{boundary} between regions agreeing with distinct patterns. The advantage is that contour models encode configurations more efficiently, using fewer vertices to describe deviations from a dominant pattern. 
This type of approach was recently applied with great success by Peled and Spinka~\cite{peled2017condition, peled2018rigidity, peled2020long}, to show that a general class of spin systems exhibit long-range order on $\Z^n$ (see the end of this section). 

Despite the obstacles posed by isoperimetry, we believe that the intuition provided by our polymer models are still a useful guide. 
In particular, in the regime where $n$ is a fixed large constant and $m$ is large with respect to $n$, we expect a typical configuration from $\mu_{H,\lam}$ to globally agree with a single dominant pattern with regions of disagreeing vertices appearing in $G^2$-connected components of size at most $C\log m$ (where $C$ may depend on $n$ and $(H,\lam)$). 

As a starting point, we offer the following conjecture, inspired by 
Theorem~\ref{cormain2}. As usual we let $G$ denote $\Z_m^n$.

\begin{conj}
For a fixed weighted graph $(H,\lam)$
\begin{align}\label{conjfreeenergy}
\lim_{m\to\infty}\frac{1}{m^n}\ln Z_G^H(\lam)=\frac{1}{2}\ln \eta_{\lam}(H)+ (1+o_n(1))\eps(H, \lam, n) \, ,
\end{align}
where the limit is taken over even $m$ and
\[
\eps(H, \lam, n):=\max_{(A,B)\in \cD_{\lam}(H)} \left(\frac{1}{2\lam_A\lam_B^{2n}}\sum_{v\in A^c}\lam_v \lam_{N(v)\cap B}^{2n} + \frac{1}{2\lam_B\lam_A^{2n}}\sum_{v\in B^c}\lam_v \lam_{N(v)\cap A}^{2n}\right)\, .
\]
\end{conj}
We note that the expression in parentheses in the above conjecture is $L_{A,B}(1)/m^n$. The RHS of~\eqref{conjfreeenergy} is easily seen to be a lower bound. 
Specialising to the case of approximating $c_q(G)$, the number of
 proper $q$-colourings of $G$,
the above conjecture predicts that for $q$ even
\begin{align}\label{eqqfree}
\lim_{m\to\infty}\frac{1}{m^n}\ln c_q(\Z_m^n)=\ln(q/2)+ (1+o_n(1))(1-2/q)^{2n}\, ,
\end{align}
where again we take the limit over even $m$.
We obtain an analogous prediction for $q$ odd. 
The current best upper bound is due to Peled and Spinka~\cite{peled2018rigidity} whose results imply  
that there exists $c>0$, such that for $q$ fixed and $n$ large, the LHS of~\eqref{eqqfree} is at most
 \(
 \tfrac{1}{2}\ln\left( \left\lfloor \tfrac{q}{2} \right\rfloor \left\lceil \tfrac{q}{2} \right\rceil\right) + e^{-\frac{cn}{q^3(q+\log n)}}\, .
\)

\item In the regime where the dimension $n$ is fixed and $m$ is large,
 various analogues of our torpid mixing result 
Theorem~\ref{thmslowmix} have
been proved throughout the literature 
for specific choices of $(H,\lam)$:
Galvin~\cite{galvin2008sampling} established such a result 
in the case of the hard-core model; 
Galvin and Randall \cite{galvin2007torpid} for the $3$-colouring model; 
and Borgs, Chayes, Dyer and Tetali~\cite{chayes2004sampling} for $(H,\lam)$
with carefully chosen $\lam$.
It would be interesting to study to what extent the 
intuition of Theorem~\ref{thmslowmix}, 
that non-trivial weighted 
graphs give rise to torpid mixing,
holds true in this regime.
\end{itemize}

\noindent \textbf{Odd sidelength.} 
Throughout this paper we assumed that the sidelength $m$ of the torus is even.
It is natural to wonder what happens in the case where $m$ is odd. 
One starting point would be to investigate the asymptotics of 
$Z_G^H(\lam)$ for fixed odd $m$, and $n\to\infty$. To our knowledge, even the asymptotics of $i(\Z_3^n)$ are unknown.\\

 \noindent\textbf{General spin systems.}
The spin systems considered in this paper fit into a more general framework
where one allows for soft interactions between spins. 
For a set of spins $[q]$ and functions $\lam: [q]\to\R_{\geq 0}$ and
$\delta:[q]\times[q]\to \R_{\geq 0}$ where $\delta_{i,j}=\delta_{j,i}$ for all $i,j\in [q]$,
one can define the probability of a configuration
$f: V(G) \to [q]$ to be proportional to 
\[
\prod_{v\in V(G)}\lam_{f(v)}\prod_{\{u,v\}\in E(G)}\delta_{f(u), f(v)}\, .
\]
This well-studied class of probability distributions notably includes the Ising model and Potts model.

The results of Peled and Spinka~\cite{peled2017condition, peled2020long} mentioned above, show that a general class of such spin systems exhibit long-range order in $\Z^n$.
Soon after the first appearance of our paper, Peled and Spinka posted~\cite{peled2020long} (a companion to~\cite{peled2017condition}) detailing this remarkable result. 
It would be very interesting to explore the extent to which the results of
our paper enjoy generalisations in this setting.
\bibliography{ClusterBib}
\bibliographystyle{abbrv}

\end{document}